\documentclass{amsart}
\usepackage[foot]{amsaddr}
\usepackage{amssymb}
\usepackage{amsfonts}
\usepackage{amsmath}
\usepackage{hyperref}
\usepackage{geometry}
\usepackage{xcolor}

\newtheorem{theorem}{Theorem}[section]
\theoremstyle{plain}
\newtheorem{acknowledgement}{Acknowledgement}

\newtheorem{claim}{Claim}

\newtheorem{corollary}[theorem]{Corollary}

\newtheorem{definition}[theorem]{Definition}

\newtheorem{lemma}[theorem]{Lemma}

\newtheorem{problem}{Problem}

\newtheorem{remark}[theorem]{Remark}

\numberwithin{equation}{section}
\usepackage{enumerate}
\input{tcilatex}
\begin{document}
	\title[Sums of squares I]{Sums of squares I: scalar functions}
	\author{Lyudmila Korobenko}
	\author{Eric Sawyer}
	\maketitle

\begin{abstract}
This is the first in a series of three papers dealing with sums of squares
and hypoellipticity in the infinitely degenerate regime. A result of C.
Fefferman and D. H. Phong shows that every $C^{3,1}$ nonnegative function on 
$\mathbb{R}^{n}$ can be written as a finite sum of squares of $C^{1,1}$
functions, and was used by them to improve G\aa rding's inequality, and
subsequently by P. Guan to prove regularity for certain degenerate operators.

In this paper we investigate sharp criteria sufficient for writing a smooth
nonnegative function $f$ on $\mathbb{R}^{n}$ as a finite sum of squares of $%
C^{2,\delta }$ functions for some $\delta >0$, and we denote this property
by saying $f$ is $SOS_{\func{regular}}$. The emphasis on $C^{2,\delta }$, as
opposed to $C^{1,1}$, arises because of applications to hypoellipticity for
smooth infinitely degenerate operators in the spirit of M. Christ, which are
pursued in the third paper of this series.

Thus we consider the case where $f$ is smooth and flat at the origin, and
positive away from the origin. Our sufficient condition for such an $f$ to
be $SOS_{\func{regular}}$ is that $f$ is $\omega $-monotone for some modulus
of continuity $\omega _{s}\left( t\right) =t^{s}$, $0<s\leq 1$, where $%
\omega $-monotone means 
\begin{equation*}
f\left( y\right) \leq C\omega \left( f\left( x\right) \right) ,\ \ \ \ \
y\in B_{x}\ ,
\end{equation*}%
and where $B_{x}=B\left( \frac{x}{2},\frac{\left\vert x\right\vert }{2}%
\right) $ is the ball having a diameter with endpoints $0$ and $x$ (this is
the interval $\left( 0,x\right) $ in dimension $n=1$). On the other hand, we
show that if $\omega $ is any modulus of continuity with $\lim_{t\rightarrow
0}\frac{\omega \left( t\right) }{\omega _{s}\left( t\right) }=\infty $ for
all $s>0$, then there exists a smooth nonnegative function $f$ that is flat
at the origin, and positive away from the origin, that is \emph{not} $SOS_{%
\func{regular}}$, answering in particular a question left open by Bony.

Refinements of these result are given for $f\in C^{4,2\delta }$, and the
related problem of extracting smooth positive roots from such smooth
functions is also considered.
\end{abstract}

\tableofcontents


\section{Introduction}

It is an open problem whether or not there are smooth nonnegative functions $%
\lambda $ on the real line (even vanishing only at the origin, and to
infinite order there), such that they \textbf{cannot} be written as a finite
sum $\lambda =\sum_{n=1}^{N}f_{n}^{2}$ of squares of smooth functions $f_{n}$%
. Examples of such functions are attributed to\ Paul Cohen in both \cite{Bru}
and \cite{BoCoRo}, but apparently no example has ever appeared in the
literature, and the existence of such an example is an open problem, see 
\cite[Remark 5.1]{Pie}\footnote{%
See also\ https://mathoverflow.net/a/106072}. Such sum of squares
decompositions are relevant to hypoellipticity questions in partial
differential equations, see e.g. H\"{o}rmander \cite{Ho}, and especially in
the infinitely degenerate regime, see e.g. Christ \cite{Chr} and references
given there. In particular we point to the theorem of Christ there that
asserts hypoellipticity for a second order differential operator $L$ if it
is a finite sum $\sum X_{k}^{\func{tr}}X_{k}$ of squares of smooth vector
fields $X_{k}$ satisfying certain conditions relevant to hypoellipticity in
the infinitely degenerate regime. In the third paper \cite{KoSa3} of this
series, the authors have extended this theorem to $C^{2,\delta }$ vector
fields, which is essentially optimal for second order operators. Thus for
partial differential equations, the crucial sum of squares question is this.

\begin{problem}
When can a nonnegative scalar or matrix function $f\left( x\right) $ on $%
\mathbb{R}^{n}$ be written as a sum of squares of $C^{2,\delta }$ scalar or
vector functions for some $\delta >0$?\footnote{%
In this paper we consider the scalar problem. The equally relevant problem of writing a nonnegative matrix function as a
sum of squares is treated in \cite{KoSa2}.}
\end{problem}

A well known and important construction of Fefferman and Phong in 1978, with
only a bare sketch of a proof given in \cite{FePh}, was used by Guan \cite[%
see the end of the paper]{Gua} in the mid 1990's to prove the following
result that Guan attributed to Fefferman: every smooth (even $C^{3,1}$)
nonnegative function $f$ on $\mathbb{R}^{n}$ can be written as a sum of
squares of $C^{1,1}$ functions. However, while this decomposition was a
perfect fit for the $C^{2}$ \emph{a priori} estimates proved for the
Monge-Ampere equation by Guan, this decomposition falls short for
applications which require $C^{2,\delta }$ coefficients or vector fields for
some $\delta >0$. The classical such application is Schauder theory, where $%
C^{2,\delta }$ coefficients play a pivotal role, and more importantly for us
is a generalization of a sum of squares theorem of Christ that we prove in 
\cite{KoSa3} using $C^{2,\delta }$ vector fields. As a consequence, we will
refer to a function $g$ in $\dbigcup\limits_{\delta >0}C^{2,\delta }$ as a 
\emph{regular} function, so that in the context of partial differential
equations in the infinitely degenerate regime at the origin, where some of
the coefficients are flat (i.e. vanish to an infinite order) at the origin,
the scalar question becomes this. We say that a scalar or matrix function is 
\emph{elliptical} if it is positive definite away from the origin.

\begin{problem}
\label{main prob}When can an elliptical flat smooth scalar function $f$ on $%
\mathbb{R}^{n}$ be written as a sum of squares of regular scalar functions?
\end{problem}

The corresponding question for elliptical \emph{finite type} smooth scalar
functions has been well studied in the wake of Hilbert's $17^{th}$ problem,
and there are algebraic obstructions to writing a smooth function as a sum
of squares of smooth functions. For example, the homogeneous Motzkin
polynomial $M$ in $n=3$ dimensions, and a generalization $L$ to dimension $%
n=4$,%
\begin{eqnarray*}
M\left( x,y,z\right) &=&z^{6}+x^{2}y^{2}\left( x^{2}+y^{2}-3\lambda
z^{2}\right) ,\ \ \ \ \ \left( x,y,z\right) \in \mathbb{R}^{3}, \\
L\left( x,y,z,w\right) &=&w^{4}+x^{2}y^{2}+y^{2}z^{2}+z^{2}x^{2}-4\lambda
xyzw,\ \ \ \ \ \left( x,y,z,w\right) \in \mathbb{R}^{4},
\end{eqnarray*}%
are nonnegative for $0\leq \lambda \leq 1$, vanish only at the origin for $%
0<\lambda <1$, and are \emph{not} finite sums of squares of polynomials for $%
0<\lambda \leq 1$, see \cite{BoBrCoPe}. As pointed out by Bony \cite{Bon}, Taylor expansions can
then be used to show that $M$ cannot be written as a finite sum of squares
of $C^{3}$ functions, and that $L$ cannot be written as a finite sum of
squares of $C^{2}$ functions. This latter observation, along with $L\left(
x,y,z,w\right) $ itself, will play a critical role in establishing sharpness
for finite sums of squares of regular functions.

Our main sum of squares theorem for scalar nonnegative functions gives a
sharp answer to this question in terms of an $\omega $\emph{-monotone
property}, defined below for any moduolus of continuity $\omega $, namely
that the answer to Problem \ref{main prob} is affirmative if $f$ is \emph{H%
\"{o}lder monotone}. On the other hand, part (2) of Theorem \ref{log
counter'} below, shows that the sum of squares decomposition can fail for
any $\omega $-monotone property weaker than H\"{o}lder monotone. In
particular, this settles a question left open in \cite[Remark 1.4 on page 141%
]{BoBrCoPe}, that asked if there exists an \emph{elliptical} flat smooth
function that is not a sum of squares of $C^{2,\omega }$ functions.

There are several notions of monotonicity\ for nonnegative functions of
several variables used in this paper, and we illustrate them here by giving
these definitions for functions $f\left( x\right) $ defined on the unit
interval $\left[ 0,1\right] $, with higher dimensional definitions given
later. Let $\omega \left( t\right) $ be a modulus of continuity defined on $%
\left[ 0,1\right] $, i.e. $\omega $ is continuous, nondecreasing and
strictly concave, and satisfies $\omega \left( 0\right) =0$ and $\omega
\left( 1\right) =1$. Then we define varying degrees of monotonicity that are
weaker than traditional monotonicity as follows.

\begin{definition}
\label{weakly}Suppose $f:\left[ 0,1\right] \rightarrow \left[ 0,\infty
\right) $ and that $\omega \left( t\right) $ is a modulus of continuity on $%
\left[ 0,1\right] $.

\begin{enumerate}
\item $f$ is $\omega $\emph{-monotone} if $0\leq f\left( y\right) \leq
C\omega \left( f\left( x\right) \right) $ for $0\leq y\leq x\leq 1$ and some
positive constant $C$,

\item $\omega _{s}\left( t\right) =\left\{ 
\begin{array}{ccc}
t\left( 1+\ln \frac{1}{t}\right) & \text{ if } & s=1 \\ 
t^{s} & \text{ if } & 0<s<1 \\ 
\frac{1}{1+\ln \frac{1}{t}} & \text{ if } & s=0%
\end{array}%
\right. $, for $0\leq s,t\leq 1$,

\item $f$ is \emph{nearly monotone} if $f$ is $\omega _{s}$-monotone for
every $0\leq s<1$,

\item $f$ is \emph{H\"{o}lder monotone} if $f$ is $\omega _{s}$-monotone for
some $0<s\leq 1$.
\end{enumerate}
\end{definition}

Then $\omega _{s}$ is a modulus of continuity for all $0\leq s\leq 1$, and 
\begin{equation*}
t\ll \omega _{1}\left( t\right) \ll \omega _{s^{\prime }}\left( t\right) \ll
\omega _{s}\left( t\right) \ll \omega _{0}\left( t\right) \ll 1,\ \ \ \ \
0<s<s^{\prime }\leq 1,
\end{equation*}%
where for positive functions on $\left( 0,1\right) $, $a\left( t\right) \ll
b\left( t\right) $ means lim$_{t\searrow 0}\frac{a\left( t\right) }{b\left(
t\right) }=0$\footnote{%
We define $\omega _{1}\left( t\right) $ to be dual to $\omega _{0}\left(
t\right) $ in the sense that $\omega _{1}\left( t\right) \omega _{0}\left(
t\right) =t$ (note that $\omega _{\frac{1}{p}}\left( t\right) \omega _{\frac{%
1}{p^{\prime }}}\left( t\right) =1$), and $\omega _{1}$ plays a role in
Theorem \ref{main intro} below.}.

We now extend the definition of $\omega $-monotone to several variables.

\begin{definition}
Given a modulus of continuity $\omega $, we say that a function $f:\mathbb{R}
^{n}\rightarrow \left[ 0,\infty \right) $ is $\omega $\emph{-monotone} if $%
0\leq f\left( y\right) \leq C\omega \left( f\left( x\right) \right) $ for $%
y\in B\left( \frac{x}{2},\frac{\left\vert x\right\vert }{2}\right) $ and
some positive constant $C$, and we set 
\begin{equation}
\left\Vert f\right\Vert _{\omega -\func{mon}}\equiv \sup_{x\in B\left(
0,1\right) ,\ y\in B\left( \frac{x}{2},\frac{\left\vert x\right\vert }{2}
\right) }\frac{f\left( y\right) }{\omega \left( f\left( x\right) \right) }.
\label{omega mon}
\end{equation}
We say that $f$ is \emph{H\"{o}lder monotone} if $f$ is $\omega _{s}$
-monotone for some $0<s\leq 1$.
\end{definition}

See Remark \ref{ball remark} below for a discussion of the condition $y\in
B\left( \frac{x}{2},\frac{\left\vert x\right\vert }{2}\right) $ in (\ref%
{omega mon}).

We begin by stating our main results on sums of squares in Section 2.
Section 3 is devoted to connections between vanishing to infinite order,
derivative estimates, the notion of $\omega $-monotone for a modulus of
continuity $\omega $, and smoothness of positive roots. Section 4 then uses
ideas of Fefferman-Phong \cite{FePh}, Tataru \cite{Tat} and Bony \cite{Bon}
to establish conditions under which smooth nonnegative functions can be
written as a sum of squares of regular functions. In the final Section 5 of
the paper, we construct examples that demonstrate sharpness of our sums of
regular squares results.

For the reader's convenience we include a schematic diagram of connections
between some of the lemmas and theorems in this paper. Results in a double
box are logical ends. Theorem \ref{diff prov thm} is the main result that will be used in the
subsequent papers \cite{KoSa2} and \cite{KoSa3}.

\begin{equation*}
\fbox{$%
\begin{array}{ccccccc}
&  & \overset{\func{start}}{\fbox{$%
\begin{array}{c}
\text{Lemma}\ \mathbf{\ref{first high}} \\ 
\max_{I}\left\vert \nabla ^{m}f\right\vert \text{\ }control%
\end{array}%
$}} &  & \overset{\func{start}}{\fbox{$%
\begin{array}{c}
\text{Lemma}s\ \mathbf{\ref{first local}, \ref{second}} \\ 
odd\lesssim even \\ 
\text{slowly\ varying}%
\end{array}%
$}} &  &  \\ 
& \swarrow & \downarrow &  & \downarrow &  &  \\ 
\overset{\func{end}}{\fbox{\fbox{$%
\begin{array}{c}
\text{Theorem }\mathbf{\ref{main intro}} \\ 
f\ EFS\ \text{and }NM \\ 
\Longrightarrow f^{\gamma }\text{ smooth} \\ 
\text{for all }\gamma >0\text{ }%
\end{array}%
$}}} &  & \downarrow &  & \fbox{$%
\begin{array}{c}
\text{Theorem}\ \mathbf{\ref{provisional}} \\ 
\text{diff\ ineq}\Longrightarrow \\ 
SOS_{\func{reg}}+n-1%
\end{array}%
$} &  &  \\ 
& \swarrow & \downarrow &  & \downarrow & \searrow &  \\ 
\overset{\func{end}}{\fbox{\fbox{$%
\begin{array}{c}
\text{Theorem}\ \mathbf{\ref{1/4 thm}} \\ 
s>\frac{3}{4}\Longrightarrow \\ 
\sqrt{f}\ regular%
\end{array}%
$}}} &  & \fbox{$%
\begin{array}{c}
\text{Theorem}\ \mathbf{\ref{s'^m}} \\ 
f\ \omega _{s}\text{-}M\Longrightarrow \\ 
\left\vert \nabla ^{m}f\right\vert \lesssim f^{\left( s^{\prime }\right)
^{m}}%
\end{array}%
$} &  & \fbox{$%
\begin{array}{c}
\text{Theorem}\ \mathbf{\ref{efs eps}} \\ 
\text{diff\ ineq}\Longrightarrow \\ 
SOS_{\func{reg}}%
\end{array}%
$} &  & \overset{\func{end}}{\fbox{\fbox{$%
\begin{array}{c}
\text{Theorem}\ \mathbf{\ref{main 2D}} \\ 
2D\ SOS_{\func{reg}}%
\end{array}%
$}}} \\ 
&  & \downarrow & \swarrow &  &  &  \\ 
&  & \overset{\func{end}}{\fbox{\fbox{$%
\begin{array}{c}
\text{Theorem}\ \mathbf{\ref{diff prov thm}\ }\text{Part }(1) \\ 
f\ EFS\ HM\Longrightarrow SOS_{\func{reg}}%
\end{array}%
$}}} &  & \overset{\func{end}}{\fbox{$\fbox{$%
\begin{array}{c}
\text{Theorem}\ \mathbf{\ref{log counter'}\ }\text{Part }(2) \\ 
\omega \gg \omega _{s}\Longrightarrow \\ 
\exists \ f\text{ }\omega \text{-monotone} \\ 
\text{but not }SOS_{\func{reg}}%
\end{array}%
$}$}} &  & 
\end{array}%
$}
\end{equation*}

\section{Statements of main theorems on sums of squares and extracting roots}

Here is our adaptation of the Fefferman-Phong algorithm, following Tataru 
\cite{Tat} and Bony \cite{Bon}, to sums of squares of regular functions.

\begin{theorem}
\label{diff prov thm} Suppose $0<\delta ,\eta <\frac{1}{2}$ and that $f$ is
a nonnegative $C^{4,2\delta }$ function on $\mathbb{R}^{n}$. If%
\begin{equation}
\left\vert \nabla ^{4}f\left( x\right) \right\vert \leq Cf\left( x\right) ^{%
\frac{\delta }{2+\delta }},\text{ and }\sup_{\Theta \in \mathbb{S}^{n-1}}%
\left[ \partial _{\Theta }^{2}f\left( x\right) \right] _{+}\leq Cf\left(
x\right) ^{\eta },  \label{diff prov'}
\end{equation}%
then $f$ can be decomposed as a finite sum of squares of functions $g_{\ell
}\in C^{2,\delta _{n-1}}\left( \mathbb{R}^{n}\right) $, 
\begin{equation*}
f\left( x\right) =\sum_{\ell =1}^{N}g_{\ell }\left( x\right) ^{2},\ \ \ \ \
x\in \mathbb{R}^{n},
\end{equation*}%
where $\delta _{n-1}$ is defined recursively by $\delta _{0}=\delta $ and%
\begin{equation*}
\frac{\delta _{k+1}}{2+\delta _{k+1}}=\eta \frac{\delta _{k}}{1+\delta _{k}}%
,\ \ \ \ \ 0\leq k\leq n-2.
\end{equation*}%
In the case that $f$ doesn't vanish, except possibly at the origin, both of
the above differential inequalities in (\ref{diff prov'}) hold provided one
of the following three conditions hold:

\begin{enumerate}
\item $f$ is flat, smooth and $\omega _{s}$-monotone for some $s<1$
satisfying%
\begin{equation*}
s>\sqrt[4]{\frac{\delta }{2+\delta }},\ \ \ \text{and \ \ \ }s>\sqrt{\eta }.
\end{equation*}

\item $f$ is strongly finite type, i.e. $f\left( x\right) \geq \left\vert
x\right\vert ^{N}$ for some $N\in \mathbb{N}$, and vanishes to order at
least four.

\item $f$ is bounded below by a positive constant.
\end{enumerate}
\end{theorem}

More detailed information on the size and smoothness of the functions $%
g_{\ell }\left( x\right) $ is given in Theorem \ref{efs eps} below.

\begin{definition}
A nonnegative function $f:B\left( 0,a\right) \rightarrow \mathbb{R}$ is 
\emph{flat}, or vanishes to infinite order at the origin in $\mathbb{R}^{n}$%
, if 
\begin{equation*}
\lim_{x\rightarrow 0}\left\vert x\right\vert ^{-N}f\left( x\right) =0,\ \ \
\ \ \text{for all }N\in \mathbb{N}.
\end{equation*}
\end{definition}

\begin{definition}
A function $g:\mathbb{R}^{n}\rightarrow \mathbb{R}$ is \emph{regular} if $%
g\in \dbigcup\limits_{\delta >0}C^{2,\delta }\left( \mathbb{R}^{n}\right) $,
i.e. $g$ is $C^{2,\delta }$ for some $0<\delta <1$.
\end{definition}

\begin{definition}
A function $f:\mathbb{R}^{n}\rightarrow \mathbb{R}$ is \emph{elliptical} if $%
f\left( x\right) >0$ for $x\neq 0$, and more generally, an $N\times N$
matrix-valued function $F:\mathbb{R}^{n}\rightarrow \mathbb{R}^{N\times N}$
is \emph{elliptical} if $F\left( x\right) $ is positive definite for $x\neq
0 $.
\end{definition}

Here is a strengthening of the counterexample in \cite[Theorem 1.2 (d)]%
{BoBrCoPe}.

\begin{theorem}
\label{log counter'}Let $n\geq 5$.

\begin{enumerate}
\item If $\beta >s>0$, there is an elliptical, flat, smooth $\omega _{s}$%
-monotone function $f$ that cannot be written as a finite sum of squares of $%
C^{2,\beta }\left( \mathbb{R}^{n}\right) $.

\item Suppose $\omega $ is a modulus of continuity such that $\omega _{s}\ll
\omega $ for all $0<s<1$. Then there is an elliptical, flat, smooth $\omega $%
-monotone function $f$ that cannot be written as a finite sum of squares of
regular functions. In particular we can take $\omega =\omega _{0}$.
\end{enumerate}
\end{theorem}

The following corollary highlights the\ sharpness of the above results
within the scale of $\omega $-monotone conditions.

\begin{corollary}
Suppose that $f:\mathbb{R}^{n}\rightarrow \left[ 0,\infty \right) $ is
elliptical, flat and smooth.

\begin{enumerate}
\item Then $f$ can written as a finite sum of squares of regular functions
if $f$ is H\"{o}lder monotone.

\item Conversely, for any modulus of continuity $\omega $ satisfying $\omega
\gg \omega _{s}$ for all $0<s<1$, there is an $\omega $-monotone function $f$
that cannot be written as a finite sum of squares of regular functions.
\end{enumerate}
\end{corollary}

\begin{remark}
The answer to Problem \ref{main prob} is affirmative in all dimensions $%
n\geq 1$ if $f$ is H\"{o}lder monotone. However, in dimension $n=1$, the
answer is affirmative without any additional $\omega $-monotone assumptions
at all \cite{Bon}, while in dimension $n\geq 5$, the assumption of H\"{o}%
lder monotone is essentially sharp. We do not know if significantly weaker
montonicity assumptions will imply an affirmative answer to Problem \ref%
{main prob} in the remaining dimensions $n=2,3,4$.
\end{remark}

The following theorem on extracting smooth roots motivates our definition of
nearly monotone, and is what initially led us to consider $\omega _{s}$%
-monotone functions, and which then ultimately played a key role in the
above decompositions into sums of regular\ functions.

\begin{theorem}
\label{main intro}Let $n\geq 1$. Suppose that $f:B\left( 0,a\right)
\rightarrow \left[ 0,\infty \right) $ is an elliptical flat smooth function
on $B\left( 0,a\right) \subset \mathbb{R}^{n}$. Then the first three of the
following four conditions are equivalent. Moreover, the fourth condition,
which holds in particular if $f$ is $\omega _{1}$ monotone, implies the
first three conditions, but not conversely. Finally, for any $0<s<1$, there
is an $\omega _{s}$-monotone function $f$ such that $f^{\frac{1}{s}-1}$ is
not smooth.

\begin{enumerate}
\item There is $\delta >0$ such that $f\left( x\right) ^{\gamma }$ is smooth
on $B\left( 0,a\right) $ for all $0<\gamma <\delta $.

\item For every $m\geq 1$ and $0<s<1$, there is a positive constant $\Gamma
_{n,m,s}$ such that%
\begin{equation*}
\left\vert \nabla ^{m}f\left( x\right) \right\vert \leq \Gamma
_{n,m,s}f\left( x\right) ^{s},\ \ \ \ \ \text{for }x\in B\left( 0,a\right) .
\end{equation*}

\item The functions $f\left( x\right) ^{\gamma }$ are flat smooth functions
on $B\left( 0,a\right) $ for all $\gamma >0$.

\item The function $f$ is nearly monotone.
\end{enumerate}
\end{theorem}

\begin{description}
\item[Terminology] If for $0<s<1$, we set $\mathcal{M}_{s}$ to be the
collection of all $\omega _{s}$-monotone functions, then the \textbf{%
intersection} of all of these collections 
\begin{equation*}
\mathcal{NM\equiv }\dbigcap\limits_{0<s<1}\mathcal{M}_{s},
\end{equation*}%
which is the set of \emph{nearly monotone} functions, is closely related to
the set of $elliptical\ flat\ smooth$ functions all of whose positive powers
are smooth. On the other hand, the \textbf{union} of all of these
collections 
\begin{equation*}
\mathcal{HM\equiv }\dbigcup\limits_{0<s<1}\mathcal{M}_{s},
\end{equation*}%
which is the set of \emph{H\"{o}lder monotone} functions, is closely related
to the set of $elliptical\ flat\ smooth$ functions that can be written as a
finite sum of squares of regular functions.
\end{description}

\section{Smooth nonnegative functions}

Here we discuss some connections between vanishing to infinite order,
derivative estimates, and the notion of $\omega $-monotone for a modulus of
continuity $\omega $, and finish with the proof of Theorem \ref{main intro}
on smoothness of positive roots. We use the following notation,%
\begin{eqnarray}
	\nabla ^{m} &=&\nabla \otimes \nabla \otimes ...\otimes \nabla =\left[
	\partial _{i_{1}}\partial _{i_{2}}...\partial _{i_{m}}\right] _{1\leq
		i_{1},i_{2},...i_{m}\leq n},  \label{not} \\
	\left( \nabla h\right) ^{m} &=&\nabla h\otimes \nabla h\otimes ...\otimes
	\nabla h=\left[ \left( \partial _{i_{1}}h\right) \left( \partial
	_{i_{2}}h\right) ...\left( \partial _{i_{m}}h\right) \right] _{1\leq
		i_{1},i_{2},...i_{m}\leq n},  \notag \\
	\left( x\cdot \nabla \right) ^{m} &=&\left( x_{1}\partial _{1}+\dots+x_{n}\partial
	_{n}\right) ^{m}=\sum_{i=\left( i_{1},i_{2},...i_{m}\right) \in \left\{
		1,2,...,n\right\} ^{m}\equiv \Gamma^{m}}x_{i_{1}}x_{i_{2}}...x_{i_{m}}\partial
	_{i_{1}}\partial _{i_{2}}...\partial _{i_{m}}=\sum_{i\in \Gamma
		^{m}}x^{i}\partial_{i}.  \notag
\end{eqnarray}

\subsection{Infinite order vanishing}

\begin{lemma}
\label{smooth}Suppose that $f:B\left( 0,a\right) \rightarrow \left( -\infty
,\infty \right) $ is a flat smooth function. Then $D^{\mu }f$\ is a flat
smooth function for all multiindices $\mu \in \mathbb{Z}_{+}^{n}$.
\end{lemma}

\begin{proof}
Clearly $D^{\mu }f$\ is smooth for all multiindices $\mu \in \mathbb{Z}%
_{+}^{n}$. We first use induction on $m$ to establish that%
\begin{equation}
D^{\mu }f\left( 0\right) =0\text{\ \ \ \ \ for all }\left\vert \mu
\right\vert \leq m,  \label{est}
\end{equation}%
holds for all $m\in \mathbb{N}$. The case $m=0$ is a tautology, so suppose $%
M\in \mathbb{N}$ and that (\ref{est}) holds for all $m<M$. Then Taylor's
formula of order $M$ yields%
\begin{equation*}
f\left( x\right) =\sum_{\left\vert \mu \right\vert =M}\left( 
\begin{array}{c}
\left\vert \mu \right\vert \\ 
\mu%
\end{array}%
\right) D^{\mu }f\left( 0\right) x^{\mu }+O\left( \left\vert x\right\vert
^{M+1}\right) ,
\end{equation*}%
and thus%
\begin{equation*}
\left\vert \sum_{\left\vert \mu \right\vert =M}\left( 
\begin{array}{c}
\left\vert \mu \right\vert \\ 
\mu%
\end{array}%
\right) D^{\mu }f\left( 0\right) x^{\mu }\right\vert \leq C\left\vert
x\right\vert ^{M+1},
\end{equation*}%
since $f$ vanishes to infinite order. It follows that the polynomial $%
\sum_{\left\vert \mu \right\vert =M}\left( 
\begin{array}{c}
\left\vert \mu \right\vert \\ 
\mu%
\end{array}%
\right) D^{\mu }f\left( 0\right) x^{\mu }$ vanishes identically, and so $%
D^{\mu }f\left( 0\right) =0$ for all $\left\vert \mu \right\vert =M$, which
gives (\ref{est}) for $m=M$, and completes the inductive proof.

Now apply Taylor's formula of order $k-1$ to the smooth function $%
t\rightarrow D^{\mu }f\left( tx\right) $ to obtain%
\begin{equation*}
D^{\mu }f\left( x\right) =\sum_{\ell =0}^{k-1}\frac{1}{\ell !}\left( x\cdot
\nabla \right) ^{\ell }D^{\mu }f\left( 0\right) +\frac{\left( x\cdot \nabla
\right) ^{k}}{k!}D^{\mu }f\left( \theta _{x}x\right) =\frac{\left( x\cdot
\nabla \right) ^{k}}{k!}D^{\mu }f\left( \theta _{x}x\right) ,
\end{equation*}%
where $0<\theta _{x}<1$, and thus the smoothness of $f$ implies%
\begin{equation*}
\left\vert D^{\mu }f\left( x\right) \right\vert \leq C_{k,\mu }\left\vert
x\right\vert ^{k},
\end{equation*}%
which shows that $D^{\mu }f$ is flat for all multiindices $\mu \in \mathbb{Z}%
_{+}^{n}$.
\end{proof}

\begin{remark}
\label{smooth away}If $f$ is a flat function that is smooth only on $\left(
-a,a\right) \setminus \left\{ 0\right\} $, then its derivative need not be
bounded in any open interval $\left( 0,\varepsilon \right) $ for $%
\varepsilon >0$, e.g. 
\begin{eqnarray*}
f\left( x\right) &=&e^{-\frac{1}{x}}\left( \sin ^{2}e^{\frac{1}{x}}+1\right)
; \\
f^{\prime }\left( x\right) &=&\frac{e^{-\frac{1}{x}}}{x^{2}}\left( \sin
^{2}e^{\frac{1}{x}}+1\right) -e^{-\frac{1}{x}}\left( 2\sin e^{\frac{1}{x}%
}\cos e^{\frac{1}{x}}\right) \frac{e^{\frac{1}{x}}}{x^{2}} \\
&=&\frac{e^{-\frac{1}{x}}}{x^{2}}\left( \sin ^{2}e^{\frac{1}{x}}+1\right) -%
\frac{1}{x^{2}}\left( \sin 2e^{\frac{1}{x}}\right) ,
\end{eqnarray*}%
where $\lim_{n\rightarrow \infty }f^{\prime }\left( x_{n}\right) =-\infty $
if $x_{n}\searrow 0$ and $\sin 2e^{\frac{1}{x_{n}}}=1$ for all $n\geq 1$.
\end{remark}

\subsection{Derivative estimates}

The proof of our first main theorem will use a generalization of Lemma 5.13
from \cite{GuSa}, which is the case $n=m=1$ of the following lemma. We
denote the diameter of a ball $B$ by $\ell \left( B\right) $.

\begin{lemma}
\label{first high}For each triple of integers $k,m,n\in \mathbb{N}$ with $%
k\geq m\geq 1$ and$\ n\geq 1$, there is a constant $C_{k,m,n}>0$ \ such that
for any ball $B$ in $\mathbb{R}^{n}$ and $f\in C^{k-1,1}\left( B\right) $,
we have%
\begin{eqnarray*}
&&\max_{z\in B}\left\vert \nabla ^{m}f\left( z\right) \right\vert \leq
C_{k,m,n}\frac{1}{\ell \left( B\right) ^{m}}\max_{t_{1},t_{2}\in
B}\left\vert f\left( t_{1}\right) -f\left( t_{2}\right) -\left(
t_{1}-t_{2}\right) \cdot \nabla f\left( t_{2}\right) -...-\frac{\left[
\left( t_{1}-t_{2}\right) \cdot \nabla \right] ^{m-1}}{\left( m-1\right) !}%
f\left( t_{2}\right) \right\vert \\
&&+C_{k,m,n}\left( \max_{t_{1},t_{2}\in B}\left\vert f\left( t_{1}\right)
-f\left( t_{2}\right) -\left( t_{1}-t_{2}\right) \cdot \nabla f\left(
t_{2}\right) -...-\frac{\left[ \left( t_{1}-t_{2}\right) \cdot \nabla \right]
^{m-1}}{\left( m-1\right) !}f\left( t_{2}\right) \right\vert \right) ^{1-%
\frac{m}{k}} \\
&&\ \ \ \ \ \ \ \ \ \ \ \ \ \ \ \ \ \ \ \ \ \ \ \ \ \ \ \ \ \ \ \ \ \ \ \ \
\ \ \ \ \ \ \times \left( \max_{t\in B}\left\vert \nabla ^{k}f\left(
t\right) \right\vert \right) ^{\frac{m}{k}}.
\end{eqnarray*}
\end{lemma}

\begin{proof}
Let $k\geq m\geq 1$ and $n\geq 1$, and fix $z\in B$. Taylor's formula gives%
\begin{equation*}
f\left( t\right) =f\left( z\right) +\left( t-z\right) \cdot \nabla f\left(
z\right) +\dots +\frac{\left[ \left( t-z\right) \cdot \nabla \right] ^{m}}{m!%
}f\left( z\right) +...+\frac{\left[ \left( t-z\right) \cdot \nabla \right]
^{k-1}}{\left( k-1\right) !}f\left( z\right) +r\left( t\right) .
\end{equation*}%
Define 
\begin{eqnarray*}
P\left( t\right) &=&\frac{\left[ \left( t-z\right) \cdot \nabla \right] ^{m}%
}{m!}f\left( z\right) +...+\frac{\left[ \left( t-z\right) \cdot \nabla %
\right] ^{k-1}}{\left( k-1\right) !}f\left( z\right) \\
&=&f\left( t\right) -f\left( z\right) -\left( t-z\right) \cdot \nabla
f\left( z\right) -\dots -\frac{\left[ \left( t-z\right) \cdot \nabla \right]
^{m-1}}{(m-1)!}f\left( z\right) -r\left( t\right) .
\end{eqnarray*}%
Let $D$ be a ball such that $z\in D\subset B$ and $\ell \left( D\right)
=\min \left\{ \ell \left( B\right) ,\delta \right\} $ where $\ell \left(
D\right) $ denotes the radius of the ball $D$ and where 
\begin{equation*}
\delta =\left( \frac{\max_{t_{1},t_{2}\in B}\left\vert f\left( t_{1}\right)
-f\left( t_{2}\right) -\left( t_{1}-t_{2}\right) \cdot \nabla f\left(
t_{2}\right) -\dots \frac{\left[ \left( t_{1}-t_{2}\right) \cdot \nabla %
\right] ^{m-1}}{(m-1)!}f\left( t_{2}\right) \right\vert }{\max_{t\in
B}\left\vert \nabla ^{k}f\left( t\right) \right\vert }\right) ^{\frac{1}{k}}.
\end{equation*}%
Since $P$ is a polynomial of degree $k-1$ there is a constant $C_{k,m}>0$,
independent of $D$ and $B$, such that%
\begin{equation*}
\max_{t\in D}\left\vert \nabla ^{m}P\left( t\right) \right\vert \leq C_{k,m}%
\frac{1}{\ell \left( D\right) ^{m}}\max_{t\in D}\left\vert P\left( t\right)
\right\vert .
\end{equation*}%
Then using a standard estimate for the remainder in Taylor's formula we get%
\begin{eqnarray}
\left\vert \nabla ^{m}f\left( z\right) \right\vert &=&\left\vert \nabla
^{m}P\left( z\right) \right\vert \leq \max_{t\in D}\left\vert \nabla
^{m}P\left( t\right) \right\vert \leq C_{k,m}\frac{1}{\ell \left( D\right)
^{m}}\max_{t\in D}\left\vert P\left( t\right) \right\vert  \notag
\label{dom high'} \\
&\leq &\frac{1}{\ell \left( D\right) ^{m}}\max_{t_{1},t_{2}\in D}\left\vert
f\left( t_{1}\right) -f\left( t_{2}\right) -\left( t_{1}-t_{2}\right) \cdot
\nabla f\left( t_{2}\right) -\dots -\frac{\left[ \left( t_{1}-t_{2}\right)
\cdot \nabla \right] ^{m-1}}{(m-1)!}f\left( t_{2}\right) \right\vert \\
&\quad &+\frac{1}{\ell \left( D\right) ^{m}}\max_{t\in D}\left\vert r\left(
t\right) \right\vert  \notag \\
&\leq &\frac{1}{\ell \left( D\right) ^{m}}\max_{t_{1},t_{2}\in D}\left\vert
f\left( t_{1}\right) -f\left( t_{2}\right) -\dots -\frac{\left[ \left(
t_{1}-t_{2}\right) \cdot \nabla \right] ^{m-1}}{(m-1)!}f\left( t_{2}\right)
\right\vert +\frac{1}{\ell \left( D\right) ^{m}}\max_{t\in D}\left\vert
\nabla ^{k}f\left( t\right) \right\vert \ell \left( D\right) ^{k}.  \notag
\end{eqnarray}%
Moreover, $\ell \left( D\right) ^{k-m}\leq \delta ^{k-m}$ gives 
\begin{eqnarray*}
&&\frac{1}{\ell \left( D\right) ^{m}}\max_{t\in D}\left\vert \nabla
^{k}f\left( t\right) \right\vert \ell \left( D\right) ^{k}=\max_{t\in
D}\left\vert \nabla ^{k}f\left( t\right) \right\vert \ell \left( D\right)
^{k-m}\leq \max_{t\in D}\left\vert \nabla ^{k}f\left( t\right) \right\vert
\delta ^{k-m} \\
&&\quad =\max_{t\in D}\left\vert \nabla ^{k}f\left( t\right) \right\vert
\left( \frac{\max_{t_{1},t_{2}\in B}\left\vert f\left( t_{1}\right) -f\left(
t_{2}\right) -\left( t_{1}-t_{2}\right) \cdot \nabla f\left( t_{2}\right)
-\dots \frac{\left[ \left( t_{1}-t_{2}\right) \cdot \nabla \right] ^{m-1}}{%
(m-1)!}f\left( t_{2}\right) \right\vert }{\max_{t\in B}\left\vert \nabla
^{k}f\left( t\right) \right\vert }\right) ^{\frac{k-m}{k}} \\
&&\quad \leq \left( \max_{t\in B}\left\vert \nabla ^{k}f\left( t\right)
\right\vert \right) ^{\frac{m}{k}}\left( \max_{t_{1}t_{2},\in B}\left\vert
f\left( t_{1}\right) -f\left( t_{2}\right) -\left( t_{1}-t_{2}\right) \cdot
\nabla f\left( t_{2}\right) -\dots \frac{\left[ \left( t_{1}-t_{2}\right)
\cdot \nabla \right] ^{m-1}}{(m-1)!}f\left( t_{2}\right) \right\vert \right)
^{1-\frac{m}{k}}.
\end{eqnarray*}%
Now in the case that $\delta \leq \ell \left( B\right) $, we are done since
then $\ell \left( D\right) =\delta $ and 
\begin{eqnarray*}
&&\frac{1}{\ell \left( D\right) ^{m}}\max_{t_{1},t_{2}\in D}\left\vert
f\left( t_{1}\right) -f\left( t_{2}\right) -\left( t_{1}-t_{2}\right) \cdot
\nabla f\left( t_{2}\right) -\dots \frac{\left[ \left( t_{1}-t_{2}\right)
\cdot \nabla \right] ^{m-1}}{(m-1)!}f\left( t_{2}\right) \right\vert \\
&=&\left( \frac{\max_{t_{1},t_{2}\in B}\left\vert f\left( t_{1}\right)
-f\left( t_{2}\right) -\left( t_{1}-t_{2}\right) \cdot \nabla f\left(
t_{2}\right) -\dots \frac{\left[ \left( t_{1}-t_{2}\right) \cdot \nabla %
\right] ^{m-1}}{(m-1)!}f\left( t_{2}\right) \right\vert }{\max_{t\in
B}\left\vert \nabla ^{k}f\left( t\right) \right\vert }\right) ^{-\frac{m}{k}}
\\
&&\quad \times \max_{t_{1},t_{2}\in D}\left\vert f\left( t_{1}\right)
-f\left( t_{2}\right) -\left( t_{1}-t_{2}\right) \cdot \nabla f\left(
t_{2}\right) -\dots \frac{\left[ \left( t_{1}-t_{2}\right) \cdot \nabla %
\right] ^{m-1}}{(m-1)!}f\left( t_{2}\right) \right\vert \\
&\leq &\left( \max_{t_{1},t_{2}\in B}\left\vert f\left( t_{1}\right)
-f\left( t_{2}\right) -\left( t_{1}-t_{2}\right) \cdot \nabla f\left(
t_{2}\right) -\dots \frac{\left[ \left( t_{1}-t_{2}\right) \cdot \nabla %
\right] ^{m-1}}{(m-1)!}f\left( t_{2}\right) \right\vert \right) ^{1-\frac{m}{%
k}}\left( \max_{t\in B}\left\vert \nabla ^{k}f\left( t\right) \right\vert
\right) ^{\frac{m}{k}}.
\end{eqnarray*}%
On the other hand, in the case $\delta >\ell (B)$, we have $D=$ $B$ and the
inequality 
\begin{equation*}
\max_{t\in B}\left\vert \nabla ^{k}f\left( t\right) \right\vert \ell
(B)^{k}<\max_{t_{1}t_{2},\in B}\left\vert f\left( t_{1}\right) -f\left(
t_{2}\right) -\left( t_{1}-t_{2}\right) \cdot \nabla f\left( t_{2}\right)
-\dots \frac{\left[ \left( t_{1}-t_{2}\right) \cdot \nabla \right] ^{m-1}}{%
(m-1)!}f\left( t_{2}\right) \right\vert .
\end{equation*}%
Thus from (\ref{dom high'}) with $D=B$ we conclude that 
\begin{eqnarray*}
\left\vert \nabla ^{m}f\left( z\right) \right\vert &\leq &\frac{1}{\ell
(B)^{m}}\max_{t_{1},t_{2}\in B}\left\vert f\left( t_{1}\right) -f\left(
t_{2}\right) -\left( t_{1}-t_{2}\right) \cdot \nabla f\left( t_{2}\right)
-\dots \frac{\left[ \left( t_{1}-t_{2}\right) \cdot \nabla \right] ^{m-1}}{%
(m-1)!}f\left( t_{2}\right) \right\vert \\
&&\quad +\frac{1}{\ell (B)^{m}}\max_{t\in B}\left\vert \nabla ^{k}f\left(
t\right) \right\vert \ell (B)^{k} \\
&\leq &2\frac{1}{\ell (B)^{m}}\max_{t_{1},t_{2}\in B}\left\vert f\left(
t_{1}\right) -f\left( t_{2}\right) -\left( t_{1}-t_{2}\right) \cdot \nabla
f\left( t_{2}\right) -\dots \frac{\left[ \left( t_{1}-t_{2}\right) \cdot
\nabla \right] ^{m-1}}{(m-1)!}f\left( t_{2}\right) \right\vert ,
\end{eqnarray*}%
which completes the proof.
\end{proof}

\subsection{Nearly monotone functions}

In the introduction, we have defined the notion of an $\omega _{s}$-monotone
function on intervals of the real line. The extension of this definition to
higher dimensions is for the most part straightforward, with the only
wrinkle being the region over which the supremum is to be taken. Here is the
generalization of $\omega $-monotone to higher dimensions.

\begin{definition}
\label{weakly high}Let $B\left( 0,a\right) $ be the ball of radius $a$
centered at the origin $\mathbb{R}^{n}$. Define a nonnegative function $%
f:B\left( 0,a\right) \rightarrow \left[ 0,\infty \right) $ to be $\omega $%
\emph{-monotone in }$B\left( 0,a\right) $\emph{\ }if there is a positive $C$
such that%
\begin{equation*}
f\left( t\right) \leq C\omega \left( f\left( x\right) \right) ,\ \ \ \ \ 
\text{for all }t\in B\left( \frac{x}{2},\frac{\left\vert x\right\vert }{2}%
\right) ,\ \text{and }x\in B\left( 0,a\right) .
\end{equation*}
\end{definition}

\begin{remark}
If $f\left( x\right) =g\left( \left\vert x\right\vert \right) $ is a radial
function, then $f$ is $\omega $\emph{-}monotone in $B\left( 0,a\right) $ if
and only if $g$ is $\omega $\emph{-}monotone in $\left( 0,a\right) $.
\end{remark}

\begin{remark}
\label{ball remark}We make some comments on the role played by the ball $%
B_{x}=B\left( \frac{x}{2},\frac{\left\vert x\right\vert }{2}\right) $ in the
higher dimensional definition of $\omega $-monotone. First we point out that
the family of balls $\left\{ B_{x}\right\} _{x\in \mathbb{R}^{n}}$ is \emph{%
dilation and rotation invariant} in the sense that $B_{\delta \Theta
x}=\delta \Theta B_{x}$ for all rotations $\Theta $ and dilations $\delta $.
We now claim that for the purposes of this paper, the family of balls can be
replaced by any \emph{dilation and rotation invariant} family of open convex
sets $\left\{ E_{x}\right\} _{x\in \mathbb{R}^{n}}$ satisfying (\textbf{i}) $%
0,x\in \partial E_{x}$, (\textbf{ii})\ the eccentricity of $E_{x}$ is
uniformly controlled in $x$, and (\textbf{iii}) the set $E_{x}$ is starlike
with respect to each of its boundary points. Indeed, such sets $E_{x}$ are
not zero sets for polynomials, and so the rescaling argument used in Lemma %
\ref{first high} remains in force. The starlike property is used in the
proof of Lemma \ref{first high} to show that given any point $z\in E_{x}$
and any number $0<\varepsilon \leq 1$, there is a set $D$ that is a
translate, dilate and rotation of $E_{x}$, and that satisfies $z\in D\subset
E_{x}$ and $\func{diam}D=\varepsilon \func{diam}E_{x}$. Finally, the
important property in Theorem \ref{efs eps} below that (\ref{eps delta})
implies (\ref{diff prov}) also remains in force. However, it appears that
the definitions of $\omega $-monotonicity using these more general families
of convex sets are essentially equivalent when restricted to elliptical flat
smooth functions, and so nothing significant appears to be gained by their
use.\newline
As an example of such a family in the plane, we mention the case when $E_{x}$
is the tilted square having opposite corners at $0$ and $x$.
\end{remark}

A nearly monotone function on the line is quite close to being monotone,
while a H\"{o}lder monotone function can be far removed from being monotone,
but not as far removed from monotone as is an $\omega _{0}$-monotone
function with logarithmic modulus of continuity $\omega _{0}\left( t\right) =%
\frac{1}{1+\ln \frac{1}{t}}$. Each of these notions in higher dimensions
will play a role in this paper.

We need two more results in preparation for the proof of our near
characterization of elliptical flat smooth functions having smooth positive
powers.

First, we recall a more general version of an elementary composition formula
from \cite{MaSaUrVu}. Let $\psi :\mathbb{R}\rightarrow \mathbb{R}$ and $h:%
\mathbb{R}^{n}\rightarrow \mathbb{R}$ be two smooth functions. With $\psi
^{\left( k\right) }$ understood to be $\psi ^{\left( k\right) }\circ h$ on
the right hand side, we have%
\begin{eqnarray*}
\left[ \partial _{i}\left( \psi \circ h\right) \right] _{1\leq i\leq n}
&=&\psi ^{\prime }\left[ \partial _{i}h\right] _{1\leq i\leq n}, \\
\left[ \partial _{j}\partial _{i}\left( \psi \circ h\right) \right] _{1\leq
i,j\leq n} &=&\psi ^{\prime \prime }\left[ \left( \partial _{j}h\right)
\left( \partial _{i}h\right) \right] _{1\leq i,j\leq n}+\psi ^{\prime }\left[
\partial _{j}\partial _{i}h\right] _{1\leq i,j\leq n}, \\
\left[ \partial _{k}\partial _{j}\partial _{i}\left( \psi \circ h\right) %
\right] _{1\leq i,j,k\leq n} &=&\psi ^{\prime \prime \prime }\left[ \left(
\partial _{k}h\right) \left( \partial _{j}h\right) \left( \partial
_{i}h\right) \right] _{1\leq i,j,k\leq n}+\psi ^{\prime \prime }\left[
\left( \partial _{k}\partial _{j}h\right) \left( \partial _{i}h\right)
+\left( \partial _{j}h\right) \left( \partial _{k}\partial _{i}h\right) %
\right] _{1\leq i,j\leq n} \\
&&+\psi ^{\prime \prime }\left[ \left( \partial _{k}h\right) \left( \partial
_{j}\partial _{i}h\right) \right] _{1\leq i,j,k\leq n}+\psi ^{\prime }\left[
\partial _{k}\partial _{j}\partial _{i}h\right] _{1\leq i,j,k\leq n}.
\end{eqnarray*}%
We can write this more compactly using the notation of (\ref{not}) and
symmetrizing products, to obtain 
\begin{eqnarray*}
\nabla \left( \psi \circ h\right) &=&\psi ^{\prime }\nabla h, \\
\nabla ^{2}\left( \psi \circ h\right) &=&\psi ^{\prime \prime }\left( \nabla
h\right) ^{2}+\psi ^{\prime }\nabla ^{2}h, \\
\nabla ^{3}\left( \psi \circ h\right) &=&\psi ^{\prime \prime \prime }\left(
\nabla h\right) ^{3}+3\psi ^{\prime \prime }\left( \nabla ^{2}h\right)
\otimes \left( \nabla h\right) +\psi ^{\prime }\left( \nabla ^{3}h\right) .
\end{eqnarray*}%
In general we have the formula 
\begin{equation}
\nabla ^{M}\left( \psi \circ h\right) =\sum_{m=1}^{M}\left( \psi ^{\left(
m\right) }\circ h\right) \left( \sum_{\substack{ \alpha =\left( \alpha
_{1},...,\alpha _{M}\right) \in \mathbb{Z}_{+}^{M}  \\ \alpha _{1}+\alpha
_{2}+...+\alpha _{M}=m  \\ \alpha _{1}+2\alpha _{2}+...+M\alpha _{M}=M}}%
\left[ 
\begin{array}{c}
M \\ 
\alpha%
\end{array}%
\right] \left( \nabla h\right) ^{\alpha _{1}}\otimes \left( \nabla
^{2}h\right) ^{\alpha _{2}}\otimes ...\otimes \left( \nabla ^{M}h\right)
^{\alpha _{M}}\right) ,  \label{comp fla}
\end{equation}%
where $\left[ 
\begin{array}{c}
M \\ 
\alpha%
\end{array}%
\right] $ is defined for $\alpha =\left( \alpha _{1},...,\alpha _{M}\right)
\in \mathbb{Z}_{+}^{M}$ satisfying $\alpha _{1}+2\alpha _{2}+...+M\alpha
_{M}=M$. We will not need to evaluate or even estimate $\left[ 
\begin{array}{c}
M \\ 
\alpha%
\end{array}%
\right] $ for our purposes in this paper, but a recursion formula for these
coefficients in a special case can be found in \cite{MaSaUrVu}.

Second, we give the explicit dependence on bounding an $m^{th}$ derivative
of an $\omega _{s}$-monotone function, a fact which will be used later in
our sum of squares theorem.

\begin{theorem}
\label{s'^m}Let $m,n\in \mathbb{N}$ and $0<s^{\prime }<s<1$ be given. Fix a
ball $B\left( 0,a\right) \subset \mathbb{R}^{n}$ with radius $a>0$. Then
there are positive constants $\Gamma _{m,n,s,s^{\prime },a}$ such that 
\begin{equation}
\left\vert \nabla ^{m}f\left( x\right) \right\vert \leq \Gamma
_{m,n,s,s^{\prime },a}f\left( x\right) ^{\left( s^{\prime }\right) ^{m}},\ \
\ \ \ \text{\ for all }x\in B\left( 0,a\right) ,  \label{first high'}
\end{equation}%
and for all $f$ that are elliptical, flat, smooth, $\omega _{s}$-monotone,
and satisfy $\left\vert f\left( x\right) \right\vert \leq 1$ on $B\left(
0,a\right) $.
\end{theorem}

\begin{proof}
First using Lemma \ref{first high} and the fact that $f$ is $\omega _{s}$%
-monotone, we have 
\begin{eqnarray*}
\max_{t\in B\left( 0,\left\vert x\right\vert \right) }\left\vert \nabla
f\left( t\right) \right\vert &\leq &C_{k,1,n}\left\{ \frac{1}{|x|}%
\max_{t_{1},t_{2}\in B\left( \frac{x}{2},\frac{\left\vert x\right\vert }{2}%
\right) }\left\vert f\left( t_{1}\right) -f\left( t_{2}\right) \right\vert
\right. \\
&&\ \ \ \ \ \ \ \ \ \ +\left. \left( \max_{t_{1}t_{2},\in B\left( \frac{x}{2}%
,\frac{\left\vert x\right\vert }{2}\right) }\left\vert f\left( t_{1}\right)
-f\left( t_{2}\right) \right\vert \right) ^{1-\frac{1}{k}}\left( \max_{t\in
B\left( \frac{x}{2},\frac{\left\vert x\right\vert }{2}\right) }\left\vert
\nabla ^{k}f\left( t\right) \right\vert \right) ^{\frac{1}{k}}\right\} \\
&\leq &C_{k,1,n}\left\{ \frac{1}{|x|}\Gamma _{s}f\left( x\right) ^{s}+\left(
\Gamma _{s}f\left( x\right) ^{s}\right) ^{1-\frac{1}{k}}\left( \max_{t\in
B\left( \frac{x}{2},\frac{\left\vert x\right\vert }{2}\right) }\left\vert
\nabla ^{k}f\left( t\right) \right\vert \right) ^{\frac{1}{k}}\right\} \\
&\leq &C_{k,1,n}\left\{ \frac{1}{|x|}\Gamma _{s}f\left( x\right) ^{s}+M_{k}^{%
\frac{1}{k}}\left( \Gamma _{s}f\left( x\right) ^{s}\right) ^{1-\frac{1}{k}%
}\right\} ,
\end{eqnarray*}%
where $\max_{t\in B\left( 0,\left\vert x\right\vert \right) }\left\vert
\nabla ^{k}f\left( t\right) \right\vert \leq M_{k}$ for $x\in \overline{B}%
(0,a)$ follows since $f$ is smooth. Moreover, since $f$ is flat, we have $%
f\left( x\right) \leq A_{k}\left\vert x\right\vert ^{k}$ , and thus $f\left(
x\right) ^{\frac{1}{k}}\leq A_{k}^{\frac{1}{k}}\left\vert x\right\vert $,
and so for $x\in B(0,a)$, $x\neq 0$, and $0<f\left( x\right) \leq 1$, we
have 
\begin{eqnarray*}
\left\vert \nabla f\left( x\right) \right\vert &\leq &C_{k,1,n}\left\{
\Gamma _{\varepsilon }f\left( x\right) ^{s-\frac{1}{k}}\frac{f\left(
x\right) ^{\frac{1}{k}}}{|x|}+M_{k}^{\frac{1}{k}}\Gamma _{s}^{1-\frac{1}{k}%
}f\left( x\right) ^{\left( s\right) \left( 1-\frac{1}{k}\right) }\right\} \\
&\leq &B_{k,n,s}f\left( x\right) ^{s-\frac{1}{k}}+D_{k,n,\varepsilon
}f\left( x\right) ^{s\left( 1-\frac{1}{k}\right) },
\end{eqnarray*}%
where $B_{k,n,s}=C_{k,1,n}\Gamma _{\varepsilon }A_{k}^{\frac{1}{k}}$ and $%
D_{k,n,\varepsilon }=C_{k,1,n}\Gamma _{s}^{1-\frac{1}{k}}M_{k}^{\frac{1}{k}}$%
. Thus by choosing $k$ sufficiently large, we see that for every $%
0<s^{\prime }<s<1$, there is a positive constant $\Gamma _{n,s,s^{\prime
},a} $ such that 
\begin{equation}
\left\vert \nabla f\left( x\right) \right\vert \leq \Gamma _{n,s,s^{\prime
},a}f\left( x\right) ^{s^{\prime }},\ \ \ \ \ x\neq 0,\ \ \ 0<s^{\prime
}<s<1,  \label{first est high}
\end{equation}%
which proves the case $m=1$ of (\ref{first high'}). We now prove the general
case by induction on $m$. Fix $m\geq 1$ and suppose that for all $1\leq \ell
\leq m$ and all $0<s^{\prime }<s<1$ there holds 
\begin{equation}
\left\vert \nabla ^{\ell }f\left( x\right) \right\vert \leq \Gamma _{\ell
,n,s,s^{\prime }}f\left( x\right) ^{\left( s^{\prime }\right) ^{l}},\ \ \ \
\ \text{for }x\in B\left( 0,a\right) .  \label{ind-assump}
\end{equation}%
Since $f$ is a flat smooth function we have from Lemma \ref{first high} that 
\begin{eqnarray*}
&&\max_{t\in B}\left\vert \nabla ^{m+1}f\left( t\right) \right\vert \leq
C_{k,m,n}\frac{1}{|x|^{m+1}}\max_{t_{1},t_{2}\in B}\left\vert f\left(
t_{1}\right) -f\left( t_{2}\right) -\left( t_{1}-t_{2}\right) \cdot \nabla
f\left( t_{2}\right) -...-\frac{\left[ \left( t_{1}-t_{2}\right) \cdot
\nabla \right] ^{m}}{m!}f\left( t_{2}\right) \right\vert \\
&&+C_{k,m,n}\left( \max_{t_{1},t_{2}\in B}\left\vert f\left( t_{1}\right)
-f\left( t_{2}\right) -\left( t_{1}-t_{2}\right) \cdot \nabla f\left(
t_{2}\right) -...-\frac{\left[ \left( t_{1}-t_{2}\right) \cdot \nabla \right]
^{m}}{m!}f\left( t_{2}\right) \right\vert \right) ^{1-\frac{m+1}{k}} \\
&&\ \ \ \ \ \ \ \ \ \ \ \ \ \ \ \ \ \ \ \ \ \ \ \ \ \ \ \ \ \ \ \ \ \ \ \ \
\ \ \ \ \ \ \times \left( \max_{t\in B}\left\vert \nabla ^{k}f\left(
t\right) \right\vert \right) ^{\frac{m+1}{k}}.
\end{eqnarray*}%
Now fix $s^{\prime }<s$ and let $\varepsilon =(s-s^{\prime })/2$, so that $%
s^{\prime }+\varepsilon <s$. Then using (\ref{ind-assump}) with $s^{\prime
}+\varepsilon $ in place of $s^{\prime }$ and the fact that $f$ is $\omega
_{s}$-monotone, we conclude that for $B=B\left( 0,\left\vert x\right\vert
\right) $ we have 
\begin{eqnarray*}
&&\left\vert f\left( t_{1}\right) -f\left( t_{2}\right) -\left(
t_{1}-t_{2}\right) \cdot \nabla f\left( t_{2}\right) -...-\frac{\left[
\left( t_{1}-t_{2}\right) \cdot \nabla \right] ^{m}}{m!}f\left( t_{2}\right)
\right\vert \\
&&\quad \leq \left\vert f\left( t_{1}\right) -f\left( t_{2}\right)
\right\vert +\left\vert \nabla f\left( t_{2}\right) \cdot \left(
t_{1}-t_{2}\right) \right\vert +\dots +\left\vert \frac{\left[ \left(
t_{1}-t_{2}\right) \cdot \nabla \right] ^{m}}{m!}f\left( t_{2}\right)
\right\vert \\
&&\quad \leq C_{n,s}f(x)^{s}+\sum_{\ell =1}^{m}\Gamma _{\ell ,n,s,s^{\prime
}}f(t_{2})^{\left( s^{\prime }+\varepsilon \right) ^{\ell }}|x|^{\ell } \\
&&\quad \leq C_{n,s}f(x)^{s}+\sum_{\ell =1}^{m}\Gamma _{\ell ,n,s,s^{\prime
}}\left( f(x)^{s}\right) ^{\left( s^{\prime }+\varepsilon \right) ^{\ell
}}|x|^{\ell }\leq \sum_{\ell =0}^{m}\Gamma _{\ell ,n,s,s^{\prime
}}f(x)^{s\left( s^{\prime }+\varepsilon \right) ^{\ell }}|x|^{\ell }.
\end{eqnarray*}%
Thus we have 
\begin{eqnarray*}
&&\frac{1}{|x|^{m+1}}\max_{t_{1},t_{2}\in B}\left\vert f\left( t_{1}\right)
-f\left( t_{2}\right) -\left( t_{1}-t_{2}\right) \cdot \nabla f\left(
t_{2}\right) -...-\frac{\left[ \left( t_{1}-t_{2}\right) \cdot \nabla \right]
^{m}}{m!}f\left( t_{2}\right) \right\vert \\
&&\quad \leq \sum_{\ell =0}^{m}\frac{\Gamma _{\ell ,n,s,s^{\prime
}}f(x)^{s\left( s^{\prime }+\varepsilon \right) ^{\ell }}}{|x|^{m+1-\ell }}%
=\sum_{\ell =1}^{m+1}\Gamma _{\ell ,n,s,s^{\prime }}f(x)^{\left( s^{\prime
}-\varepsilon \right) ^{\ell }}\frac{f(x)^{s-(s^{\prime }+\varepsilon )}}{%
|x|^{m+1-\ell }}\leq \sum_{\ell =1}^{m+1}\Gamma _{\ell ,n,s,s^{\prime
},a}f(x)^{\left( s^{\prime }+\varepsilon \right) ^{\ell }},
\end{eqnarray*}%
where in the last inequality we used the fact that $f$ is flat, and thus $%
f(x)\leq C_{M}|x|^{M}$ for all $M>0$. Therefore we obtain 
\begin{eqnarray*}
\left\vert \nabla ^{m+1}f\left( x\right) \right\vert &\leq &\sum_{\ell
=1}^{m+1}\Gamma _{\ell ,n,s,s^{\prime },a}f(x)^{\left( s^{\prime
}+\varepsilon \right) ^{\ell }}+\left( \sum_{l=1}^{m+1}\Gamma _{\ell
,n,s,s^{\prime },a}f(x)^{\left( s^{\prime }+\varepsilon \right) ^{\ell
}}\right) ^{1-\frac{m+1}{k}}M_{k}^{\frac{m+1}{k}} \\
&\leq &\Gamma _{m,n,s,s^{\prime },a}f(x)^{\left( s^{\prime }\right) ^{m+1}},
\end{eqnarray*}%
upon taking $k$ sufficiently large so that $\left( s^{\prime }+\varepsilon
\right) ^{m+1}\left( 1-\frac{m+1}{k}\right) \geq \left( s^{\prime }\right)
^{m+1}$.
\end{proof}

We are now ready to proceed with the proof of Theorem \ref{main intro}.

\begin{proof}[Proof of Theorem \protect\ref{main intro}]
First note that $\left( \mathbf{3}\right) \mathbf{\Longrightarrow }\left( 
\mathbf{1}\right) $ is trivial and $\left( \mathbf{4}\right) \mathbf{%
\Longrightarrow }\left( \mathbf{2}\right) $ follows from Theorem \ref{s'^m}.

$\left( \mathbf{1}\right) \mathbf{\Longrightarrow }\left( \mathbf{2}\right) $%
: Since $f\left( x\right) ^{\beta }$ is smooth and nonnegative for $0<\beta
<\delta $, we have the classical inequality of Malgrange, see e.g. \cite[%
Lemme I]{Gla},%
\begin{equation*}
\left\vert \nabla \left[ f\left( x\right) ^{\beta }\right] \right\vert \leq C%
\sqrt{f\left( x\right) ^{\beta }},
\end{equation*}%
which implies that%
\begin{eqnarray*}
\beta f\left( x\right) ^{\beta -1}\left\vert \nabla f\left( x\right)
\right\vert &\leq &Cf\left( x\right) ^{\frac{\beta }{2}}, \\
\text{hence }\left\vert \nabla f\left( x\right) \right\vert ^{2} &\leq
&Cf\left( x\right) ^{2-\beta }.
\end{eqnarray*}%
Next we compute%
\begin{equation*}
\nabla ^{2}\left[ f\left( x\right) ^{\beta }\right] =\nabla \left( \beta
f^{\beta -1}\nabla f(x)\right) =\beta \left( \beta -1\right) f\left(
x\right) ^{\beta -2}\left( \nabla f(x)\right) ^{2}+\beta f\left( x\right)
^{\beta -1}\nabla ^{2}f\left( x\right) ,
\end{equation*}%
which implies that%
\begin{eqnarray*}
\beta f\left( x\right) ^{\beta -1}\left\vert \nabla ^{2}f\left( x\right)
\right\vert &\leq &Cf\left( x\right) ^{\beta -2}\left\vert \nabla f\left(
x\right) \right\vert ^{2}+C\left\vert \nabla ^{2}\left[ f\left( x\right)
^{\beta }\right] \right\vert \\
\text{hence }\left\vert \nabla ^{2}f\left( x\right) \right\vert &\leq
&f\left( x\right) ^{1-\beta }\left\{ Cf\left( x\right) ^{\beta -2}f\left(
x\right) ^{2-\beta }+C_{\beta }\right\} \\
&\leq &Cf\left( x\right) ^{1-\beta }+C_{\beta }f\left( x\right) ^{1-\beta
}=C_{\beta }f\left( x\right) ^{1-\beta },
\end{eqnarray*}%
where we have used the fact that $f\left( x\right) ^{\beta }$ is smooth,
hence $\left\vert \nabla ^{2}\left[ f\left( x\right) ^{\beta }\right]
\right\vert $ is bounded on compact subsets of $B\left( 0,a\right) $. We now
prove by induction that%
\begin{equation*}
\left\vert \nabla ^{M}f\left( x\right) \right\vert \leq \Gamma _{M,\gamma
}f\left( x\right) ^{1-M\gamma },\ \ \ \ \ \text{where }\gamma =\frac{\beta }{%
2}.
\end{equation*}%
Define the nonnegative power functions $s_{\gamma }:\left[ 0,\infty \right)
\rightarrow \left[ 0,\infty \right) $ by $s_{\gamma }\left( t\right)
=t^{\gamma }$ for $t\in \left[ 0,\infty \right) $, and note that 
\begin{equation}
s_{\gamma }^{\left( k\right) }\left( t\right) =\left( 
\begin{array}{c}
\gamma \\ 
k%
\end{array}%
\right) t^{\gamma -k},\ \ \ \ \ \text{for }k\geq 0\text{ and }t\in \left(
0,\infty \right) .  \label{root der}
\end{equation}%
Indeed, with $g\left( x\right) =f\left( x\right) ^{\gamma }=s_{\gamma }\circ
f\left( x\right) $ we have using the composition formula (\ref{comp fla}) 
\begin{eqnarray*}
\nabla ^{M}g\left( x\right) &=&\nabla ^{M}\left( s_{\gamma }\circ f\right)
\left( x\right) \\
&=&\sum_{m=1}^{M}\left( s_{\gamma }^{\left( m\right) }\circ f\right) \left(
x\right) \left( \sum_{\substack{ \alpha =\left( \alpha _{1},...,\alpha
_{M}\right) \in \mathbb{Z}_{+}^{M}  \\ \alpha _{1}+\alpha _{2}+...+\alpha
_{M}=m  \\ \alpha _{1}+2\alpha _{2}+...+M\alpha _{M}=M}}\left[ 
\begin{array}{c}
M \\ 
\alpha%
\end{array}%
\right] \left( \nabla f\left( x\right) \right) ^{\alpha _{1}}...\left(
\nabla ^{M}f\left( x\right) \right) ^{\alpha _{M}}\right) ,
\end{eqnarray*}%
and since $\alpha _{M}>0$ \ implies $\alpha _{M}=m=1$, we obtain that%
\begin{eqnarray*}
&&\left[ 
\begin{array}{c}
M \\ 
\alpha%
\end{array}%
\right] \left( s_{\gamma }^{\left( 1\right) }\circ f\right) \left( x\right)
\nabla ^{M}f\left( x\right) \\
&=&\nabla ^{M}g\left( x\right) -\sum_{m=2}^{M}\left( s_{\gamma }^{\left(
m\right) }\circ f\right) \left( x\right) \left( \sum_{\substack{ \alpha
=\left( \alpha _{1},...,\alpha _{M}\right) \in \mathbb{Z}_{+}^{M}  \\ \alpha
_{1}+\alpha _{2}+...+\alpha _{M}=m  \\ \alpha _{1}+2\alpha _{2}+...+M\alpha
_{M}=M}}\left[ 
\begin{array}{c}
M \\ 
\alpha%
\end{array}%
\right] \left( \nabla f\left( x\right) \right) ^{\alpha _{1}}...\left(
\nabla ^{M-1}f\left( x\right) \right) ^{\alpha _{M-1}}\right) ,
\end{eqnarray*}%
hence using the inductive assumption and the fact that $g=f^{\gamma }$ is
smooth, 
\begin{eqnarray*}
&&\left\vert \left( s_{\gamma }^{\left( 1\right) }\circ f\right) \left(
x\right) \nabla ^{M}f\left( x\right) \right\vert \\
&\leq &C+C_{\gamma ,M}\sum_{m=2}^{M}\left( 
\begin{array}{c}
\gamma \\ 
m%
\end{array}%
\right) f\left( x\right) ^{\gamma -m}\left( \sum_{\substack{ \alpha =\left(
\alpha _{1},...,\alpha _{M-1}\right) \in \mathbb{Z}_{+}^{M-1}  \\ \alpha
_{1}+\alpha _{2}+...+\alpha _{M-1}=m  \\ \alpha _{1}+2\alpha _{2}+...+\left(
M-1\right) \alpha _{M-1}=M}}\left[ 
\begin{array}{c}
M \\ 
\alpha%
\end{array}%
\right] f\left( x\right) ^{\left( 1-\gamma \right) \alpha _{1}}...f\left(
x\right) ^{\left( 1-\left( M-1\right) \gamma \right) \alpha _{M-1}}\right) ,
\end{eqnarray*}%
where $C$ above is a bound for $\left\vert \nabla ^{M}g\left( x\right)
\right\vert $, and finally that%
\begin{equation*}
\left\vert \nabla ^{M}f\left( x\right) \right\vert \leq Cf\left( x\right)
^{1-\gamma }\left\{ C+C_{\gamma ,M}\sum_{m=2}^{M}f\left( x\right) ^{\gamma
-m}f\left( x\right) ^{m-M\gamma }\right\} \leq C_{\gamma ,M}^{\prime
}f\left( x\right) ^{1-\gamma M}.
\end{equation*}

$\left( \mathbf{2}\right) \mathbf{\Longrightarrow }\left( \mathbf{3}\right) $%
: Again set $g\left( x\right) =f\left( x\right) ^{\gamma }=\left( s_{\gamma
}\circ f\right) \left( x\right) $, and as before we have 
\begin{eqnarray*}
\nabla^{M}g\left( x\right) &=&\nabla^{M}\left( s_{\gamma }\circ f\right)
\left( x\right) \\
&=&\sum_{m=1}^{M}\left( s_{\gamma }^{\left( m\right) }\circ f\right) \left(
x\right) \left( \sum_{\substack{ \alpha =\left( \alpha _{1},...,\alpha
_{M}\right) \in \mathbb{Z}_{+}^{M}  \\ \alpha _{1}+\alpha _{2}+...+\alpha
_{M}=m  \\ \alpha _{1}+2\alpha _{2}+...+M\alpha _{M}=M}}\left[ 
\begin{array}{c}
M \\ 
\alpha%
\end{array}
\right] \left( \nabla f\left( x\right) \right) ^{\alpha _{1}}...\left(
\nabla^{M}f\left( x\right) \right) ^{\alpha _{M}}\right) ,
\end{eqnarray*}%
Now we use (\ref{root der}), i.e. 
\begin{equation*}
\left\vert \left( s_{\gamma }^{\left( m\right) }\circ f\right) \left(
x\right) \right\vert =\left\vert s_{\gamma }^{\left( m\right) }\left(
f\left( x\right) \right) \right\vert =\left\vert \left( 
\begin{array}{c}
\gamma \\ 
m%
\end{array}
\right) \right\vert f\left( x\right) ^{\gamma -\ m},
\end{equation*}%
and condition $\left( \mathbf{2}\right) $ with $s=1-\varepsilon$, i.e. $%
\left\vert\nabla^{k} f\left( x\right) \right\vert \leq \Gamma
_{k,\varepsilon }f\left( x\right) ^{1-\varepsilon }$, to obtain%
\begin{eqnarray*}
\left\vert \nabla^{M}g\left( x\right) \right\vert &\leq
&C\sum_{m=1}^{M}f\left( x\right) ^{\gamma -m}\left( \sum_{\substack{ \alpha
=\left( \alpha _{1},...,\alpha _{M}\right) \in \mathbb{Z}_{+}^{M}  \\ \alpha
_{1}+\alpha _{2}+...+\alpha _{M}=m  \\ \alpha _{1}+2\alpha _{2}+...+M\alpha
_{M}=M}}\left[ 
\begin{array}{c}
M \\ 
\alpha%
\end{array}
\right] \left( f\left( x\right) ^{1-\varepsilon }\right) ^{\alpha
_{1}}...\left( f\left( x\right) ^{1-\varepsilon }\right) ^{\alpha
_{M}}\right) \\
&\leq &C\sum_{m=1}^{M}f\left( x\right) ^{\gamma -m}\left( f\left( x\right)
^{m\left( 1-\varepsilon \right) }\right) =C\sum_{m=1}^{M}f\left( x\right)
^{\gamma -m\varepsilon }\leq CMf\left( x\right) ^{\gamma -M\varepsilon }.
\end{eqnarray*}%
If we choose $\varepsilon <\frac{\gamma }{M}$, then we see that $%
\nabla^{M}g\left( x\right)$ is a flat function for each $M\geq 0$, and it
follows that $g$ is a flat smooth function.

$\left( \mathbf{3}\right) \mathbf{\nRightarrow }\left( \mathbf{4}\right) $:
Let $g$ be any elliptical flat smooth function that fails to be nearly
monotone, or even just fails the inequality $g\left( t\right) \leq 4g\left(
x\right) $ for some $0\leq t<x<a$. Then if $f\left( x\right) =e^{-\frac{1}{%
g\left( x\right) }}$, the functions $f\left( x\right) ^{\alpha }$ are smooth
for all $\alpha >0$, but $f$ is clearly not nearly monotone since in
particular, $f$ fails the inequality $f\left( t\right) \leq C_{s}f\left(
x\right) ^{s},\ \ \ 0\leq t<x<a$ for every $\frac{1}{2}\leq s<1$. Indeed, if
this inequality holds for some $s\geq \frac{1}{2}$, then 
\begin{eqnarray*}
f\left( t\right) &\leq &C_{s}f\left( x\right) ^{s},\ \ \ \ \ 0\leq t<x<a, \\
&\Longrightarrow &\ln \frac{1}{f\left( t\right) }\geq \ln \frac{1}{C_{s}}%
+s\ln \frac{1}{f\left( x\right) },\ \ \ \ \ 0\leq t<x<a, \\
&\Longrightarrow &\frac{1}{g\left( t\right) }\geq -\ln C_{s}+\frac{s}{%
g\left( x\right) },\ \ \ \ \ 0\leq t<x<a, \\
&\Longrightarrow &g\left( t\right) \leq \frac{1}{\frac{s}{g\left( x\right) }%
-\ln C_{s}}=\frac{g\left( x\right) }{s-\left( \ln C_{s}\right) g\left(
x\right) },\ \ \ \ \ 0\leq t<x<a,
\end{eqnarray*}%
which shows that for $x$ small enough, namely $g\left( x\right) <\frac{s}{%
2\ln C_{s}}$, we have $g\left( t\right) \leq \frac{2}{s}g\left( x\right)
\leq 4g\left( x\right) $, contradicting our assumption on $g$.

Finally, we give a modification of Glaseser's example in \cite{Gla} that
shows that for any $0\leq s<1$, there is an $\omega _{s}$-monotone function $%
f$ such that $f^{\alpha }$ is not smooth if $0<\alpha \leq \frac{1}{s}-1$.
Suppose $\varphi $ is an elliptical flat smooth function on $\left(
-1,1\right) $ that is decreasing on $\left( -1,0\right] $ and increasing on $%
\left[ 0,1\right) $. Suppose further that $\varphi $ is constant in a
neighbourhood of $\frac{1}{n}$ for each $n\in \mathbb{N}$,\ say in $\left( 
\frac{1}{n}-\varepsilon _{n},\frac{1}{n}+\varepsilon _{n}\right) $. See \cite%
[page 206]{Gla} for a construction of such a function. Let $0<\gamma <1$ and
define%
\begin{equation*}
f_{\gamma }\left( x\right) \equiv \varphi \left( x\right) ^{\frac{1}{\gamma }%
-1}\left( \sin ^{2}\frac{\pi }{x}+\varphi \left( x\right) \right) ,\ \ \ \ \ 
\text{for }-1<x<1.
\end{equation*}%
Then $f_{\gamma }$ is a flat smooth function vanishing only at $0$. Indeed,
Theorem \ref{main intro} shows in particular that $\varphi \left( x\right) ^{%
\frac{1}{\gamma }-1}$ is a smooth flat function for $0<\gamma <1$, and then
the smoothness of $f_{\gamma }$ at\ the\ origin follows\ easily from the
inequalities 
\begin{equation*}
\left\vert \frac{d^{n}}{dx^{n}}\sin ^{2}\frac{\pi }{x}\right\vert \leq
C_{n}\left\vert x\right\vert ^{-2n}.
\end{equation*}

The assumption that $\varphi $ is positive away from the origin shows that $%
f_{\gamma }$ is as well. Following Glaeser's argument, we now show that $%
g_{\gamma }\left( x\right) =\left( f_{\gamma }\left( x\right) \right)
^{\gamma }$ doesn't have a bounded second derivative in any neighbourhood of
the origin. Indeed, if $x=\frac{1}{n}+y$ where $y\in \left( -\varepsilon
_{n},\varepsilon _{n}\right) $, then%
\begin{eqnarray*}
\sin ^{2}\frac{\pi }{x} &=&\sin ^{2}\frac{\pi }{\frac{1}{n}+y}=\sin
^{2}\left( n\pi -n^{2}\pi \frac{y}{1+ny}\right) \\
&=&\sin ^{2}\left( n^{2}\pi \frac{y}{1+ny}\right) =\left( n^{2}\pi y\right)
^{2}+o\left( y^{2}\right) ,
\end{eqnarray*}%
and so%
\begin{eqnarray*}
f_{\gamma }\left( \frac{1}{n}+y\right) &=&\varphi \left( \frac{1}{n}
+y\right) ^{\frac{1}{\gamma }-1}\left( \sin ^{2}\frac{\pi }{\frac{1}{n}+y}
+\varphi \left( \frac{1}{n}+y\right) \right) \\
&=&\varphi \left( \frac{1}{n}\right) ^{\frac{1}{\gamma }-1}\left( n^{4}\pi
^{2}y^{2}+o\left( y^{2}\right) +\varphi \left( \frac{1}{n}\right) \right) \\
&=&\varphi \left( \frac{1}{n}\right) ^{\frac{1}{\gamma }}\left( \frac{
n^{4}\pi ^{2}y^{2}}{\varphi \left( \frac{1}{n}\right) }+1+\frac{o\left(
y^{2}\right) }{\varphi \left( \frac{1}{n}\right) }\right) ,
\end{eqnarray*}%
implies that for $y$ sufficiently small depending on $n$, we have%
\begin{eqnarray*}
g_{\gamma }\left( \frac{1}{n}+y\right) &=&\left( f_{\gamma }\left( \frac{1}{%
n }+y\right) \right) ^{\gamma }=\varphi \left( \frac{1}{n}\right) \left( 1+ 
\frac{n^{4}\pi ^{2}y^{2}}{\varphi \left( \frac{1}{n}\right) }+\frac{o\left(
y^{2}\right) }{\varphi \left( \frac{1}{n}\right) }\right) ^{\gamma } \\
&=&\varphi \left( \frac{1}{n}\right) \left( 1+\gamma \frac{n^{4}\pi
^{2}y^{2} }{\varphi \left( \frac{1}{n}\right) }+\frac{o\left( y^{2}\right) }{%
\varphi \left( \frac{1}{n}\right) }\right) =\varphi \left( \frac{1}{n}%
\right) +\gamma n^{4}\pi ^{2}y^{2}+o\left( y^{2}\right) ,
\end{eqnarray*}%
which in turn shows that $g_{\gamma }^{\prime \prime }\left( \frac{1}{n}%
\right) =2\gamma n^{4}\pi ^{2}$.

On the other hand, if $\frac{\pi }{t}=n\pi +\frac{\pi }{2}$ and $\frac{\pi }{%
x}=n\pi $, then $t<x$, $\sin ^{2}\frac{\pi }{t}=1$, $\sin ^{2}\frac{\pi }{x}%
=0$ and 
\begin{eqnarray*}
\frac{f_{\gamma }\left( t\right) }{f_{\gamma }\left( x\right)
^{1-\varepsilon }} &=&\frac{\varphi \left( t\right) ^{\frac{1}{\gamma }%
}\left( \sin ^{2}\frac{\pi }{t}+\varphi \left( t\right) \right) }{\varphi
\left( x\right) ^{\frac{1}{\gamma }\left( 1-\varepsilon \right) }\left( \sin
^{2}\frac{\pi }{x}+\varphi \left( x\right) \right) ^{1-\varepsilon }} \\
&=&\frac{\varphi \left( t\right) ^{\frac{1}{\gamma }}\left( 1+\varphi \left(
t\right) \right) }{\varphi \left( x\right) ^{\frac{1}{\gamma }\left(
1-\varepsilon \right) +1-\varepsilon }}>\frac{\varphi \left( t\right) ^{%
\frac{1}{\gamma }}}{\varphi \left( x\right) ^{\left( \frac{1}{\gamma }%
+1\right) \left( 1-\varepsilon \right) }}=\varphi \left( t\right)
^{\varepsilon \left( \frac{1}{\gamma }+1\right) -1}
\end{eqnarray*}%
is bounded as $n\rightarrow \infty $ if $\varepsilon \left( \frac{1}{\gamma }%
+1\right) \geq 1$, i.e. $\varepsilon \geq \frac{\gamma }{1+\gamma }$. Since
these pairs $\left( t,x\right) $ are the worst choices, it follows easily
that%
\begin{equation*}
\sup_{0<t<x<1}\frac{f_{\gamma }\left( t\right) }{f_{\gamma }\left( x\right)
^{s}}<\infty \Longleftrightarrow s\leq \frac{1}{1+\gamma }.
\end{equation*}%
Thus for $0<s=\frac{1}{1+\gamma }<1$, this gives an example of an elliptical
flat smooth function $f_{\gamma }$ that satisfies $\omega _{s}$%
-monotonicity, but the power function $\left( f_{\gamma }\left( x\right)
\right) ^{\gamma }$ is not smooth.

This completes the proof of Theorem \ref{main intro}.
\end{proof}

\begin{remark}
\label{main remark}Let $M\geq 0$ and $0<s<1$. If $\left\vert \nabla
^{k}f\left( x\right) \right\vert \leq \Gamma _{k,s}f\left( x\right) ^{s}$
holds for $0\leq k\leq M$, then $f^{\gamma }\in C^{M-1,1}$ for all $\gamma
\geq M\left( 1-s\right) $. For this, see the end of the proof of $\left( 
\mathbf{2}\right) \mathbf{\Longrightarrow }\left( \mathbf{3}\right) $ above.
In particular $\sqrt{f}\in C^{1,1}$ if $s\geq \frac{3}{4}$ and $\sqrt{f}\in
C^{2,1}$ if $s\geq \frac{5}{6}$. When $s>\frac{3}{4}$, we show in Theorem %
\ref{1/4 thm} just below that $\sqrt{f}\in C^{2,\delta }$ for some $\delta
>0 $. Finally, we see that if we assume $f$ is $\omega _{s}$-monotone for
some $s>1-\frac{\gamma }{M}$, then we conclude that $f^{\gamma }\in
C^{M-1,1} $.
\end{remark}

But we can do better than the previous remark indicates, as the next and
last theorem in this section shows.

\begin{theorem}
\label{1/4 thm}Let $M\geq 2$. Suppose that $f$ is elliptical, flat, smooth
and $\omega _{s}$-monotone on $\mathbb{R}^{d}$ for some $1-\frac{1}{2M}%
<s\leq 1$. Then there is $\delta >0$ and $g\in C^{M,\delta }\left( \mathbb{R}%
^{d}\right) $ such that $f=g^{2}$.
\end{theorem}

For the proof, we follow Bony \cite[Subsection 5.1]{Bon}, and define for a
multiindex $\alpha $ and $0<\delta <1$,%
\begin{equation}
\left[ h\right] _{\alpha ,\delta }\left( x\right) \equiv
\limsup_{y,z\rightarrow x}\frac{\left\vert D^{\alpha }h\left( y\right)
-D^{\alpha }h\left( z\right) \right\vert }{\left\vert y-z\right\vert
^{\delta }}.  \label{def mod D}
\end{equation}%
%
%
%
%
%
%
There is a subproduct rule,%
\begin{equation}
\left[ fg\right] _{\alpha ,\delta }\left( x\right) \lesssim \sum_{\beta \leq
\alpha }\left[ f\right] _{\alpha -\beta ,\delta }\left( x\right) \
\left\vert D^{\beta }g\left( x\right) \right\vert +\sum_{\beta \leq \alpha
}\left\vert D^{\alpha -\beta }f\left( x\right) \right\vert \ \left[ g\right]
_{\beta ,\delta }\left( x\right) ,  \label{subproduct}
\end{equation}%
which follows using\ the product rule $D^{\alpha }\left( fg\right)
=\sum_{\beta \leq \alpha }\left( 
\begin{array}{c}
\alpha \\ 
\beta%
\end{array}%
\right) \left( D^{\alpha -\beta }f\right) \left( D^{\beta }g\right) $ and
the decomposition%
\begin{eqnarray*}
&&\left( D^{\alpha -\beta }f\right) \left( y\right) \left( D^{\beta
}g\right) \left( y\right) -\left( D^{\alpha -\beta }f\right) \left( z\right)
\left( D^{\beta }g\right) \left( z\right) \\
&=&\left[ \left( D^{\alpha -\beta }f\right) \left( y\right) -\left(
D^{\alpha -\beta }f\right) \left( z\right) \right] \ \left( D^{\beta
}g\right) \left( y\right) +\left( D^{\alpha -\beta }f\right) \left( z\right)
\ \left[ \left( D^{\beta }g\right) \left( y\right) -\left( D^{\beta
}g\right) \left( z\right) \right] ,
\end{eqnarray*}%
after then dividing by $\left\vert y-z\right\vert ^{\delta }$ and taking $%
\limsup_{y,z\rightarrow x}$ inside the sum.

To derive a subchain rule we start by considering the case of $\alpha=%
\mathbf{e}_{1}=(1,0,\dots, 0)$ and $\alpha=\mathbf{e}_{1}+\mathbf{e}%
_{2}=(1,1,0,\dots,0)$. We have $\partial _{1}\left( \psi \left( h\left(
x\right) \right) \right) =\psi ^{\prime }\left( h\left( x\right) \right) \
\partial _{1}h\left( x\right) $, and therefore 
\begin{eqnarray*}
\left[ \psi \circ h\right] _{\mathbf{e}_{1},\delta }\left( x\right) &\equiv
&\limsup_{y,z\rightarrow x}\frac{\left\vert \partial _{1}\psi \circ h\left(
y\right) -\partial _{1}\psi \circ h\left( z\right) \right\vert }{\left\vert
y-z\right\vert ^{\delta }} \\
&=&\limsup_{y,z\rightarrow x}\frac{\left\vert \psi ^{\prime }\left( h\left(
y\right) \right) \ \partial _{1}h\left( y\right) -\psi ^{\prime }\left(
h\left( z\right) \right) \ \partial _{1}h\left( z\right) \right\vert }{
\left\vert y-z\right\vert ^{\delta }} \\
&\leq &\limsup_{y,z\rightarrow x}\frac{\left\vert \psi ^{\prime }\left(
h\left( y\right) \right) \ \left[ \partial _{1}h\left( y\right) -\ \partial
_{1}h\left( z\right) \right] \right\vert }{\left\vert y-z\right\vert
^{\delta }}+\limsup_{y,z\rightarrow x}\frac{\left\vert \left[ \psi ^{\prime
}\left( h\left( y\right) \right) \ -\psi ^{\prime }\left( h\left( z\right)
\right) \right] \ \partial _{1}h\left( z\right) \right\vert }{\left\vert
y-z\right\vert ^{\delta }} \\
&=&\left\vert \psi ^{\prime }\circ h\left( x\right) \right\vert \ \left[ h %
\right] _{\mathbf{e}_{1},\delta }\left( x\right) +\left[ \psi ^{\prime
}\circ h\right] _{\mathbf{0},\delta }\left( x\right) \left\vert \ \partial
_{1}h\left( x\right) \right\vert .
\end{eqnarray*}

Next we have%
\begin{equation*}
\partial _{1}\partial _{2}\left( \psi \circ h\right) =\psi ^{\prime \prime
}\left( \partial _{1}h\right) \left( \partial _{2}h\right) +\psi ^{\prime
}\partial _{2}\partial _{1}h
\end{equation*}%
and%
\begin{eqnarray*}
&&\left[ \psi \circ h\right] _{\mathbf{e}_{1}+\mathbf{e}_{2},\delta }\left(
x\right) \equiv \limsup_{y,z\rightarrow x}\frac{\left\vert \partial
_{1}\partial _{2}\left( \psi \circ h\right) \left( y\right) -\partial
_{1}\partial _{2}\left( \psi \circ h\right) \left( z\right) \right\vert }{
\left\vert y-z\right\vert ^{\delta }} \\
&\leq &\limsup_{y,z\rightarrow x}\frac{\left\vert \left\{ \left( \psi
^{\prime \prime }\circ h\right) \left( \partial _{1}h\right) \left( \partial
_{2}h\right) \right\} \left( y\right) -\left\{ \left( \psi ^{\prime \prime
}\circ h\right) \left( \partial _{1}h\right) \left( \partial _{2}h\right)
\right\} \left( z\right) \right\vert }{\left\vert y-z\right\vert ^{\delta }}
\\
&&+\limsup_{y,z\rightarrow x}\frac{\left\vert \left\{ \left( \psi ^{\prime
}\circ h\right) \partial _{2}\partial _{1}h\right\} \left( y\right) -\left\{
\left( \psi ^{\prime }\circ h\right) \partial _{2}\partial _{1}h\right\}
\left( z\right) \right\vert }{\left\vert y-z\right\vert ^{\delta }} \\
&\leq &\left[ \psi ^{\prime \prime }\circ h\right] _{\mathbf{0},\delta
}\left( x\right) \left\vert \ \partial _{1}h\left( x\right) \partial
_{2}h\left( x\right) \right\vert +\left\vert \psi ^{\prime \prime}\circ
h\left( x\right) \right\vert \ \left(\left[ h\right] _{\mathbf{e}_{1},\delta
}\left( x\right) \partial_{2}h\left( x\right) +\left[ h\right] _{\mathbf{e}
_{2},\delta }\left( x\right)\partial_{1}h\left( x\right) \right) \\
&&+\left[ \psi ^{\prime }\circ h\right] _{\mathbf{0},\delta }\left( x\right)
\left\vert \ \partial _{1}h\left( x\right) \partial _{2}h\left( x\right)
\right\vert+\left\vert \psi ^{\prime }\circ h\left( x\right) \right\vert %
\left[ h\right] _{\mathbf{e}_{1}+\mathbf{e}_{2},\delta }\left( x\right).
\end{eqnarray*}

Generalizing to $\left[ \psi \circ h\right] _{\alpha ,\delta }$ with $%
|\alpha |=M$ one obtains 
\begin{align*}
\left[ \psi \circ h\right] _{\alpha ,\delta }(x)& \lesssim
\sum_{m=1}^{M}\left( \left[ \psi ^{(m)}\circ h\right] _{0,\delta }(x)\sum 
_{\substack{ 0<\beta _{i}\leq \alpha  \\ |\beta _{1}|+\dots +|\beta _{m}|=M}}%
D^{\beta _{1}}h\cdot \dots \cdot D^{\beta _{m}}h\right) \\
& \quad +\sum_{m=1}^{M}\left( \left\vert \psi ^{(m)}\circ h(x)\right\vert
\sum_{\substack{ 0<\beta _{i}\leq \alpha  \\ |\beta _{1}|+\dots +|\beta
_{m}|=M}}\sum_{j=1}^{m}[h]_{\beta _{j},\delta }D^{\beta _{1}}h\cdot \dots
\cdot D^{\beta _{j-1}}hD^{\beta _{j+1}}h\cdot \dots \cdot D^{\beta
_{m}}h\right) .
\end{align*}

Indeed, $m$ indicates how many factors we will have in the product of
derivatives of $h$; each $\beta _{i}$ is a multiindex, which is nonzero and
does not exceed $\alpha $; the total number of derivatives we take is $%
|\beta _{1}|+\dots +|\beta _{m}|=M=|\alpha|$.

In the first line above we will replace $\left[ \psi ^{(m)}\circ h\right]
_{0,\delta }\left( x\right) $ with 
\begin{eqnarray*}
\left[ \psi ^{(m)}\circ h\right] _{0,\delta }\left( x\right)
&=&\limsup_{y,z\rightarrow x}\frac{\left\vert \psi ^{(m)}\circ h\left(
y\right) -\psi ^{(m)}\circ h\left( z\right) \right\vert }{\left\vert
y-z\right\vert ^{\delta }} \\
&=&\limsup_{y,z\rightarrow x}\frac{\left\vert \psi ^{(m)}\left( h\left(
y\right) \right) -\psi ^{(m)}\left( h\left( z\right) \right) \right\vert }{
\left\vert h\left( y\right) -h\left( z\right) \right\vert }\frac{\left\vert
h\left( y\right) -h\left( z\right) \right\vert }{\left\vert y-z\right\vert
^{\delta }} \\
&=&\psi ^{(m+1)}\left( h\left( x\right) \right) \ \left[ h\right] _{0,\delta
}\left( x\right) ,
\end{eqnarray*}%
to obtain 
\begin{align*}
&\left[ \psi \circ h\right] _{\alpha ,\delta }(x) \lesssim
\sum_{m=1}^{M}\left( \psi ^{(m+1)}\left( h\left( x\right) \right) \ \left[ h%
\right] _{0,\delta }\left( x\right) \sum_{\substack{ 0<\beta _{i}\leq \alpha 
\\ |\beta _{1}|+\dots +|\beta _{m}|=M}}D^{\beta _{1}}h\cdot \dots \cdot
D^{\beta _{m}}h\right) \\
& \quad +\sum_{m=1}^{M}\left( \left\vert \psi ^{(m)}\circ h(x)\right\vert
\sum_{\substack{ 0<\beta _{i}\leq \alpha  \\ |\beta _{1}|+\dots +|\beta
_{m}|=M}}\sum_{j=1}^{m}[h]_{\beta _{j},\delta }D^{\beta _{1}}h\cdot \dots
\cdot D^{\beta _{j-1}}hD^{\beta _{j+1}}h\cdot \dots \cdot D^{\beta
_{m}}h\right) .
\end{align*}

\begin{proof}[Proof of Theorem \protect\ref{1/4 thm}]
In the special case $\psi(t)=s_{1/2}(t)=t^{1/2}$ and $h=f$ we have 
\begin{equation*}
|D^{k}\psi \left( t\right)|\leq C_k t^{1/2-k}\text{ and }\left[ \psi \right]
_{k,\delta }\left( t\right) \leq C_k t^{1/2 -k-\delta },
\end{equation*}%
and therefore 
\begin{align*}
&\left[ \sqrt{f}\right] _{\alpha ,\delta }(x) \lesssim \sum_{m=1}^{M}\left(
f(x)^{1/2-m-1}\ f\left( x\right) ^{\left( s^{\prime }\right) ^{\delta }}
\sum _{\substack{ 0<\beta _{i}\leq \alpha  \\ |\beta _{1}|+\dots +|\beta
_{m}|=M}}f\left( x\right) ^{\left( s^{\prime}\right) ^{|\beta_1| }} \cdot
\dots \cdot f\left( x\right) ^{\left( s^{\prime }\right) ^{|\beta_m| }}
\right) \\
& \quad +\sum_{m=1}^{M}\left( f(x)^{1/2-m} \sum_{\substack{ 0<\beta _{i}\leq
\alpha  \\ |\beta _{1}|+\dots +|\beta _{m}|=M}}\sum_{j=1}^{m}f\left(
x\right) ^{\left( s^{\prime }\right) ^{|\beta_1| }} \cdot \dots\cdot f\left(
x\right) ^{\left( s^{\prime }\right) ^{|\beta_j|+\delta }}\cdot\dots \cdot
f\left( x\right) ^{\left( s^{\prime }\right) ^{|\beta_m| }} \right) .
\end{align*}

We now combine this inequality with the inequalities from (\ref{first high'}%
) and their analogues for $\left\vert D\right\vert $, namely%
\begin{equation*}
\left\vert D^{\ell }f\left( x\right) \right\vert \lesssim f\left( x\right)
^{\left( s^{\prime }\right) ^{\ell }}\text{ and }\left[ f\right] _{\ell
,\delta }\left( x\right) \lesssim f\left( x\right) ^{\left( s^{\prime
}\right) ^{\ell +\delta }},
\end{equation*}%
to see that $\left[ \sqrt{f}\right] _{\alpha ,\delta }(x) \lesssim 1$ for a
sufficiently small $\delta >0$ when $s>s^{\prime }>1-\frac{1}{2M}$. Let $%
\varepsilon^{\prime}=1-s^{\prime}$ so $\varepsilon^{\prime}=1-s^{\prime
}\in\left(0,\frac{1}{2M}\right)$. We use the estimate 
\begin{equation*}
\left(s^{\prime}\right)^{k}=\left(1-\varepsilon^{\prime}\right)^{k}\geq
1-k\varepsilon^{\prime}
\end{equation*}
to obtain 
\begin{equation*}
\left( s^{\prime}\right) ^{|\beta_1| }+\dots+\left( s^{\prime}\right)
^{|\beta_m| }\geq m-M\varepsilon^{\prime},\ \ \left( s^{\prime}\right)
^{|\beta_1| }+\dots+\left( s^{\prime}\right) ^{|\beta_j|+\delta
}+\dots+\left( s^{\prime}\right) ^{|\beta_m| }\geq
m-M\varepsilon^{\prime}-\delta\varepsilon^{\prime}
\end{equation*}
since $|\beta _{1}|+\dots +|\beta_{m}|=M$. This gives 
\begin{align*}
\left[ \sqrt{f}\right] _{\alpha ,\delta }(x) &\lesssim \sum_{m=1}^{M}\left(
f(x)^{1/2-m-1+1-\delta\varepsilon^{\prime}}\
f(x)^{m-M\varepsilon^{\prime}}\right) \quad +\sum_{m=1}^{M}\left(
f(x)^{1/2-m} \
f(x)^{m-M\varepsilon^{\prime}-\delta\varepsilon^{\prime}}\right) \\
&\lesssim f(x)^{1/2-M\varepsilon^{\prime}-\delta\varepsilon^{\prime}},
\end{align*}%
which is bounded if $\delta >0$ is chosen sufficiently small since $%
\varepsilon ^{\prime }<\frac{1}{2M}$. This completes the proof that $\sqrt{f}%
\in C^{M,\delta }$.
\end{proof}

\section{Sum of squares via Bony's H\"{o}lder adaptation of Fefferman-Phong}

Here we will follow Tataru's adaptation of the Fefferman-Phong argument,
incorporating Bony's H\"{o}lder modification, that uses the implicit
function theorem and Lemma \ref{first local} below on controlling odd
derivatives by even derivatives, plus a bit more. But we begin here by
stating and proving the implicit function theorem in the form we will use
it, and then giving the control of odd derivatives by even derivatives for
nonegative functions.

\begin{theorem}
\label{IFT}Let $H:\mathbb{R}^{n}=\mathbb{R}^{n-1}\times \mathbb{R}%
\rightarrow \mathbb{R}$ be $C^{1}$ and let $y=\left( y^{\prime
},y_{n}\right) \in \mathbb{R}^{n-1}\times \mathbb{R}$ satisfy%
\begin{equation*}
H\left( y^{\prime },y_{n}\right) =0\text{ and }\frac{\partial H}{\partial
x_{n}}\left( y^{\prime },y_{n}\right) \not=0.
\end{equation*}

\begin{enumerate}
\item Then there is a ball $U=B\left( y^{\prime },r\right) \subset \mathbb{R}%
^{n-1}$ and an interval $V=\left( y_{n}-r,y_{n}+r\right) $ such that there
is a unique function $h:U\rightarrow V$ so that $z=h\left( x^{\prime
}\right) $ satisfies 
\begin{equation*}
H\left( x^{\prime },h\left( x^{\prime }\right) \right) =0,\ \ \ \ \
x^{\prime }\in U.
\end{equation*}

\item Moreover $h$ is continuously differentiable and%
\begin{eqnarray*}
Dh\left( x^{\prime }\right) &=&-\frac{1}{\frac{\partial H}{\partial x_{n}}%
\left( x^{\prime },h\left( x^{\prime }\right) \right) }\left( D_{x^{\prime
}}H\right) \left( x^{\prime },h\left( x^{\prime }\right) \right) , \\
\text{i.e. }\frac{\partial h}{\partial x_{i}} &=&-\frac{\frac{\partial H}{%
\partial x_{i}}\left( x^{\prime },h\left( x^{\prime }\right) \right) }{\frac{%
\partial H}{\partial x_{n}}\left( x^{\prime },h\left( x^{\prime }\right)
\right) },\ \ \ \ \ \text{for }1\leq i\leq n-1.
\end{eqnarray*}

\item If in addition $H$ is $C^{2}$, then $h$ is also $C^{2}$ and%
\begin{equation*}
\frac{\partial ^{2}h}{\partial x_{i}\partial x_{j}}\left( x^{\prime }\right)
=-\frac{\frac{\partial ^{2}H}{\partial x_{i}\partial x_{j}}}{\frac{\partial H%
}{\partial x_{n}}}+\frac{\frac{\partial H}{\partial x_{j}}\frac{\partial
^{2}H}{\partial x_{i}\partial x_{n}}+\frac{\partial H}{\partial x_{i}}\frac{%
\partial ^{2}H}{\partial x_{j}\partial x_{n}}}{\left( \frac{\partial H}{%
\partial x_{n}}\right) ^{2}}-\frac{\frac{\partial H}{\partial x_{i}}\frac{%
\partial H}{\partial x_{j}}\frac{\partial ^{2}H}{\partial x_{n}^{2}}}{\left( 
\frac{\partial H}{\partial x_{n}}\right) ^{3}},
\end{equation*}%
where $H$ and its partial derivatives are evaluated at $\left( x^{\prime
},h\left( x^{\prime }\right) \right) $ for $x^{\prime }\in U$.

\item If in addition $H$ is $C^{m}$ for some $m\in \mathbb{N}$, then $h$ is
also $C^{m}$ and there is a formula for the $m^{th}$ order partial
derivatives of $h$ having the following form for $\alpha \in \mathbb{Z}%
_{+}^{n-1}$ with $\left\vert \alpha \right\vert =m$,%
\begin{equation*}
	\frac{\partial ^{m}h}{\partial x^{\alpha }}\left( x^{\prime }\right)
	=\sum_{\ell =0}^{m}\frac{(-1)^{\ell+1}\ell\,!}{\left( \frac{\partial H}{\partial x_{n}}\right)
		^{\ell +1}}\sum_{\beta \in \mathbb{Z}_{+}^{n-1}:\left\vert \beta \right\vert
		=\ell }\frac{\partial ^{m}H}{\partial ^{\alpha -\beta }x\ \partial
		x_{n}^{\ell }}\dprod\limits_{\substack{ \gamma _{j}\in \mathbb{Z}_{+}^{n-1} 
			\\ \gamma _{1}+...+\gamma _{k}=\beta }}\frac{\partial ^{\left\vert \gamma
			_{j}\right\vert }H}{\partial x^{\gamma _{j}}}.
\end{equation*}
\end{enumerate}
\end{theorem}

\begin{proof}
Parts (1) and (2) are the classical implicit function theorem. For part (3),
if $H$ is $C^{2}$ we have,%
\begin{eqnarray*}
\frac{\partial ^{2}h}{\partial x_{i}\partial x_{j}} &=&-\frac{\partial }{%
\partial x_{j}}\frac{\frac{\partial H}{\partial x_{i}}\left( x^{\prime
},h\left( x^{\prime }\right) \right) }{\frac{\partial H}{\partial x_{n}}%
\left( x^{\prime },h\left( x^{\prime }\right) \right) } \\
&=&-\frac{\frac{\partial H}{\partial x_{n}}\left( x^{\prime },h\left(
x^{\prime }\right) \right) \frac{\partial }{\partial x_{j}}\left\{ \frac{%
\partial H}{\partial x_{i}}\left( x^{\prime },h\left( x^{\prime }\right)
\right) \right\} -\frac{\partial H}{\partial x_{i}}\left( x^{\prime
},h\left( x^{\prime }\right) \right) \frac{\partial }{\partial x_{j}}\left\{ 
\frac{\partial H}{\partial x_{n}}\left( x^{\prime },h\left( x^{\prime
}\right) \right) \right\} }{\left( \frac{\partial H}{\partial x_{n}}\left(
x^{\prime },h\left( x^{\prime }\right) \right) \right) ^{2}} \\
&=&-\frac{\frac{\partial H}{\partial x_{n}}\left( x^{\prime },h\left(
x^{\prime }\right) \right) \left\{ \frac{\partial ^{2}H}{\partial
x_{i}\partial x_{j}}\left( x^{\prime },h\left( x^{\prime }\right) \right) +%
\frac{\partial ^{2}H}{\partial x_{i}\partial x_{n}}\left( x^{\prime
},h\left( x^{\prime }\right) \right) \left( -\frac{\frac{\partial H}{%
\partial x_{j}}\left( x^{\prime },h\left( x^{\prime }\right) \right) }{\frac{%
\partial H}{\partial x_{n}}\left( x^{\prime },h\left( x^{\prime }\right)
\right) }\right) \right\} }{\left( \frac{\partial H}{\partial x_{n}}\left(
x^{\prime },h\left( x^{\prime }\right) \right) \right) ^{2}} \\
&&+\frac{\frac{\partial H}{\partial x_{i}}\left( x^{\prime },h\left(
x^{\prime }\right) \right) \left\{ \frac{\partial ^{2}H}{\partial
x_{j}\partial x_{n}}\left( x^{\prime },h\left( x^{\prime }\right) \right) +%
\frac{\partial ^{2}H}{\partial x_{n}^{2}}\left( x^{\prime },h\left(
x^{\prime }\right) \right) \left( -\frac{\frac{\partial H}{\partial x_{j}}%
\left( x^{\prime },h\left( x^{\prime }\right) \right) }{\frac{\partial H}{%
\partial x_{n}}\left( x^{\prime },h\left( x^{\prime }\right) \right) }%
\right) \right\} }{\left( \frac{\partial H}{\partial x_{n}}\left( x^{\prime
},h\left( x^{\prime }\right) \right) \right) ^{2}},
\end{eqnarray*}%
which gives%
\begin{eqnarray*}
\frac{\partial ^{2}h}{\partial x_{i}\partial x_{j}} &=&-\frac{\frac{\partial
H}{\partial x_{n}}\frac{\partial ^{2}H}{\partial x_{i}\partial x_{j}}-\frac{%
\partial H}{\partial x_{n}}\frac{\partial ^{2}H}{\partial x_{i}\partial x_{n}%
}\left( \frac{\frac{\partial H}{\partial x_{j}}}{\frac{\partial H}{\partial
x_{n}}}\right) }{\left( \frac{\partial H}{\partial x_{n}}\right) ^{2}}+\frac{%
\frac{\partial H}{\partial x_{i}}\frac{\partial ^{2}H}{\partial
x_{j}\partial x_{n}}-\frac{\partial H}{\partial x_{i}}\frac{\partial ^{2}H}{%
\partial x_{n}^{2}}\left( \frac{\frac{\partial H}{\partial x_{j}}}{\frac{%
\partial H}{\partial x_{n}}}\right) }{\left( \frac{\partial H}{\partial x_{n}%
}\right) ^{2}} \\
&=&-\frac{\frac{\partial ^{2}H}{\partial x_{i}\partial x_{j}}}{\frac{%
\partial H}{\partial x_{n}}}+\frac{\frac{\partial H}{\partial x_{j}}\frac{%
\partial ^{2}H}{\partial x_{i}\partial x_{n}}}{\left( \frac{\partial H}{%
\partial x_{n}}\right) ^{2}}+\frac{\frac{\partial H}{\partial x_{i}}\frac{%
\partial ^{2}H}{\partial x_{j}\partial x_{n}}}{\left( \frac{\partial H}{%
\partial x_{n}}\right) ^{2}}-\frac{\frac{\partial H}{\partial x_{i}}\frac{%
\partial H}{\partial x_{j}}\frac{\partial ^{2}H}{\partial x_{n}^{2}}}{\left( 
\frac{\partial H}{\partial x_{n}}\right) ^{3}}.
\end{eqnarray*}%
Part (4) is established in a similar fashion.
\end{proof}

Now we recall from Fefferman-Phong \cite{FePh} and Tataru \cite[Lemma 5.1]%
{Tat}, the control of odd derivatives in terms of even derivatives for a 
\emph{nonnegative} $C^{3,1}$ function $f$. For the convenience of the
reader, we repeat the argument of Tataru \cite[Lemma 5.1]{Tat} in slightly
greater detail here.

\begin{lemma}
\label{first local}Suppose $f\left( x\right) \geq 0$ and $\left\vert
f^{\prime \prime \prime \prime }\left( x\right) \right\vert \leq 1$ for $%
x\in \mathbb{R}$. Then%
\begin{eqnarray}
\left\vert f^{\prime }\left( x\right) \right\vert &\leq &\frac{8}{3}f\left(
x\right) ^{\frac{3}{4}}+\frac{8}{3}f\left( x\right) ^{\frac{1}{2}}\left\vert
f^{\prime \prime }\left( x\right) \right\vert ^{\frac{1}{2}},  \label{nonneg}
\\
\left\vert f^{\prime \prime \prime }\left( x\right) \right\vert &\leq
&8f\left( x\right) ^{\frac{1}{4}}+8\left\vert f^{\prime \prime }\left(
x\right) \right\vert ^{\frac{1}{2}},  \notag \\
-f^{\prime \prime }\left( x\right) &\leq &\frac{5}{3}f\left( x\right) ^{%
\frac{1}{2}},  \notag
\end{eqnarray}
\end{lemma}

Due to the control of the negative part of $f^{\prime \prime }$ in the third
line of (\ref{nonneg}), we can rewrite the first two lines in terms of the
positive part of $f^{\prime \prime }$.

\begin{corollary}
\label{combine}If $f\geq 0$ and $\left\vert f^{\prime \prime \prime \prime
}\left( x\right) \right\vert \leq 1$, then for $f\left( x\right) \leq 1$, 
\begin{eqnarray*}
\left\vert f^{\prime }\left( x\right) \right\vert &\leq &\max \left\{
8f\left( x\right) ^{\frac{3}{4}},\frac{8}{3}f\left( x\right) ^{\frac{3}{4}}+%
\frac{8}{3}f\left( x\right) ^{\frac{1}{2}}f^{\prime \prime }\left( x\right)
_{+}^{\frac{1}{2}}\right\} \\
&\leq &8f\left( x\right) ^{\frac{3}{4}}+\frac{8}{3}f\left( x\right) ^{\frac{1%
}{2}}f^{\prime \prime }\left( x\right) _{+}^{\frac{1}{2}}, \\
\left\vert f^{\prime \prime \prime }\left( x\right) \right\vert &\leq &\max
\left\{ 24f\left( x\right) ^{\frac{1}{4}},8f\left( x\right) ^{\frac{1}{4}%
}+8f^{\prime \prime }\left( x\right) _{+}^{\frac{1}{2}}\right\} \\
&\leq &24f\left( x\right) ^{\frac{1}{4}}+8f^{\prime \prime }\left( x\right)
_{+}^{\frac{1}{2}}.
\end{eqnarray*}
\end{corollary}

Finally we note that these inequalities extend to $x\in \mathbb{R}^{n}$ in
the form%
\begin{eqnarray}
\left\vert \nabla f\left( x\right) \right\vert &\lesssim &f\left( x\right) ^{%
\frac{3}{4}}+f\left( x\right) ^{\frac{1}{2}}\left\vert \nabla ^{2}f\left(
x\right) \right\vert ^{\frac{1}{2}},  \label{nonneg n} \\
\left\vert \nabla ^{3}f\left( x\right) \right\vert &\lesssim &f\left(
x\right) ^{\frac{1}{4}}+\left\vert \nabla ^{2}f\left( x\right) \right\vert ^{%
\frac{1}{2}},  \notag \\
\left\vert \nabla ^{2}f\left( x\right) \right\vert &\lesssim &\sup_{\Theta
\in \mathbb{S}^{n-1}}\left[ \partial _{\Theta }^{2}f\left( x\right) \right]
_{+}+f\left( x\right) ^{\frac{1}{2}},  \notag
\end{eqnarray}%
provided $\left\vert \nabla ^{4}f\left( x\right) \right\vert \leq 1$ on $%
\mathbb{R}^{n}$, upon using the equivalence of norms, 
\begin{equation*}
\sup_{\Theta \in \mathbb{S}^{n-1}}\left\vert \partial _{\Theta }^{k}f\left(
x\right) \right\vert \approx \left\vert \nabla ^{k}f\left( x\right)
\right\vert ,\ \ \ \ \ x\in \mathbb{R}^{n},1\leq k\leq 4,
\end{equation*}%
on the finite dimensional vector space of homogeneous polynomials on $%
\mathbb{R}^{n}$ of degree $k$. Here $\partial _{\Theta }$ denotes the
directional derivative in the direction of the unit vector $\Theta $ in the
sphere $\mathbb{S}^{n-1}$. For example, when $n=2$, we can identify $\Theta
=\theta $ with $\left( \cos \theta ,\sin \theta \right) $ and we have 
\begin{equation*}
\partial _{\theta }f=\left( \cos \theta ,\sin \theta \right) \cdot \nabla
f=\cos \theta \frac{\partial f}{\partial x_{1}}+\sin \theta \frac{\partial f%
}{\partial x_{2}}.
\end{equation*}

\begin{proof}
To see the inequalities in (\ref{nonneg}) we may suppose that $x=0$. Since
the inequalities are invariant under the rescalings $f\left( x\right)
\rightarrow \lambda ^{-4}f\left( \lambda x\right) $ for $\lambda >0$, we may
also assume $f\left( 0\right) \leq 1$. We write%
\begin{equation*}
0\leq f\left( y\right) \leq f\left( 0\right) +f^{\prime }\left( 0\right) y+%
\frac{1}{2}f^{\prime \prime }\left( 0\right) y^{2}+\frac{1}{6}f^{\prime
\prime \prime }\left( 0\right) y^{3}+\frac{1}{24}y^{4},
\end{equation*}%
to obtain%
\begin{equation*}
\left\vert f^{\prime }\left( 0\right) y+\frac{1}{6}f^{\prime \prime \prime
}\left( 0\right) y^{3}\right\vert \leq f\left( 0\right) +\frac{1}{2}%
f^{\prime \prime }\left( 0\right) y^{2}+\frac{1}{24}y^{4}.
\end{equation*}%
The same bound for $2y$ is%
\begin{equation*}
\left\vert f^{\prime }\left( 0\right) 2y+\frac{4}{3}f^{\prime \prime \prime
}\left( 0\right) y^{3}\right\vert \leq f\left( 0\right) +2f^{\prime \prime
}\left( 0\right) y^{2}+\frac{2}{3}y^{4},
\end{equation*}%
and we claim that combining the bounds yields%
\begin{eqnarray}
\left\vert f^{\prime }\left( 0\right) y\right\vert &\leq &\frac{3}{2}f\left(
0\right) +f^{\prime \prime }\left( 0\right) y^{2}+\frac{1}{6}y^{4},
\label{combining} \\
\left\vert f^{\prime \prime \prime }\left( 0\right) y^{3}\right\vert &\leq
&3f\left( 0\right) +3f^{\prime \prime }\left( 0\right) y^{2}+y^{4}.  \notag
\end{eqnarray}%
Indeed, we have both%
\begin{eqnarray*}
-f^{\prime }\left( 0\right) y-\frac{1}{6}f^{\prime \prime \prime }\left(
0\right) y^{3} &\leq &f\left( 0\right) +\frac{1}{2}f^{\prime \prime }\left(
0\right) y^{2}+\frac{1}{24}y^{4}, \\
f^{\prime }\left( 0\right) 2y+\frac{4}{3}f^{\prime \prime \prime }\left(
0\right) y^{3} &\leq &f\left( 0\right) +2f^{\prime \prime }\left( 0\right)
y^{2}+\frac{2}{3}y^{4},
\end{eqnarray*}%
and adding $8$ times the first inequality to the second gives%
\begin{eqnarray*}
-6f^{\prime }\left( 0\right) y &\leq &9f\left( 0\right) +6f^{\prime \prime
}\left( 0\right) y^{2}+y^{4}; \\
-f^{\prime }\left( 0\right) y &\leq &\frac{3}{2}f\left( 0\right) +f^{\prime
\prime }\left( 0\right) y^{2}+\frac{1}{6}y^{4}.
\end{eqnarray*}%
On the other hand, we also have both%
\begin{eqnarray*}
f^{\prime }\left( 0\right) y+\frac{1}{6}f^{\prime \prime \prime }\left(
0\right) y^{3} &\leq &f\left( 0\right) +\frac{1}{2}f^{\prime \prime }\left(
0\right) y^{2}+\frac{1}{24}y^{4}, \\
-f^{\prime }\left( 0\right) 2y-\frac{4}{3}f^{\prime \prime \prime }\left(
0\right) y^{3} &\leq &f\left( 0\right) +2f^{\prime \prime }\left( 0\right)
y^{2}+\frac{2}{3}y^{4},
\end{eqnarray*}%
and adding $8$ times the first inequality to the second gives%
\begin{eqnarray*}
6f^{\prime }\left( 0\right) y &\leq &9f\left( 0\right) +6f^{\prime \prime
}\left( 0\right) y^{2}+y^{4}; \\
f^{\prime }\left( 0\right) y &\leq &\frac{3}{2}f\left( 0\right) +f^{\prime
\prime }\left( 0\right) y^{2}+\frac{1}{6}y^{4}.
\end{eqnarray*}%
Altogether this gives the first inequality in (\ref{combining}), and the
second inequality is proved similarly.

Now set $y=\frac{f\left( 0\right) ^{\frac{1}{2}}}{f\left( 0\right) ^{\frac{1%
}{4}}+\left\vert f^{\prime \prime }\left( 0\right) \right\vert ^{\frac{1}{2}}%
}$ in the first inequality in (\ref{combining}) to obtain%
\begin{eqnarray*}
\left\vert f^{\prime }\left( 0\right) \right\vert \frac{f\left( 0\right) ^{%
\frac{1}{2}}}{f\left( 0\right) ^{\frac{1}{4}}+\left\vert f^{\prime \prime
}\left( 0\right) \right\vert ^{\frac{1}{2}}} &\leq &\frac{3}{2}f\left(
0\right) \frac{f\left( 0\right) ^{\frac{1}{2}}}{f\left( 0\right) ^{\frac{1}{4%
}}}+\left\vert f^{\prime \prime }\left( 0\right) \right\vert \left( \frac{%
f\left( 0\right) ^{\frac{1}{2}}}{\left\vert f^{\prime \prime }\left(
0\right) \right\vert ^{\frac{1}{2}}}\right) ^{2}+\frac{1}{6}\left( \frac{%
f\left( 0\right) ^{\frac{1}{2}}}{f\left( 0\right) ^{\frac{1}{4}}}\right) ^{4}
\\
&\leq &\frac{3}{2}f\left( 0\right) ^{\frac{5}{4}}+f\left( 0\right) +\frac{1}{%
6}f\left( 0\right) ,
\end{eqnarray*}%
which gives%
\begin{eqnarray*}
\left\vert f^{\prime }\left( 0\right) \right\vert &\leq &\left( f\left(
0\right) ^{\frac{1}{4}}+\left\vert f^{\prime \prime }\left( 0\right)
\right\vert ^{\frac{1}{2}}\right) \left( \frac{3}{2}f\left( 0\right) ^{\frac{%
3}{4}}+\frac{7}{6}f\left( 0\right) ^{\frac{1}{2}}\right) \\
&=&\left( f\left( 0\right) ^{\frac{3}{4}}+f\left( 0\right) ^{\frac{1}{2}%
}\left\vert f^{\prime \prime }\left( 0\right) \right\vert ^{\frac{1}{2}%
}\right) \left( \frac{3}{2}f\left( 0\right) ^{\frac{1}{4}}+\frac{7}{6}%
\right) .
\end{eqnarray*}%
Using $f\left( 0\right) \leq 1$ we thus obtain 
\begin{equation*}
\left\vert f^{\prime }\left( 0\right) \right\vert \leq \frac{8}{3}\left(
f\left( 0\right) ^{\frac{3}{4}}+f\left( 0\right) ^{\frac{1}{2}}\left\vert
f^{\prime \prime }\left( 0\right) \right\vert ^{\frac{1}{2}}\right) ,
\end{equation*}%
which is the first line in (\ref{nonneg}).

The second line in (\ref{nonneg}) is proved by setting%
\begin{equation*}
y=\max \left\{ f\left( 0\right) ^{\frac{1}{4}},\left\vert f^{\prime \prime
}\left( 0\right) \right\vert ^{\frac{1}{2}}\right\} ,
\end{equation*}%
and is left for the reader.

Finally, the third line in (\ref{nonneg}) is obtained by setting $y=f\left(
0\right) ^{\frac{1}{4}}$ in the first line of (\ref{combining}), which gives%
\begin{equation*}
0\leq \frac{3}{2}f\left( 0\right) +f^{\prime \prime }\left( 0\right) f\left(
0\right) ^{\frac{1}{2}}+\frac{1}{6}f\left( 0\right) .
\end{equation*}
\end{proof}

For $\delta >0$ define%
\begin{equation*}
r_{\delta }\left( x\right) \equiv \max \left\{ f\left( x\right) ^{\frac{1}{%
4+2\delta }},\left( \sup_{\Theta \in \mathbb{S}^{n-1}}\left[ \partial
_{\Theta }^{2}f\left( x\right) \right] _{+}\right) ^{\frac{1}{2+2\delta }%
}\right\} ,\ \ \ \ \ x\in \mathbb{R}^{n}.
\end{equation*}%
Following Tataru \cite{Tat} we now show that $r_{\delta }$ is \emph{slowly
varying}, i.e. there are $0<c,\gamma <1$ such that%
\begin{equation*}
\left\vert r_{\delta }\left( x\right) -r_{\delta }\left( y\right)
\right\vert \leq \gamma r_{\delta }\left( x\right) ,\ \ \ \ \ \text{for }%
\left\vert x-y\right\vert \leq cr\left( x\right) .
\end{equation*}%
We prove this only in $\mathbb{R}$, and leave the straightforward extension
to higher dimensions for the reader.

\begin{lemma}
\label{second}Let $\delta >0$. If $f\left( x\right) \geq 0$ and $\left\vert
\nabla ^{4}f\left( x\right) \right\vert \leq 1$ for $x\in \mathbb{R}$, then%
\begin{equation*}
\left\vert r_{\delta }\left( x\right) -r_{\delta }\left( y\right)
\right\vert \leq \left( \frac{1}{2}\right) ^{\frac{1}{4+2\delta }}r_{\delta
}\left( x\right) ,\ \ \ \ \ \text{for }\left\vert x-y\right\vert \leq \frac{1%
}{200}r\left( x\right) .
\end{equation*}
\end{lemma}

\begin{proof}
By translation and rescaling we can assume that $x=0$ and $r_{\delta }\left(
0\right) =1$. Then $f\left( 0\right) ^{\frac{1}{4+2\delta }},f^{\prime
\prime }\left( 0\right) _{+}^{\frac{1}{2+2\delta }}\leq 1$ and by Corollary %
\ref{combine} we have%
\begin{equation*}
\left\vert f\left( 0\right) \right\vert \leq 1,\ \ \ \left\vert f^{\prime
}\left( 0\right) \right\vert \leq 11,\ \ \ \left\vert f^{\prime \prime
}\left( 0\right) \right\vert \leq \frac{5}{3},\ \ \ \left\vert f^{\prime
\prime \prime }\left( 0\right) \right\vert \leq 32,
\end{equation*}%
and so with $\left\vert y\right\vert =\left\vert y-x\right\vert \leq \frac{1%
}{200}$, Taylor's formula shows that both $f$ and $f_{+}^{\prime \prime }$
are slowly varying, i.e.%
\begin{equation*}
\left\vert f\left( y\right) -f\left( 0\right) \right\vert \leq \frac{11}{200}%
+\frac{\frac{5}{3}}{2\left( 200\right) ^{2}}+\frac{32}{6\left( 200\right)
^{3}}+\frac{1}{24\left( 200\right) ^{4}}<\frac{1}{2},
\end{equation*}%
and%
\begin{equation*}
\left\vert f^{\prime \prime }\left( y\right) _{+}-f^{\prime \prime }\left(
0\right) _{+}\right\vert \leq \frac{32}{200}+\frac{1}{2\left( 200\right) ^{2}%
}<\frac{1}{2},
\end{equation*}%
which yields%
\begin{eqnarray*}
\left\vert r_{\delta }\left( x\right) -r_{\delta }\left( y\right)
\right\vert &=&\left\vert \max \left\{ f\left( 0\right) ^{\frac{1}{4+2\delta 
}},\left\vert f^{\prime \prime }\left( 0\right) \right\vert ^{\frac{1}{%
2+2\delta }}\right\} -\max \left\{ f\left( y\right) ^{\frac{1}{4+2\delta }%
},\left\vert f^{\prime \prime }\left( y\right) \right\vert ^{\frac{1}{%
2+2\delta }}\right\} \right\vert \\
&\leq &\max \left\{ \left\vert f\left( 0\right) ^{\frac{1}{4+2\delta }%
}-f\left( y\right) ^{\frac{1}{4+2\delta }}\right\vert ,\left\vert \left\vert
f^{\prime \prime }\left( 0\right) \right\vert ^{\frac{1}{2+2\delta }%
}-\left\vert f^{\prime \prime }\left( y\right) \right\vert ^{\frac{1}{%
2+2\delta }}\right\vert \right\} \leq \gamma <1.
\end{eqnarray*}%
Indeed, if $\left\vert f^{\prime \prime }\left( 0\right) \right\vert =0$ and 
$\left\vert f^{\prime \prime }\left( y\right) \right\vert =\frac{1}{2}$,
then $\left\vert \left\vert f^{\prime \prime }\left( 0\right) \right\vert ^{%
\frac{1}{2+2\delta }}-\left\vert f^{\prime \prime }\left( y\right)
\right\vert ^{\frac{1}{2+2\delta }}\right\vert =\left( \frac{1}{2}\right) ^{%
\frac{1}{2+2\delta }}$ , while if $f\left( 0\right) =0$ and $f\left(
y\right) =\frac{1}{2}$, then $\left\vert f\left( 0\right) ^{\frac{1}{%
4+2\delta }}-f\left( y\right) ^{\frac{1}{4+2\delta }}\right\vert =\left( 
\frac{1}{2}\right) ^{\frac{1}{4+2\delta }}$, and since these cases are
optimal, we have the above inequality with $\gamma =\left( \frac{1}{2}%
\right) ^{\frac{1}{4+2\delta }}<1$.
\end{proof}

\subsection{A provisional SOS theorem}

Here we begin with the following provisional sum of squares theorem, an
analogue of Lemmas 1 and 2 in \cite{FePh}, which will be used to prove our
main Theorems \ref{main 2D}\ and \ref{efs eps} below. For any $0\leq \beta
\leq 1$, and any continuous function $h$ defined on a ball $B$ in $\mathbb{R}%
^{n}$, we define%
\begin{equation*}
\left\Vert h\right\Vert _{\func{Lip}_{\beta }\left( B\right) }\equiv
\sup_{x,y\in B}\frac{\left\vert h\left( x\right) -h\left( y\right)
\right\vert }{\left\vert x-y\right\vert ^{\beta }},
\end{equation*}%
and for $k\in \mathbb{Z}_{+}$ we denote by $C^{k,\beta }\left( B\right) $
the space of functions $f$ on $B$ normed by%
\begin{equation*}
\left\Vert f\right\Vert _{C^{k,\beta }\left( B\right) }\equiv \sum_{\ell
=0}^{k}\left\Vert \nabla ^{\ell }f\right\Vert _{L^{\infty }\left( B\right)
}+\left\Vert \nabla ^{k}f\right\Vert _{\func{Lip}_{\beta }\left( B\right) }\
.
\end{equation*}

We will use the following `distance function' related to derivatives of $f$
that was used in Tataru \cite{Tat} and Bony \cite{Bon}:%
\begin{equation}
\rho _{f;\delta }\left( x\right) \equiv \max \left\{ f\left( x\right) ^{%
\frac{1}{4+2\delta }},\left( \sup_{\Theta \in \mathbb{S}^{n-1}}\left[
\partial _{\Theta }^{2}f\left( x\right) \right] _{+}\right) ^{\frac{1}{%
2+2\delta }},\left\vert \nabla ^{4}f\left( x\right) \right\vert ^{\frac{1}{%
2\delta }}\right\} ,\ \ \ \ \ x\in \mathbb{R}^{n}.  \label{def rho}
\end{equation}

\begin{acknowledgement}
We thank Sullivan Francis MacDonald for pointing out an error in the
original formulation and proof of the next theorem, and which is now
weakened from its previous form. However, this has no significant effect on
the remaining results in this paper, nor on the results in the next two
papers in this series that reference the current paper.
\end{acknowledgement}

\begin{theorem}
\label{provisional}Suppose $0<\delta ,\eta <\frac{1}{2}$ and $n\geq 1$. Then
there exists a constant $N=N\left( \delta ,\eta ,n\right) $ depending on $%
\delta $, $\eta $ and $n$ with the following property. For every nonnegative 
$f\in C^{4,2\delta }\left( \mathbb{R}^{n}\right) $ satisfying%
\begin{equation}
\left\vert \nabla ^{4}f\left( x\right) \right\vert \leq Cf\left( x\right) ^{%
\frac{\delta }{2+\delta }}\text{ and }\sup_{\Theta \in \mathbb{S}^{n-1}}%
\left[ \partial _{\Theta }^{2}f\left( x\right) \right] _{+}\leq Cf\left(
x\right) ^{\eta },  \label{diff prov}
\end{equation}%
and with $\rho _{f;\delta }$ as in (\ref{def rho}), there are functions $%
g_{\ell }\in C^{2,\delta }\left( \mathbb{R}^{n}\right) $ satisfying%
\begin{eqnarray}
\left\vert D^{\alpha }g_{\ell }\left( x\right) \right\vert  &\leq &C\rho
_{f;\delta }\left( x\right) ^{2+\delta -\left\vert \alpha \right\vert },\ \
\ \ \ 0\leq \left\vert \alpha \right\vert \leq 2,  \label{Dg delta} \\
\left[ g_{\ell }\right] _{\alpha ,\delta }\left( x\right)  &\leq &C,\ \ \ \
\ \left\vert \alpha \right\vert =2,  \notag
\end{eqnarray}%
and%
\begin{eqnarray}
\left\vert D^{\alpha }g_{\ell }^{2}\left( x\right) \right\vert  &\leq &C\rho
_{f;\delta }\left( x\right) ^{4+2\delta -\left\vert \alpha \right\vert },\ \
\ \ \ 0\leq \left\vert \alpha \right\vert \leq 2,  \label{Dg^2 delta} \\
\left[ g_{\ell }^{2}\right] _{\alpha ,2\delta }\left( x\right)  &\leq &C,\ \
\ \ \ \left\vert \alpha \right\vert =2,  \notag
\end{eqnarray}%
and nonnegative functions $h_{\ell }\in C^{4,2\delta }\left( \mathbb{R}%
^{n}\right) $, for $1\leq \ell \leq N$, such that 
\begin{equation*}
f\left( x\right) =\sum_{\ell =1}^{N}g_{\ell }\left( x\right) ^{2}+\sum_{\ell
=1}^{N}h_{\ell }\left( x\right) ,\ \ \ \ \ x\in \mathbb{R}^{n},
\end{equation*}%
and where each function $h_{\ell }$ can be further decomposed into a
countable sum of functions from a bounded set in $C^{4,2\delta }\left( 
\mathbb{R}^{n}\right) $ with pairwise disjoint supports, i.e. 
\begin{eqnarray}
&&h_{\ell }=\sum_{m=1}^{\infty }k_{\ell ,m}\ ,  \label{special} \\
&&\func{Supp}k_{\ell ,m}\cap \func{Supp}k_{\ell ,m^{\prime }}=\emptyset 
\text{ for }m\neq m^{\prime },  \notag \\
&&\left\{ k_{\ell ,m}\right\} _{\ell ,m}\text{ is a bounded set of functions
in }C^{4,2\delta }\left( \mathbb{R}^{n}\right) .  \notag
\end{eqnarray}%
Moreover, each $k_{\ell ,m}$ has the two critical properties that\newline
(\textbf{i}) The functions $k_{\ell ,m}$ can be factored as%
\begin{equation*}
k_{\ell ,m}\left( x\right) =\Phi _{\ell ,m}\left( x\right) ^{2}\kappa _{\ell
,m}\left( x\right) ,\ \ \ \ \ \Phi _{\ell ,m}\in C^{2,\delta },\kappa _{\ell
,m}\in C^{4,2\delta },
\end{equation*}%
where $\kappa _{\ell ,m}\left( x\right) $ is a function of just $n-1$
variables, i.e. there is a rotation $R=R_{\ell ,m}$ depending on $\ell ,m$
such that in the rotated variables $y=Rx$, the function $\kappa _{\ell ,m}$
is \emph{independent} of $y_{n}$, and\newline
(\textbf{ii}) If we define the constant $0<\delta _{1}<\delta $ by the
equation%
\begin{equation*}
\frac{\delta _{1}}{2+\delta _{1}}=\eta \frac{\delta }{1+\delta },
\end{equation*}%
then the function $\kappa _{\ell ,m}$ of $n-1$ variables satisfies the
following analogue of (\ref{diff prov}),%
\begin{eqnarray}
\left\vert \nabla ^{4}\kappa _{\ell ,m}\left( x\right) \right\vert  &\leq
&C\kappa _{\ell ,m}\left( x\right) ^{\frac{\delta _{1}}{2+\delta _{1}}}\text{
}  \label{diff prov next} \\
\sup_{\Theta \in \mathbb{S}^{n-1}}\left[ \partial _{\Theta }^{2}\kappa
_{\ell ,m}\left( x\right) \right] _{+} &\leq &C\kappa _{\ell ,m}\left(
x\right) ^{\eta }.  \notag
\end{eqnarray}%
\newline
Here the families $\left\{ \Phi _{\ell ,m}\right\} _{\ell ,m}$ and $\left\{
\kappa _{\ell ,m}\right\} _{\ell ,m}$ lie in bounded sets in $C^{2,\delta }$
and $C^{4,2\delta }$ respectively, with bounds depending only on $\delta $, $%
\eta $ and $n$.
\end{theorem}

\begin{remark}
The purpose of the first inequality in (\ref{diff prov}) is to limit our
analysis to only the cases when $\left\vert \nabla ^{4}f\right\vert ^{\frac{1%
}{2\delta }}$ is \emph{not} the dominant term in the definition of $\rho
_{f;\delta }$ (the implicit function theorem is \emph{not} decisive if $%
\left\vert \nabla ^{4}f\right\vert ^{\frac{1}{2\delta }}$ dominates). The
purpose of the second inequality in (\ref{diff prov}) is to show that the
first inequality is inherited by the functions $\kappa _{\ell ,m}$ that
arise in the induction step in applications of Theorem \ref{provisional},
but with the smaller index $\delta _{1}$ in place of $\delta $. In dimension 
$n=1$ this differential inequality can be dropped by results of Bony in \cite%
{Bon}, while in dimension $n\geq 5$, we will see in Remark \ref{necc diff
ineq} below that some inequality of this type is in general required.
\end{remark}

\begin{proof}
We begin the proof of Theorem \ref{provisional} by further adapting the
version of the Fefferman-Phong argument due to Bony in \cite[Th\'{e}or\`{e}%
me 2]{Bon} using the `distance function' $\rho \left( x\right) $. Define $%
\Gamma \equiv \left\{ x\in \mathbb{R}^{n}:f\left( x\right) =0\right\} $ and $%
d\left( x\right) \equiv \func{dist}\left( x,\Gamma \right) $. Recall the
sublinear operators $\left[ \cdot \right] _{\alpha ,\delta }$ defined in (%
\ref{def mod D}) above,%
\begin{equation*}
\left[ \varphi \right] _{\alpha ,\delta }\left( x\right) \equiv
\limsup_{y,z\rightarrow x}\frac{\left\vert D^{\alpha }\varphi \left(
y\right) -D^{\alpha }\varphi \left( z\right) \right\vert }{\left\vert
y-z\right\vert ^{\delta }}.
\end{equation*}

One now writes $U\equiv \Gamma ^{c}$ as a countable union of cubes $Q_{\nu }$
with center $x^{\nu }$ and diameter comparable to $\rho \left( x^{\nu
}\right) $. Now for all $s>0$ sufficiently small, there is a collection of
balls $\left\{ B_{\nu }\right\} _{\nu =1}^{\infty }$ covering $U$ with
centers $x^{\nu }$ and radii 
\begin{equation*}
r_{\nu }\equiv s\rho \left( x^{\nu }\right) ,
\end{equation*}%
having bounded overlap $\sum_{\nu =1}^{\infty }\mathbf{1}_{B_{\nu }}\leq C%
\mathbf{1}_{U}$, as well as a partition of unity $\left\{ \Phi _{\nu
}\right\} _{\nu =1}^{\infty }$ subordinate to this collection satisfying%
\begin{eqnarray*}
\sum_{\nu =1}^{\infty }\Phi _{\nu }\left( x\right) ^{2} &=&1,\ \ \ \ \ \func{%
Supp}\Phi _{\nu }\subset B_{\nu }, \\
\sup_{x\in U}\left\vert D^{\alpha }\Phi _{\nu }\left( x\right) \right\vert
&\leq &C_{\alpha ,s}\frac{1}{r_{\nu }^{\left\vert \alpha \right\vert }},\ \
\ \ \ \alpha \in \mathbb{Z}_{+}^{n}.
\end{eqnarray*}

We now wish to show that each function $\Phi _{\nu }^{2}f$ can be decomposed
as a sum of a square with control, and a nonnegative function $h_{\nu }\in
C^{4,2\delta }$ with the special decomposition property as in (\ref{special}%
). For this we will use the following inequalities for $x\in B_{\nu }$, 
\begin{eqnarray}
0 &<&\rho \left( x\right) \leq Cd\left( x\right) ,  \label{bit of work} \\
\left\vert D^{\alpha }f\left( x\right) \right\vert &\leq &C\rho \left(
x\right) ^{4+2\delta -\left\vert \alpha \right\vert },\ \ \ \ \ 0\leq
\left\vert \alpha \right\vert \leq 4,  \notag \\
\left[ f\right] _{\alpha ,2\delta }\left( x\right) &\leq &C,\ \ \ \ \
\left\vert \alpha \right\vert =4,  \notag
\end{eqnarray}%
We now prove (\ref{bit of work}). From (\ref{diff prov}) we see that $%
\left\vert \nabla ^{2}f\left( x\right) \right\vert =\left\vert \nabla
^{4}f\left( x\right) \right\vert =0$ for all $x\in \Gamma $, and then the
first two lines in (\ref{nonneg n}), together with $f\in C^{4,2\delta
}\left( \mathbb{R}^{n}\right) $, show that $f\left( x\right) \leq
C\left\vert x-x_{0}\right\vert ^{4+2\delta }$ and $\left\vert \nabla
^{2}f\left( x\right) \right\vert \leq C\left\vert x-x_{0}\right\vert
^{2+2\delta }$ for all $x_{0}\in \Gamma $. Thus we have $f\left( x\right) ^{%
\frac{1}{4+2\delta }}\leq C\inf_{x_{0}\in \Gamma }\left\vert
x-x_{0}\right\vert =Cd\left( x\right) $ and $\left( \sup_{\Theta \in \mathbb{%
S}^{n-1}}\left[ \partial _{\Theta }^{2}f\left( x\right) \right] _{+}\right)
^{\frac{1}{2+2\delta }}\leq C\inf_{x_{0}\in \Gamma }\left\vert
x-x_{0}\right\vert =Cd\left( x\right) $, and so 
\begin{eqnarray*}
f\left( x\right) ^{\frac{1}{4+2\delta }} &\leq &Cd\left( x\right) \text{ and 
}\left( \sup_{\Theta \in \mathbb{S}^{n-1}}\left[ \partial _{\Theta
}^{2}f\left( x\right) \right] _{+}\right) ^{\frac{1}{2+2\delta }}\leq
Cd\left( x\right) , \\
\text{and }\rho \left( x\right) &=&\max \left\{ f\left( x\right) ^{\frac{1}{%
4+2\delta }},\left( \sup_{\Theta \in \mathbb{S}^{n-1}}\left[ \partial
_{\Theta }^{2}f\left( x\right) \right] _{+}\right) ^{\frac{1}{2+2\delta }%
}\right\} \leq Cd\left( x\right) ,
\end{eqnarray*}%
which is the first line in (\ref{bit of work}).

On the other hand,%
\begin{eqnarray*}
\left\vert D^{0}f\left( x\right) \right\vert &=&f\left( x\right) \leq \rho
\left( x\right) ^{4+2\delta }, \\
\text{and }\left\vert D^{2}f\left( x\right) \right\vert &\leq &C\sup_{\Theta
\in \mathbb{S}^{n-1}}\left[ \partial _{\Theta }^{2}f\left( x\right) \right]
_{+}+Cf\left( x\right) ^{\frac{1}{2}}\rho \left( x\right) ^{\delta }\leq
C\rho \left( x\right) ^{2+2\delta },
\end{eqnarray*}%
where the first inequality in the second line above follows using the third
line in (\ref{nonneg n}) applied to $\eta \left( x\right) \equiv \rho \left(
x^{\nu }\right) ^{-2\delta }f\left( x\right) $,%
\begin{equation*}
\left\vert \rho \left( x^{\nu }\right) ^{-2\delta }\nabla ^{2}f\left(
x\right) \right\vert \lesssim \rho \left( x^{\nu }\right) ^{-2\delta
}\sup_{\Theta \in \mathbb{S}^{n-1}}\left[ \partial _{\Theta }^{2}f\left(
x\right) \right] _{+}+\left\vert \rho \left( x^{\nu }\right) ^{-2\delta
}f\left( x\right) \right\vert ^{\frac{1}{2}}\lesssim \rho \left( x^{\nu
}\right) ^{-2\delta },
\end{equation*}%
since $\rho =r_{\delta }$ is slowly varying. From the control of odd order
derivatives by those of even order in the first two lines of Lemma \ref%
{first local} applied to $\eta \left( x\right) \equiv \rho \left( x^{\nu
}\right) ^{-2\delta }f\left( x\right) $, we then obtain%
\begin{eqnarray*}
\left\vert \rho \left( x^{\nu }\right) ^{-2\delta }\nabla f\left( x\right)
\right\vert &\lesssim &\left( \rho \left( x^{\nu }\right) ^{-2\delta
}f\left( x\right) \right) ^{\frac{3}{4}}+\left( \rho \left( x^{\nu }\right)
^{-2\delta }f\left( x\right) \right) ^{\frac{1}{2}}\left( \rho \left( x^{\nu
}\right) ^{-2\delta }\nabla ^{2}f\left( x\right) \right) ^{\frac{1}{2}} \\
&\lesssim &\left( \rho \left( x^{\nu }\right) ^{-2\delta }\rho \left(
x\right) ^{4+2\delta }\right) ^{\frac{3}{4}}+\left( \rho \left( x^{\nu
}\right) ^{-2\delta }\rho \left( x\right) ^{4+2\delta }\right) ^{\frac{1}{2}%
}\left( \rho \left( x^{\nu }\right) ^{-2\delta }\rho \left( x\right)
^{2+2\delta }\right) ^{\frac{1}{2}} \\
&\lesssim &\rho \left( x^{\nu }\right) ^{3},\ \ \ \ \ x\in B_{\nu }\ ,
\end{eqnarray*}%
since $\rho $ is slowly varying on $B_{\nu }$, and similarly%
\begin{eqnarray*}
&&\left\vert \rho \left( x^{\nu }\right) ^{-2\delta }\nabla ^{3}f\left(
x\right) \right\vert \lesssim \left( \rho \left( x^{\nu }\right) ^{-2\delta
}f\left( x\right) \right) ^{\frac{1}{4}}+\left( \rho \left( x^{\nu }\right)
^{-2\delta }f\left( x\right) \right) ^{\frac{1}{2}}\left\vert \rho \left(
x^{\nu }\right) ^{-2\delta }\nabla ^{2}f\left( x\right) \right\vert ^{\frac{1%
}{2}} \\
&\lesssim &\left( \rho \left( x^{\nu }\right) ^{-2\delta }\rho \left(
x\right) ^{4+2\delta }\right) ^{\frac{1}{4}}+\left( \rho \left( x^{\nu
}\right) ^{-2\delta }\rho \left( x\right) ^{4+2\delta }\right) ^{\frac{1}{2}%
}\left\vert \rho \left( x^{\nu }\right) ^{-2\delta }\rho \left( x\right)
^{2+2\delta }\right\vert ^{\frac{1}{2}} \\
&\lesssim &\rho \left( x\right) +\rho \left( x\right) ^{3}\lesssim \rho
\left( x\right) ,\ \ \ \ \ x\in B_{\nu }\ .
\end{eqnarray*}%
Combined with $\left\vert \nabla ^{4}f\left( x\right) \right\vert \leq C\rho
\left( x\right) ^{2\delta }$, this gives the second line in (\ref{bit of
work}), and the subproduct rule (\ref{subproduct}) yields the third line.
This completes the proof of (\ref{bit of work}).

We claim 
\begin{eqnarray*}
\Phi _{\nu }f &=&g_{\nu }^{2}+h_{\nu }\ , \\
\left\Vert D^{\alpha }g_{\nu }\right\Vert _{\func{Lip}_{\delta }\left(
B_{\nu }\right) } &\leq &C_{\alpha ,s}r_{\nu }^{2-\left\vert \alpha
\right\vert },\ \ \ \ \ 0\leq \left\vert \alpha \right\vert \leq 2, \\
\left\Vert \left[ g_{\nu }\right] _{\alpha ,\delta }\right\Vert _{\infty }
&\leq &C_{\alpha ,s},\ \ \ \ \ \left\vert \alpha \right\vert =2, \\
\left\Vert D^{\alpha }h_{\nu }\right\Vert _{\func{Lip}_{2\delta }\left(
B_{\nu }\right) } &\leq &C_{\alpha ,s}r_{\nu }^{4-\left\vert \alpha
\right\vert },\ \ \ \ \ 0\leq \left\vert \alpha \right\vert \leq 4, \\
\left\Vert \left[ h_{\nu }\right] _{\alpha ,2\delta }\right\Vert _{\infty }
&\leq &C_{\alpha ,s},\ \ \ \ \ \left\vert \alpha \right\vert =4,
\end{eqnarray*}%
where the constant $C_{\alpha ,s}$ is \emph{independent} of $\nu $.
Moreover, we also have analogous inequalities for $g_{\nu }^{2}$ that mirror
those of $f$:%
\begin{eqnarray*}
\left\Vert D^{\alpha }g_{\nu }^{2}\right\Vert _{\func{Lip}_{2\delta }\left(
B_{\nu }\right) } &\leq &C_{\alpha ,s}r_{\nu }^{4-\left\vert \alpha
\right\vert },\ \ \ \ \ 0\leq \left\vert \alpha \right\vert \leq 4, \\
\left\Vert \left[ g_{\nu }^{2}\right] _{\alpha ,2\delta }\right\Vert
_{\infty } &\leq &C_{\alpha ,s},\ \ \ \ \ \left\vert \alpha \right\vert =4.
\end{eqnarray*}

\textbf{Case I}: $f\left( x^{\nu }\right) \geq c\rho \left( x^{\nu }\right)
^{4+2\delta }$.

\medskip

In this case we define 
\begin{equation*}
f_{1}\left( x\right) =\Phi _{\nu }\left( x\right) \sqrt{f\left( x\right) },
\end{equation*}%
and use the inequalities in (\ref{bit of work}). A first order partial
derivative $D^{\mu }f_{1}$ of $f_{1}$\ is%
\begin{equation*}
D^{\mu }f_{1}\left( x\right) =D^{\mu }\Phi _{\nu }\left( x\right) \sqrt{%
f\left( x\right) }+\Phi _{\nu }\left( x\right) \frac{D^{\mu }f\left(
x\right) }{2\sqrt{f\left( x\right) }},
\end{equation*}%
and its modulus is bounded by%
\begin{equation*}
\left\vert D^{\mu }\Phi _{\nu }\left( x\right) \sqrt{f\left( x\right) }%
\right\vert +\left\vert \Phi _{\nu }\left( x\right) \frac{D^{\mu }f\left(
x\right) }{2\sqrt{f\left( x\right) }}\right\vert \lesssim \frac{1}{r\left(
x^{\nu }\right) }\rho \left( x^{\nu }\right) ^{2+\delta }+\frac{\rho \left(
x\right) ^{3+2\delta }}{2\rho \left( x\right) ^{2+\delta }}\lesssim \rho
\left( x\right) ^{1+\delta },
\end{equation*}%
by the assumption of Case I, together with the slowly varying property of $%
\rho \approx r$. A second order partial derivative $D^{\mu }f_{1}$ is%
\begin{equation*}
D^{\mu }f_{1}\left( x\right) =D^{\mu }\Phi _{\nu }\left( x\right) \sqrt{%
f\left( x\right) }+D^{\alpha }\Phi _{\nu }\left( x\right) \frac{D^{\beta
}f\left( x\right) }{\sqrt{f\left( x\right) }}+\Phi _{\nu }\left( x\right) 
\frac{D^{\mu }f\left( x\right) }{2\sqrt{f\left( x\right) }}-\Phi _{\nu
}\left( x\right) \frac{D^{\alpha }f\left( x\right) D^{\beta }f\left(
x\right) }{4f\left( x\right) ^{\frac{3}{2}}},
\end{equation*}%
where $\left\vert \alpha \right\vert =\left\vert \beta \right\vert =1$. Now
for $x\in B_{\nu }$ we have%
\begin{eqnarray*}
\left\vert D^{\mu }\Phi _{\nu }\left( x\right) \sqrt{f\left( x\right) }%
\right\vert &\lesssim &\frac{1}{r\left( x^{\nu }\right) ^{2}}\rho \left(
x\right) ^{2+\delta }\approx \rho \left( x\right) ^{\delta }\ , \\
\left\vert D^{\alpha }\Phi _{\nu }\left( x\right) \frac{D^{\beta }f\left(
x\right) }{\sqrt{f\left( x\right) }}\right\vert &\lesssim &\frac{1}{r\left(
x^{\nu }\right) }\frac{\rho \left( x\right) ^{3+2\delta }}{\rho \left(
x\right) ^{2+\delta }}\lesssim \rho \left( x\right) ^{\delta }\ , \\
\Phi _{\nu }\left( x\right) \frac{D^{\mu }f\left( x\right) }{2\sqrt{f\left(
x\right) }} &\lesssim &\frac{\rho \left( x\right) ^{2+2\delta }}{\rho \left(
x\right) ^{2+\delta }}=\rho \left( x\right) ^{\delta }\ , \\
\left\vert \Phi _{\nu }\left( x\right) \frac{D^{\alpha }f\left( x\right)
D^{\beta }f\left( x\right) }{4f\left( x\right) ^{\frac{3}{2}}}\right\vert
&\lesssim &\frac{\left( \rho \left( x\right) ^{3+2\delta }\right) ^{2}}{%
\left( \rho \left( x\right) ^{4+2\delta }\right) ^{\frac{3}{2}}}=\rho \left(
x\right) ^{\delta }\ ,
\end{eqnarray*}%
and we conclude that $\left\vert \nabla ^{2}f_{1}\left( x\right) \right\vert 
$\ is bounded by a multiple of $\rho \left( x\right) ^{\delta }$ for $x\in
B_{\nu }$. Finally, using the subproduct rule (\ref{subproduct}) for $\left[
\cdot \right] _{\alpha ,\delta }$, we obtain that $f_{1}\in C^{2,\delta }$\
uniformly in $\nu $.

Similarly, we have for a first order partial derivative $D^{\mu }\left(
f_{1}\left( x\right) ^{2}\right) $ and $x\in B_{\nu }$,%
\begin{eqnarray*}
\left\vert D^{\mu }\left( f_{1}\left( x\right) ^{2}\right) \right\vert
&=&\left\vert D^{\mu }\left( \Phi _{\nu }\left( x\right) ^{2}\right) f\left(
x\right) +\Phi _{\nu }\left( x\right) ^{2}D^{\mu }f\left( x\right)
\right\vert \\
&\lesssim &r\left( x^{\nu }\right) ^{-1}\rho \left( x\right) ^{4+2\delta
}+\rho \left( x\right) ^{3+2\delta }\lesssim \rho \left( x\right)
^{3+2\delta },
\end{eqnarray*}%
which is the case $\left\vert \mu \right\vert =1$ of%
\begin{equation*}
\left\vert D^{\mu }\left( f_{1}\left( x\right) ^{2}\right) \right\vert
\lesssim r\left( x^{\nu }\right) ^{4-\left\vert \mu \right\vert +2\delta },\
\ \ \ \ 0\leq \left\vert \mu \right\vert \leq 4,
\end{equation*}%
and the remaining cases $\left\vert \mu \right\vert =2,3,4$ are proved in
the same way. Finally, using the subproduct rule (\ref{subproduct}) for $%
\left[ \cdot \right] _{\alpha ,2\delta }$, together with the third line in (%
\ref{bit of work}), we obtain that $f_{1}^{2}\in C^{4,2\delta }$\ uniformly
in $\nu $.

\medskip

\textbf{Case II}: $f\left( x^{\nu }\right) <c\rho \left( x^{\nu }\right)
^{4+2\delta }$.

\medskip

In this case we have without loss of generality that $\partial
_{x_{n}}^{2}f\left( x^{\nu }\right) =\rho \left( x^{\nu }\right) ^{2+2\delta
}$, and hence that $\partial _{x_{n}}^{2}f\left( x\right) \geq \frac{1}{2}%
\rho \left( x^{\nu }\right) ^{2+2\delta }$ for $x\in B_{\nu }$, provided $c$
is chosen sufficiently small independent of $\nu $. Let us write $x=\left(
\xi ,x_{n}\right) $ and $x^{\nu }=\left( \xi ^{\nu },x_{n}^{\nu }\right) $.
Then for $\left\vert \xi -\xi ^{\nu }\right\vert <\frac{1}{\sqrt{2}}r_{\nu }=%
\frac{s}{\sqrt{2}}\rho \left( x^{\nu }\right) $, the function $%
x_{n}\rightarrow f\left( \xi ,x_{n}\right) $ has its second derivative
bounded below by $\frac{1}{2}\rho \left( x^{\nu }\right) ^{2+2\delta }$ on
the closed interval $\left[ a_{n}^{\nu },b_{n}^{\nu }\right] \equiv \left[
x_{n}^{\nu }-\frac{1}{\sqrt{2}}r_{\nu },x_{n}^{\nu }+\frac{1}{\sqrt{2}}%
r_{\nu }\right] $, and hence has a unique minimum point in $\left[
a_{n}^{\nu },b_{n}^{\nu }\right] $, say at $x_{n}=X\left( \xi \right) $. If
moreover, $c$ is chosen to be at most $\frac{s^{2}}{8}$, then the minimum is
actually attained at $x_{n}=X\left( \xi \right) $ in the \emph{open}
interval $\left( a_{n}^{\nu },b_{n}^{\nu }\right) $. Indeed, if not, say $%
f\left( \xi ,a_{n}^{\nu }\right) $ is the minimum of $f$ on the closed
interval $\left[ a_{n}^{\nu },b_{n}^{\nu }\right] $, then $\partial
_{x_{n}}f\left( \xi ,a_{n}^{\nu }\right) \geq 0$ as well as $f\left( \xi
,a_{n}^{\nu }\right) \geq 0$, and so Taylor's formula gives for an
intermediate point $c_{n}^{\nu }$ between $a_{n}^{\nu }$ and $x_{n}^{\nu }$,%
\begin{eqnarray*}
f\left( \xi ,x_{n}^{\nu }\right) &=&f\left( \xi ,a_{n}^{\nu }\right)
+\partial _{x_{n}}f\left( \xi ,a_{n}^{\nu }\right) \left( x_{n}^{\nu
}-a_{n}^{\nu }\right) +\partial _{x_{n}}^{2}f\left( \xi ,c_{n}^{\nu }\right) 
\frac{\left( x_{n}^{\nu }-a_{n}^{\nu }\right) ^{2}}{2} \\
&\geq &\frac{1}{2}\rho \left( x^{\nu }\right) ^{2+2\delta }\frac{\left( 
\frac{1}{\sqrt{2}}r_{\nu }\right) ^{2}}{2}=\frac{s^{2}}{8}\rho \left( x^{\nu
}\right) ^{4+2\delta }\geq c\rho \left( x^{\nu }\right) ^{4+2\delta },
\end{eqnarray*}%
contradicting the Case II assumption.

Set $F\left( \xi \right) \equiv f\left( \xi ,X\left( \xi \right) \right) $.
Then%
\begin{eqnarray*}
f\left( \xi ,x_{n}\right) &=&f\left( \xi ,X\left( \xi \right) \right)
+\partial _{x_{n}}f\left( \xi ,X\left( \xi \right) \right) \left[
x_{n}-X\left( \xi \right) \right] \\
&&+\int_{0}^{1}\left( 1-t\right) \partial _{x_{n}}^{2}f\left( \xi ,\left(
1-t\right) X\left( \xi \right) +tx_{n}\right) dt\ \frac{\left( x_{n}-X\left(
\xi \right) \right) ^{2}}{2} \\
&=&F\left( \xi \right) +H\left( \xi ,x_{n}\right) \ \left( x_{n}-X\left( \xi
\right) \right) ^{2}; \\
\text{where }F\left( \xi \right) &=&f\left( \xi ,X\left( \xi \right) \right)
, \\
\text{ and }H\left( \xi ,x_{n}\right) &\equiv &\frac{1}{2}\int_{0}^{1}\left(
1-t\right) \partial _{x_{n}}^{2}f\left( \xi ,\left( 1-t\right) X\left( \xi
\right) +tx_{n}\right) dt,
\end{eqnarray*}%
where $H\left( \xi ,x_{n}\right) $ satisfies%
\begin{eqnarray*}
H\left( \xi ,x_{n}\right) &\geq &\frac{1}{2}\int_{0}^{1}\left( 1-t\right) 
\frac{1}{2}\rho \left( x^{\nu }\right) ^{2+2\delta }dt=\frac{1}{8}\rho
\left( x^{\nu }\right) ^{2+2\delta }. \\
\left\vert \partial _{x_{n}}^{2}H\left( \xi ,x_{n}\right) \right\vert &\leq &%
\frac{1}{2}\int_{0}^{1}\left( 1-t\right) \left\vert \partial
_{x_{n}}^{4}f\left( \xi ,\left( 1-t\right) X\left( \xi \right)
+tx_{n}\right) \right\vert t^{2}dt \\
&\lesssim &\int_{0}^{1}\left( 1-t\right) t^{2}\rho \left( x^{\nu }\right)
^{2\delta }dt\lesssim \rho \left( x^{\nu }\right) ^{2\delta },
\end{eqnarray*}

Now we wish to bound $\left\vert \nabla ^{2}H\left( \xi ,x_{n}\right)
\right\vert $ by $C\rho \left( x^{\nu }\right) ^{2\delta }$. For this, we
first note that by parts (2) and (3) of Theorem \ref{IFT} applied to the
function $\partial _{x_{n}}f$, we have%
\begin{eqnarray*}
\left\vert \partial _{\xi _{i}}X\left( \xi \right) \right\vert &=&\left\vert 
\frac{\partial _{\xi _{i}}\partial _{x_{n}}f}{\partial _{x_{n}}^{2}f}%
\right\vert \lesssim \frac{\rho \left( x^{\nu }\right) ^{2+2\delta }}{\rho
\left( x^{\nu }\right) ^{2+2\delta }}=1, \\
\left\vert \partial _{\xi _{i}}\partial _{\xi _{j}}X\left( \xi \right)
\right\vert &\leq &\left\vert \frac{\partial _{\xi _{i}}\partial _{\xi
_{j}}\partial _{x_{n}}f}{\partial _{x_{n}}^{2}f}\right\vert +\left\vert 
\frac{\left( \partial _{\xi _{j}}\partial _{x_{n}}f\right) \left( \partial
_{\xi _{i}}\partial _{x_{n}}^{2}f\right) +\left( \partial _{\xi
_{i}}\partial _{x_{n}}f\right) \left( \partial _{\xi _{j}}\partial
_{x_{n}}^{2}f\right) }{\left( \partial _{x_{n}}^{2}f\right) ^{2}}\right\vert
+\left\vert \frac{\left( \partial _{\xi _{j}}\partial _{x_{n}}f\right)
\left( \partial _{\xi _{j}}\partial _{x_{n}}f\right) \left( \partial
_{x_{n}}^{3}f\right) }{\left( \partial _{x_{n}}^{2}f\right) ^{3}}\right\vert
\\
&\lesssim &\frac{\rho \left( x^{\nu }\right) ^{1+2\delta }}{\rho \left(
x^{\nu }\right) ^{2+2\delta }}+\frac{\rho \left( x^{\nu }\right) ^{2+2\delta
}\rho \left( x^{\nu }\right) ^{1+2\delta }}{\left( \rho \left( x^{\nu
}\right) ^{2+2\delta }\right) ^{2}}+\frac{\left( \rho \left( x^{\nu }\right)
^{2+2\delta }\right) ^{2}\rho \left( x^{\nu }\right) ^{1+2\delta }}{\left(
\rho \left( x^{\nu }\right) ^{2+2\delta }\right) ^{3}} \\
&\lesssim &\rho \left( x^{\nu }\right) ^{-1}.
\end{eqnarray*}%
Then we compute%
\begin{eqnarray*}
\partial _{x_{n}}\left[ \partial _{x_{n}}^{2}f\left( \xi ,\left( 1-t\right)
X\left( \xi \right) +tx_{n}\right) \right] &=&\partial _{x_{n}}^{3}f\left(
\xi ,\left( 1-t\right) X\left( \xi \right) +tx_{n}\right) t, \\
\nabla _{\xi }\left[ \partial _{x_{n}}^{2}f\left( \xi ,\left( 1-t\right)
X\left( \xi \right) +tx_{n}\right) \right] &=&\nabla _{\xi }\partial
_{x_{n}}^{2}f\left( \xi ,\left( 1-t\right) X\left( \xi \right) +tx_{n}\right)
\\
&&+\partial _{x_{n}}^{3}f\left( \xi ,\left( 1-t\right) X\left( \xi \right)
+tx_{n}\right) \left( 1-t\right) \nabla X\left( \xi \right) ,
\end{eqnarray*}%
and obtain that their moduli are bounded by%
\begin{eqnarray*}
\left\vert \partial _{x_{n}}\left[ \partial _{x_{n}}^{2}f\left( \xi ,\left(
1-t\right) X\left( \xi \right) +tx_{n}\right) \right] \right\vert &\lesssim
&\rho \left( x^{\nu }\right) ^{1+2\delta }, \\
\left\vert \nabla _{\xi }\left[ \partial _{x_{n}}^{2}f\left( \xi ,\left(
1-t\right) X\left( \xi \right) +tx_{n}\right) \right] \right\vert &\lesssim
&\rho \left( x^{\nu }\right) ^{1+2\delta }.
\end{eqnarray*}%
Similarly,%
\begin{eqnarray*}
\left\vert \nabla _{\xi }^{2}\left[ \partial _{x_{n}}^{2}f\left( \xi ,\left(
1-t\right) X\left( \xi \right) +tx_{n}\right) \right] \right\vert &\lesssim
&\rho \left( x^{\nu }\right) ^{2\delta }, \\
\left\vert \partial _{x_{n}}\nabla _{\xi }\left[ \partial
_{x_{n}}^{2}f\left( \xi ,\left( 1-t\right) X\left( \xi \right)
+tx_{n}\right) \right] \right\vert &\lesssim &\rho \left( x^{\nu }\right)
^{2\delta }, \\
\left\vert \partial _{x_{n}}^{2}\left[ \partial _{x_{n}}^{2}f\left( \xi
,\left( 1-t\right) X\left( \xi \right) +tx_{n}\right) \right] \right\vert
&\lesssim &\rho \left( x^{\nu }\right) ^{2\delta }.
\end{eqnarray*}%
Thus for $1\leq \left\vert \mu \right\vert \leq 2$, we have that%
\begin{equation*}
D^{\mu }H\left( \xi ,x_{n}\right) =\frac{1}{2}\int_{0}^{1}\left( 1-t\right)
D^{\mu }\left[ \partial _{x_{n}}^{2}f\left( \xi ,\left( 1-t\right) X\left(
\xi \right) +tx_{n}\right) \right] dt
\end{equation*}%
satisfies 
\begin{equation*}
\left\vert D^{\mu }H\left( \xi ,x_{n}\right) \right\vert \lesssim \rho
\left( x^{\nu }\right) ^{2-\left\vert \mu \right\vert +2\delta }.
\end{equation*}%
Thus $K\left( \xi ,x_{n}\right) \equiv H\left( \xi ,x_{n}\right) \left(
x_{n}-X\left( \xi \right) \right) ^{2}$ has a $C^{2,\delta }$ square root $%
G\left( \xi ,x_{n}\right) \left( x_{n}-X\left( \xi \right) \right) $ where $%
G\left( \xi ,x_{n}\right) \equiv \sqrt{H\left( \xi ,x_{n}\right) }$. Indeed, 
\begin{eqnarray*}
\nabla \left[ G\left( \xi ,x_{n}\right) \left( x_{n}-X\left( \xi \right)
\right) \right] &=&\nabla G\left( \xi ,x_{n}\right) \left( x_{n}-X\left( \xi
\right) \right) +G\left( \xi ,x_{n}\right) \nabla \left( x_{n}-X\left( \xi
\right) \right) \\
&=&\frac{1}{2}\frac{\nabla H\left( \xi ,x_{n}\right) }{H\left( \xi
,x_{n}\right) ^{\frac{1}{2}}}\left( x_{n}-X\left( \xi \right) \right)
+H\left( \xi ,x_{n}\right) ^{\frac{1}{2}}\nabla \left( x_{n}-X\left( \xi
\right) \right)
\end{eqnarray*}%
satisfies%
\begin{equation*}
\left\vert \nabla \left[ G\left( \xi ,x_{n}\right) \left( x_{n}-X\left( \xi
\right) \right) \right] \right\vert \lesssim \frac{\rho \left( x^{\nu
}\right) ^{1+2\delta }}{\rho \left( x^{\nu }\right) ^{1+\delta }}\rho \left(
x^{\nu }\right) +\rho \left( x^{\nu }\right) ^{1+\delta }\approx \rho \left(
x^{\nu }\right) ^{1+\delta },
\end{equation*}%
and for $\mu =\alpha +\beta $ with $\left\vert \alpha \right\vert
=\left\vert \beta \right\vert =1$,%
\begin{eqnarray*}
D^{\mu }\left[ G\left( \xi ,x_{n}\right) \left( x_{n}-X\left( \xi \right)
\right) \right] &=&D^{\mu }G\left( \xi ,x_{n}\right) \left( x_{n}-X\left(
\xi \right) \right) +2D^{\alpha }G\left( \xi ,x_{n}\right) D^{\beta }\left[
\left( x_{n}-X\left( \xi \right) \right) \right] \\
&&+G\left( \xi ,x_{n}\right) D^{\mu }\left[ \left( x_{n}-X\left( \xi \right)
\right) \right] \\
&=&\frac{1}{2}\left( \frac{D^{\mu }H\left( \xi ,x_{n}\right) }{H\left( \xi
,x_{n}\right) ^{\frac{1}{2}}}-\frac{\left\vert D^{\alpha }H\left( \xi
,x_{n}\right) \right\vert \left\vert D^{\beta }H\left( \xi ,x_{n}\right)
\right\vert }{H\left( \xi ,x_{n}\right) ^{\frac{3}{2}}}\right) \left(
x_{n}-X\left( \xi \right) \right) \\
&&+\frac{D^{\alpha }H\left( \xi ,x_{n}\right) }{H\left( \xi ,x_{n}\right) ^{%
\frac{1}{2}}}D^{\beta }\left( x_{n}-X\left( \xi \right) \right) \\
&&+H\left( \xi ,x_{n}\right) ^{\frac{1}{2}}D^{\mu }\left( x_{n}-X\left( \xi
\right) \right) ,
\end{eqnarray*}%
and so 
\begin{eqnarray*}
\left\vert D^{\mu }\left[ G\left( \xi ,x_{n}\right) \left( x_{n}-X\left( \xi
\right) \right) \right] \right\vert &\lesssim &\left( \frac{\rho \left(
x^{\nu }\right) ^{2\delta }}{\rho \left( x^{\nu }\right) ^{1+\delta }}+\frac{%
\left( \rho \left( x^{\nu }\right) ^{1+2\delta }\right) ^{2}}{\rho \left(
x^{\nu }\right) ^{3+3\delta }}\right) \rho \left( x^{\nu }\right) \\
&&+\frac{\rho \left( x^{\nu }\right) ^{1+2\delta }}{\rho \left( x^{\nu
}\right) ^{1+\delta }}+\rho \left( x^{\nu }\right) ^{\delta } \\
&\lesssim &\rho \left( x^{\nu }\right) ^{\delta },
\end{eqnarray*}%
and finally also 
\begin{equation*}
\left[ G\left( \xi ,x_{n}\right) \left( x_{n}-X\left( \xi \right) \right)
^{2}\right] _{2,\delta }\leq C.
\end{equation*}%
It thus follows that%
\begin{equation*}
\Phi _{\nu }\left( \xi ,x_{n}\right) ^{2}\left[ f\left( \xi ,x_{n}\right)
-F\left( \xi \right) \right] =\left\{ \Phi _{\nu }\left( \xi ,x_{n}\right)
\left( x_{n}-X\left( \xi \right) \right) G\left( \xi ,x_{n}\right) \right\}
^{2}
\end{equation*}%
has a $C^{2,\delta }$ square root $f_{2}\left( x\right) \equiv \Phi _{\nu
}\left( \xi ,x_{n}\right) \left( x_{n}-X\left( \xi \right) \right) G\left(
\xi ,x_{n}\right) $.

Similarly, the function 
\begin{equation}
f_{2}\left( x\right) ^{2}=\Phi _{\nu }\left( x\right) ^{2}\left\{ f\left(
\xi ,x_{n}\right) -F\left( \xi \right) \right\}   \label{f_2}
\end{equation}%
satisfies the estimates%
\begin{equation*}
\left\vert D^{\mu }\left( f_{2}\left( x\right) ^{2}\right) \right\vert
\lesssim r\left( x^{\nu }\right) ^{4-\left\vert \mu \right\vert +2\delta },\
\ \ \ \ 0\leq \left\vert \mu \right\vert \leq 4,
\end{equation*}%
upon using the estimates obtained below for $D^{\mu }F\left( \xi \right) $, $%
0\leq m\leq 4$, and this then leads to the conclusion that $f_{2}^{2}\in
C^{4,2\delta }$\ uniformly in $\nu $.

Now we note that all cases have been exhausted by the first line in (\ref%
{diff prov}).

\medskip

\textbf{Note}: \emph{This is a key juncture in the proof since we have thus
eliminated consideration of the difficult case in which }$\left\vert \nabla
^{4}f\left( x\right) \right\vert ^{\frac{1}{2\delta }}$\emph{\ is the
dominant term in the definition of }$\rho \left( x\right) $\emph{, and where
the implicit function is no longer decisive. However, see Bony \cite{Bon}
for how to proceed when }$x\in \mathbb{R}$\emph{\ is one-dimensional. Now we
will use the second line in (\ref{diff prov}) to show that the first
inequality is `inherited' by the function }$\kappa _{\ell ,m}$\emph{, but
with a smaller index }$\delta _{1}$.

\medskip

Thus altogether we have shown so far that 
\begin{equation*}
f\left( x\right) =\sum_{\nu =1}^{\infty }\Phi _{\nu }\left( x\right)
^{2}f\left( x\right) ,
\end{equation*}%
where for each $\nu $, after a rotation of coordinates depending on $\nu $,
either%
\begin{eqnarray*}
\Phi _{\nu }\left( x\right) ^{2}f\left( x\right) &=&g_{\nu }\left( x\right)
^{2}, \\
\text{where }g_{\nu } &\in &C^{2+\delta }\left( B_{\nu }\right) ,
\end{eqnarray*}%
or%
\begin{eqnarray*}
\Phi _{\nu }\left( x\right) ^{2}f\left( x\right) &=&\Phi _{\nu }\left(
x\right) ^{2}F\left( \xi \right) +\Phi _{\nu }\left( x\right) ^{2}H\left(
\xi ,x_{n}\right) \ \left( x_{n}-X\left( \xi \right) \right) ^{2} \\
&=&\Phi _{\nu }\left( x\right) ^{2}\kappa _{\nu }\left( x\right) +g_{\nu
}\left( x\right) ^{2}, \\
\text{where }\kappa _{\nu } &\in &C^{4+2\delta }\left( B_{\nu }\right) \text{
and }\Phi _{\nu },g_{\nu }\in C^{2+\delta }\left( B_{\nu }\right) .
\end{eqnarray*}

Finally we use the bounded overlap of the balls $B_{\nu }$ to write $\mathbb{%
N}=\dbigcup\limits_{\ell =1}^{N}A_{\ell }$ as a finite pairwise disjoint
union of index sets $A_{\ell }$ such that for each $\ell $\ the balls $%
\left\{ B_{\nu }\right\} _{\nu \in A_{\ell }}$ have pairwise disjoint
triples. Then we group the sum of all the functions into finitely many
functions $h_{\ell }\left( x\right) =\sum_{\nu \in A_{\ell }}\Phi _{\nu
}\left( x\right) ^{2}\kappa _{\nu }\left( x\right) $ and $g_{\ell }\left(
x\right) =\sum_{\nu \in B_{\ell }}g_{\nu }\left( x\right) ^{2}$ that satisfy
the conclusions of the theorem, save for the assertion that $\kappa _{\nu }$
satisfies (\ref{diff prov next}), to which we now turn.

In order to prove assertion (\textbf{ii}) of Theorem \ref{provisional}, we
suppose for the moment, and only for the sake of simplicity of calculation,
that the\ dimension is $n=2$ and the variable is $\left( x,y\right) \in 
\mathbb{R}^{2}$. For convenience in notation we will use the partial
derivative convention $f_{ijk}=\frac{\partial ^{2}}{\partial x_{i}\partial
x_{j}\partial x_{k}}$, etc., not to be confused with the function $f_{2}$ in
(\ref{f_2}).

Then for a function arising from Case II, which is the only case that is
nontrivial, we have that 
\begin{equation*}
\left( f_{22}\left( x_{\nu }\right) \right) ^{\frac{1}{2+2\delta }}\approx
\rho _{f;\delta}=\max \left\{ f^{\frac{1}{4+2\delta }},\left( \sup_{\theta \in 
\mathbb{S}^{n-1}}\left[ \partial _{\theta }^{2}f\right] _{+}\right) ^{\frac{1%
}{2+2\delta }},\left\vert \nabla ^{4}f\right\vert ^{\frac{1}{2\delta }%
}\right\} ,
\end{equation*}%
and with $X=\left( x,h\left( x\right) \right) $ and $F\left( x\right) \equiv
f\left( X\right) $,%
\begin{eqnarray}
f_{2}\left( X\right)  &=&0,  \label{F''} \\
h^{\prime }\left( x\right)  &=&-\frac{f_{12}\left( X\right) }{f_{22}\left(
X\right) },  \notag \\
F^{\prime }\left( x\right)  &=&f_{1}\left( X\right) +f_{2}\left( X\right)
h^{\prime }\left( x\right) =f_{1}\left( X\right) ,  \notag \\
F^{\prime \prime }\left( x\right)  &=&f_{11}\left( X\right) +f_{12}\left(
X\right) h^{\prime }\left( x\right) =f_{11}\left( X\right) -\frac{%
f_{12}\left( X\right) f_{12}\left( X\right) }{f_{22}\left( X\right) }. 
\notag
\end{eqnarray}%
If we use 
\begin{equation*}
\left( f_{22}\right) ^{\frac{1}{2+2\delta }}\approx \rho _{f;\delta }=\max
\left\{ f^{\frac{1}{4+2\delta }},\left( \sup_{\theta \in \mathbb{S}^{n-1}}%
\left[ \partial _{\theta }^{2}f\right] _{+}\right) ^{\frac{1}{2+2\delta }%
},\left\vert \nabla ^{4}f\right\vert ^{\frac{1}{2\delta }}\right\} ,
\end{equation*}%
together with the estimates 
\begin{equation*}
\left\vert \nabla ^{\ell }f\left( x\right) \right\vert \lesssim \rho
_{f;\delta }^{4+2\delta -\ell },\ \ \ \ \ \text{for }\ell \leq 4,
\end{equation*}%
we obtain%
\begin{equation*}
\left\vert F^{\prime \prime }\left( x\right) \right\vert \lesssim \rho
_{f;\delta }^{2+2\delta }+\frac{\left( \rho _{f;\delta }^{2+2\delta }\right)
^{2}}{\rho _{f;\delta }^{2+2\delta }}=2\rho _{f;\delta }^{2+2\delta }\approx
f_{22}\left( X\right) ,
\end{equation*}%
and hence the crucial inequality%
\begin{equation}
\sup_{\theta \in \mathbb{S}^{n-2}}\left[ \partial _{\theta }^{2}F\left(
x\right) \right] _{+}=\left[ F^{\prime \prime }\left( x\right) \right] _{+}=%
\left[ f_{11}\left( X\right) -\frac{f_{12}\left( X\right) ^{2}}{f_{22}\left(
X\right) }\right] _{+}\lesssim \sup_{\Theta \in \mathbb{S}^{n-1}}\left[
\partial _{\Theta }^{2}f\left( X\right) \right] _{+}\ ,  \label{crucial}
\end{equation}%
Thus we have both 
\begin{equation}
F\left( x\right) ^{\frac{1}{4+2\delta }}=f\left( X\right) ^{\frac{1}{%
4+2\delta }}\leq \rho _{f;\delta }\text{ and }\left\vert F^{\prime \prime
}\left( x\right) \right\vert ^{\frac{1}{2+2\delta }}\lesssim \rho _{f;\delta
}.  \label{have both}
\end{equation}

Suppose for the moment that we could show 
\begin{equation}
\left\vert F^{\prime \prime \prime \prime }\left( x\right) \right\vert ^{%
\frac{1}{2\delta }}\lesssim \rho _{f;\delta }\left( X\right) ,
\label{can show}
\end{equation}%
as well. Then since we are in Case II, and since (\ref{crucial}) holds, we
have%
\begin{equation*}
\rho _{F;\delta }\left( x\right) =\max \left\{ F\left( x\right) ^{\frac{1}{%
4+2\delta }},\left[ F^{\prime \prime }\left( x\right) \right] _{+}^{\frac{1}{%
2+2\delta }},\left\vert F^{\prime \prime \prime \prime }\left( x\right)
\right\vert ^{\frac{1}{2\delta }}\right\} \lesssim \rho _{f;\delta }\left(
X\right) \approx f_{22}\left( X\right) ^{\frac{1}{2+2\delta }}\ ,
\end{equation*}%
since $\sup_{\theta \in \mathbb{S}^{0}}\partial _{\theta }^{2}F\left(
x\right) =F^{\prime \prime }\left( x\right) $, and thus it would remain only
to obtain the estimates for $\left[ F^{\prime \prime }\left( x\right) \right]
_{+}$ and $\left\vert F^{\prime \prime \prime \prime }\left( x\right)
\right\vert $ in (\ref{diff prov}), i.e.%
\begin{equation*}
\left\vert F^{\prime \prime \prime \prime }\left( x\right) \right\vert \leq
Cf\left( x\right) ^{\frac{\delta _{1}}{2+\delta _{1}}}\text{ and }\left[
F^{\prime \prime }\left( x\right) \right] _{+}\leq Cf\left( x\right) ^{\eta
}.
\end{equation*}

We begin with the easy estimate using (\ref{crucial}) to obtain,%
\begin{equation*}
\left[ F^{\prime \prime }\left( x\right) \right] _{+}\lesssim\sup_{\Theta \in 
\mathbb{S}^{1}}\left[ \partial _{\Theta }^{2}f\left( X\right) \right]
_{+}\lesssim f\left( X\right) ^{\eta }=F\left( x\right) ^{\eta },
\end{equation*}%
upon using the assumption $f_{22}\left( X\right) \approx \sup_{\Theta \in 
\mathbb{S}^{n-1}}\left[ \partial _{\Theta }^{2}f\left( x\right) \right] _{+}$
together with the second inequality in (\ref{diff prov}). For $F^{\prime
\prime \prime \prime }\left( x\right) $ we use (\ref{F''}) to compute 
\begin{eqnarray*}
F^{\prime \prime \prime \prime }\left( x\right)  &=&\frac{d^{2}}{dx^{2}}%
\left( f_{11}\left( X\right) -\frac{f_{12}\left( X\right) ^{2}}{f_{22}\left(
X\right) }\right)  \\
&&-\left[ \frac{d^{2}}{dx^{2}}\frac{f_{2}\left( X\right) }{f_{22}\left(
X\right) }\right] \left( f_{112}\left( X\right) +2\frac{f_{12}\left(
X\right) f_{122}\left( X\right) }{f_{22}\left( X\right) }-\frac{f_{12}\left(
X\right) ^{2}f_{222}\left( X\right) }{f_{22}\left( X\right) ^{2}}\right)  \\
&&-\left[ \frac{d}{dx}\frac{f_{2}\left( X\right) }{f_{22}\left( X\right) }%
\right] \frac{d}{dx}\left( f_{112}\left( X\right) +2\frac{f_{12}\left(
X\right) f_{122}\left( X\right) }{f_{22}\left( X\right) }-\frac{f_{12}\left(
X\right) ^{2}f_{222}\left( X\right) }{f_{22}\left( X\right) ^{2}}\right)  \\
&&-\frac{f_{2}\left( X\right) }{f_{22}\left( X\right) }\frac{d^{2}}{dx^{2}}%
\left( f_{112}\left( X\right) +2\frac{f_{12}\left( X\right) f_{122}\left(
X\right) }{f_{22}\left( X\right) }-\frac{f_{12}\left( X\right)
^{2}f_{222}\left( X\right) }{f_{22}\left( X\right) ^{2}}\right)  \\
&=&\frac{d^{2}}{dx^{2}}\left( f_{11}\left( X\right) -\frac{f_{12}\left(
X\right) ^{2}}{f_{22}\left( X\right) }\right) ,
\end{eqnarray*}%
and a lengthy calculation, using only the chain rule, the product rule, the
estimates (\ref{bit of work}), and the equivalence $\rho _{f}\left( X\right)
\approx f_{22}\left( X\right) $ in force in Case II, shows that the final
line is dominated in modulus by $\rho _{f}\left( X\right) ^{2\delta }$.
Indeed,%
\begin{eqnarray*}
&&\frac{d^{2}}{dx^{2}}\left( f_{11}\left( X\right) -\frac{f_{12}\left(
X\right) ^{2}}{f_{22}\left( X\right) }\right)  \\
&=&\frac{d}{dx}\left( f_{111}\left( X\right) +f_{112}\left( X\right)
h^{\prime }\left( x\right) -\frac{f_{22}\left( X\right) 2f_{12}\left(
X\right) \left[ f_{112}\left( X\right) +f_{122}\left( X\right) h^{\prime
}\left( x\right) \right] -f_{12}\left( X\right) ^{2}\left[ f_{122}\left(
X\right) +f_{222}\left( X\right) h^{\prime }\left( x\right) \right] }{%
f_{22}\left( X\right) ^{2}}\right)  \\
&=&f_{1111}\left( X\right) -f_{1112}\left( X\right) \frac{f_{12}\left(
X\right) }{f_{22}\left( X\right) }-\frac{d}{dx}\left[ \frac{f_{112}\left(
X\right) f_{12}\left( X\right) }{f_{22}\left( X\right) }\right]  \\
&&-\frac{d}{dx}\frac{f_{22}\left( X\right) 2f_{12}\left( X\right) \left[
f_{112}\left( X\right) -\frac{f_{122}\left( X\right) f_{12}\left( X\right) }{%
f_{22}\left( X\right) }\right] -f_{12}\left( X\right) ^{2}\left[
f_{122}\left( X\right) -\frac{f_{222}\left( X\right) f_{12}\left( X\right) }{%
f_{22}\left( X\right) }\right] }{f_{22}\left( X\right) ^{2}},
\end{eqnarray*}%
which we claim is dominated by $C\rho _{f}\left( X\right) ^{2\delta }$ upon
using the estimates $\left\vert \nabla ^{\ell }f\left( x\right) \right\vert
\lesssim \rho \left( x\right) ^{4+2\delta -\ell }$. For example, we compute
that the third term on the right hand side above equals%
\begin{eqnarray*}
\frac{d}{dx}\left[ \frac{f_{112}\left( X\right) f_{12}\left( X\right) }{%
f_{22}\left( X\right) }\right]  &=&\frac{\left[ f_{1112}\left( X\right)
+f_{1122}\left( X\right) h^{\prime }\left( x\right) \right] f_{12}\left(
X\right) }{f_{22}\left( X\right) } \\
&&+\frac{f_{112}\left( X\right) \left[ f_{112}\left( X\right) +f_{122}\left(
X\right) h^{\prime }\left( x\right) \right] }{f_{22}\left( X\right) } \\
&&-\frac{\left[ f_{1112}\left( X\right) +f_{1122}\left( X\right) h^{\prime
}\left( x\right) \right] f_{12}\left( X\right) }{f_{22}\left( X\right) ^{2}}%
\left[ f_{122}\left( X\right) +f_{222}\left( X\right) h^{\prime }\left(
x\right) \right] ,
\end{eqnarray*}%
and the estimates (\ref{bit of work}) then easily show both 
\begin{eqnarray*}
\left\vert \frac{d}{dx}\left[ \frac{f_{112}\left( X\right) f_{12}\left(
X\right) }{f_{22}\left( X\right) }\right] \right\vert  &\leq &C\rho \left(
X\right) ^{2\delta }, \\
\text{and }\left[ \frac{f_{112}f_{12}}{f_{22}}\right] _{1,2\delta }\left(
X\right)  &\leq &C.
\end{eqnarray*}

The remaining estimates for $\left\vert F^{\prime \prime \prime \prime
}\left( x\right) \right\vert $ are similar and left for the reader.

Thus we have completed the proof of (\ref{can show}), and now we can use the
second inequality in (\ref{diff prov}) to obtain%
\begin{equation*}
\left\vert \nabla ^{4}F\left( x\right) \right\vert ^{\frac{1}{2\delta _{1}}%
}\lesssim \rho _{f;\delta }\left( X\right) ^{\frac{2\delta }{2\delta _{1}}%
}\lesssim \left( \sup_{\Theta \in \mathbb{S}^{n-1}}\left[ \partial _{\Theta
}^{2}f\left( x\right) \right] _{+}\right) ^{\frac{1}{2+2\delta }\frac{%
2\delta }{2\delta _{1}}}\lesssim f\left( X\right) ^{\eta \frac{1}{2+2\delta }%
\frac{2\delta }{2\delta _{1}}}=F\left( x\right) ^{\frac{1}{4+2\delta _{1}}},
\end{equation*}%
where the final equality follows from the definition of $\delta _{1}$, i.e. $%
\frac{\delta _{1}}{2+\delta _{1}}=\eta \frac{\delta }{1+\delta }$.

The analogous derivative calculations in higher dimensions $n>2$ are mostly
a straightforward exercise in extending notation. For example, if we write $%
x=\left( x^{\prime },x_{n}\right) \in \mathbb{R}^{n-1}\times \mathbb{R}$ and
suppose $f_{nn}\left( x,x_{n}\right) \approx \rho _{f;\delta }\left(
x\right) ^{2+2\delta }$, then we can use the Implicit Function Theorem to
locally define $h\left( x^{\prime }\right) $ by $f_{n}\left( X\right)
=f_{n}\left( x^{\prime },h\left( x^{\prime }\right) \right) =0$. Then with $%
F\left( x^{\prime }\right) \equiv f\left( x^{\prime },h\left( x^{\prime
}\right) \right) $, we have%
\begin{eqnarray*}
f_{n}\left( X\right) &=&f_{n}\left( x^{\prime },h\left( x^{\prime }\right)
\right) =0, \\
\frac{\partial h}{\partial x_{i}} &=&-\frac{\frac{\partial H}{\partial x_{i}}%
\left( x^{\prime },h\left( x^{\prime }\right) \right) }{\frac{\partial H}{%
\partial x_{n}}\left( x^{\prime },h\left( x^{\prime }\right) \right) },\ \ \
\ \ \text{for }1\leq i\leq n-1, \\
\frac{\partial F}{\partial x_{i}}\left( x^{\prime }\right) &=&f_{i}\left(
X\right) +f_{n}\left( X\right) \frac{\partial h}{\partial x_{i}}\left(
x^{\prime }\right) =f_{i}\left( X\right) , \\
\frac{\partial ^{2}F}{\partial x_{i}^{2}}\left( x^{\prime }\right)
&=&f_{ii}\left( X\right) +f_{in}\left( X\right) h^{\prime }\left( x^{\prime
}\right) =f_{ii}\left( X\right) -\frac{f_{in}\left( X\right) f_{in}\left(
X\right) }{f_{nn}\left( X\right) },
\end{eqnarray*}%
and hence, after a rotation in $x^{\prime }$, the crucial inequality%
\begin{equation*}
\sup_{\Theta \in \mathbb{S}^{n-1}}\left[ \partial _{\Theta }^{2}F\left(
x^{\prime }\right) \right] _{+}=\left[ \frac{\partial ^{2}F}{\partial
x_{i}^{2}}\left( x^{\prime }\right) \right] _{+}=\left[ f_{ii}\left(
X\right) -\frac{f_{in}\left( X\right) ^{2}}{f_{nn}\left( X\right) }\right]
_{+}\leq \left[ f_{ii}\left( X\right) \right] _{+}\leq \sup_{\Theta \in 
\mathbb{S}^{n-1}}\left[ \partial _{\Theta }^{2}f\left( x\right) \right]
_{+}\ .
\end{equation*}%
The $n$-dimensional proof now proceeds as in the two-dimensional case.
\end{proof}

\subsection{A two dimensional SOS decomposition}

Here we sketch the proof of a decomposition into a sum of squares of $%
C^{2,\delta }\left( \mathbb{R}^{2}\right) $ functions in the plane, in which
the second differential inequality in (\ref{diff prov}) can be dropped. In
dimension $n\geq 5$, this second inequality cannot be dropped as shown in
Remark \ref{necc diff ineq} below.

\begin{theorem}
\label{main 2D}Suppose $0<\delta <\frac{1}{2}$ and that $f\in C^{4,2\delta
}\left( \mathbb{R}^{2}\right) $ satisfies the pointwise inequality%
\begin{equation}
\left\vert \nabla ^{4}f\left( x\right) \right\vert \leq f\left( x\right) ^{%
\frac{\delta }{2+\delta }}.  \label{crucial point}
\end{equation}%
Then $f=\sum_{\ell =1}^{N}g_{\ell }^{2}$ can be decomposed as a finite sum
of squares of functions $g_{\ell }\in C^{2+\delta }\left( \mathbb{R}%
^{2}\right) $ where 
\begin{eqnarray*}
\left\vert D^{\alpha }g_{\ell }\left( x\right) \right\vert &\leq &C\rho
\left( x\right) ^{2+\delta -\left\vert \alpha \right\vert },\ \ \ \ \ 0\leq
\left\vert \alpha \right\vert \leq 2, \\
\left[ g_{\ell }\right] _{\alpha ,\delta }\left( x\right) &\leq &C,\ \ \ \ \
\left\vert \alpha \right\vert =2.
\end{eqnarray*}%
and%
\begin{eqnarray*}
\left\vert D^{\alpha }g_{\ell }^{2}\left( x\right) \right\vert &\leq &C\rho
\left( x\right) ^{4+2\delta -\left\vert \alpha \right\vert },\ \ \ \ \ 0\leq
\left\vert \alpha \right\vert \leq 4, \\
\left[ g_{\ell }^{2}\right] _{\alpha ,2\delta }\left( x\right) &\leq &C,\ \
\ \ \ \left\vert \alpha \right\vert =4.
\end{eqnarray*}
\end{theorem}

\begin{proof}
The pointwise inequality on $\left\vert \nabla ^{4}f\right\vert $ shows that 
$\left\vert \nabla ^{4}f\left( x\right) \right\vert ^{\frac{1}{2\delta }%
}\leq f\left( x\right) ^{\frac{1}{4+2\delta }}$, and hence 
\begin{eqnarray*}
\rho \left( x\right)  &\equiv &\max \left\{ f\left( x\right) ^{\frac{1}{%
4+2\delta }},\left( \sup_{\theta \in \mathbb{S}^{n-1}}\left[ \partial
_{\theta }^{2}f\left( x\right) \right] _{+}\right) ^{\frac{1}{2+2\delta }%
},\left\vert \nabla ^{4}f\left( x\right) \right\vert ^{\frac{1}{2\delta }%
}\right\}  \\
&=&\max \left\{ f\left( x\right) ^{\frac{1}{4+2\delta }},\left( \sup_{\theta
\in \mathbb{S}^{n-1}}\left[ \partial _{\theta }^{2}f\left( x\right) \right]
_{+}\right) ^{\frac{1}{2+2\delta }}\right\} ,
\end{eqnarray*}%
by (\ref{crucial point}). Now the H\"{o}lder argument of Bony \cite[%
Subsection 5.1]{Bon} proves the result since the function $F\left( x^{\prime
}\right) =f\left( x^{\prime },X\left( x^{\prime }\right) \right) $ that
arises in Case II of the argument is in $C^{4,2\delta }\left( \mathbb{R}%
\right) $, and so Bony's one-dimensional result shows that $F$ can be
written as a sum of two squares of $C^{2,\delta }\left( \mathbb{R}\right) $
functions. Now we proceed with the Fefferman-Phong argument as modified by
Bony, and along the lines of the argument used in the proof of the
provisional Theorem \ref{provisional} above.
\end{proof}

\subsection{A higher dimensional SOS decomposition}

Here we prove our main decomposition of a smooth nonnegative function into a
sum of squares of $C^{2,\delta }\left( \mathbb{R}^{n}\right) $ functions in
arbitrary dimension, but restricted to elliptical flat smooth functions that
satisfy certain differential inequalities, that are in turn implied by
assuming $f$ is $\omega _{s}$-monotone for appropriate $0<s<1$.

\begin{theorem}
\label{efs eps}Suppose $0<\delta ,\eta <\frac{1}{2}$, that $f$ is a $%
C^{4,2\delta }$ function on $\mathbb{R}^{n}$, and that $\rho \left( x\right) 
$ is as defined in the formula (\ref{def rho}) above. Define $\delta _{n-1}$
recursively by $\delta _{0}=\delta $ and%
\begin{equation}
\frac{\delta _{k+1}}{2+\delta _{k+1}}=\eta \frac{\delta _{k}}{1+\delta _{k}}%
,\ \ \ \ \ 0\leq k\leq n-2.  \label{recurse}
\end{equation}

\begin{enumerate}
\item If $f$ satisfies both of the differential inequalities in (\ref{diff
prov}), i.e.%
\begin{equation*}
\left\vert \nabla ^{4}f\left( x\right) \right\vert \leq Cf\left( x\right) ^{%
\frac{\delta }{2+\delta }}\text{ and }\sup_{\theta \in \mathbb{S}^{n-1}}%
\left[ \partial _{\theta }^{2}f\left( x\right) \right] _{+}\leq Cf\left(
x\right) ^{\eta },
\end{equation*}%
then $f=\sum_{\ell =1}^{N}g_{\ell }^{2}$ can be decomposed as a finite sum
of squares of functions $g_{\ell }\in C^{2+\delta _{n-1}}\left( \mathbb{R}%
^{2}\right) $ where 
\begin{eqnarray*}
\left\vert D^{\alpha }g_{\ell }\left( x\right) \right\vert &\leq &C\rho
_{f;\delta }\left( x\right) ^{2+\delta _{n-1}-\left\vert \alpha \right\vert
},\ \ \ \ \ 0\leq \left\vert \alpha \right\vert \leq 2, \\
\left[ g_{\ell }\right] _{\alpha ,\delta _{n-1}}\left( x\right) &\leq &C,\ \
\ \ \ \left\vert \alpha \right\vert =2.
\end{eqnarray*}%
and%
\begin{eqnarray*}
\left\vert D^{\alpha }g_{\ell }^{2}\left( x\right) \right\vert &\leq &C\rho
_{f;\delta }\left( x\right) ^{4+2\delta _{n-1}-\left\vert \alpha \right\vert
},\ \ \ \ \ 0\leq \left\vert \alpha \right\vert \leq 4, \\
\left[ g_{\ell }^{2}\right] _{\alpha ,2\delta _{n-1}}\left( x\right) &\leq
&C,\ \ \ \ \ \left\vert \alpha \right\vert =4.
\end{eqnarray*}%
The inequality $\rho _{f;\delta }\left( x\right) \leq Cf\left( x\right)
^{\min \left\{ \frac{1}{4+2\delta },\frac{\eta }{2+2\delta }\right\} }$ can
be used to further dominate these derivatives by positive powers of $f\left(
x\right) $.

\item In particular, the inequalities (\ref{diff prov}) hold provided $f$ is
also flat, smooth and $\omega _{s}$-monotone for some $0<s<1$ satisfying%
\begin{equation}
s>\max \left\{ \sqrt[4]{\frac{\delta }{2+\delta }},\sqrt{\eta }\right\} .
\label{eps delta}
\end{equation}
\end{enumerate}
\end{theorem}

\begin{proof}
For (1) use induction on dimension together with Theorem \ref{provisional}.
For (2) use inequality (\ref{first high'}) in Theorem \ref{s'^m}, i.e.%
\begin{equation*}
\left\vert \nabla ^{m}f\left( x\right) \right\vert \leq C_{s^{\prime
},s}f\left( x\right) ^{\left( s^{\prime }\right) ^{m}},\ \ \ \ \ \text{for }%
0<s^{\prime }<s.
\end{equation*}%
Thus we obtain both $\left\vert \nabla ^{4}f\left( x\right) \right\vert \leq
Cf\left( x\right) ^{\frac{\delta }{2+\delta }}$ and $\sup_{\theta \in 
\mathbb{S}^{n-1}}\left[ \partial _{\theta }^{2}f\left( x\right) \right]
_{+}\leq Cf\left( x\right) ^{\eta }$ if we take $s>s^{\prime }\geq \max
\left\{ \sqrt[4]{\frac{\delta }{2+\delta }},\sqrt{\eta }\right\} $.
\end{proof}

\begin{remark}
With $s_{k}\equiv \frac{\delta _{k}}{2+\delta _{k}}$ and $\delta _{0}=\delta 
$, we have from (\ref{recurse}) that%
\begin{eqnarray*}
s_{k+1} &=&\frac{\delta _{k+1}}{2+\delta _{k+1}}=\eta \frac{\delta _{k}}{%
1+\delta _{k}}=\eta \frac{2+\delta _{k}}{1+\delta _{k}}\frac{\delta _{k}}{%
2+\delta _{k}}=\eta \frac{2+\delta _{k}}{1+\delta _{k}}s_{k}, \\
&&\text{i.e. }\frac{s_{k+1}}{s_{k}}=\eta \frac{2+\delta _{k}}{1+\delta _{k}}%
=\eta \left( 1+\frac{1}{1+\delta _{k}}\right) \text{,}
\end{eqnarray*}%
and since $0<\delta _{k+1}\leq \delta _{k}\leq \delta \leq \frac{1}{2}$, we
have the crude estimate%
\begin{equation*}
\left( \frac{5}{3}\eta \right) ^{n-1}\leq \frac{s_{n-1}}{s_{0}}\leq \left(
2\eta \right) ^{n-1}.
\end{equation*}%
Using $s_{k}\equiv \frac{\delta _{k}}{2+\delta _{k}}$, this becomes%
\begin{equation*}
\frac{4}{5}\left( \frac{5}{3}\eta \right) ^{n-1}\leq \frac{2+\delta _{n-1}}{%
2+\delta }\left( \frac{5}{3}\eta \right) ^{n-1}\leq \frac{\delta _{n-1}}{%
\delta }\leq \frac{2+\delta _{n-1}}{2+\delta }\left( 2\eta \right)
^{n-1}\leq \frac{5}{4}\left( 2\eta \right) ^{n-1},
\end{equation*}%
which shows that $g_{\ell }\in C^{2,\delta _{n-1}}$ where 
\begin{equation*}
\frac{4}{5}\left( \frac{5}{3}\eta \right) ^{n-1}\delta \leq \delta
_{n-1}\leq \frac{5}{4}\left( 2\eta \right) ^{n-1}\delta .
\end{equation*}%
In particular we see that $\delta _{n-1}$ is much smaller than $\delta $
when $\eta $ is much smaller than $\frac{1}{2}$.
\end{remark}

Given $f$ flat, smooth and $\omega _{s}$-monotone for some $s<1$, and $%
0<\delta ,\eta <1$, we will now see that the choice $\eta =\sqrt{\frac{%
\delta }{2+\delta }}$, i.e. $\delta =\frac{2\eta ^{2}}{1-\eta ^{2}}$, in
Theorem \ref{efs eps}\ gives the following corollary.

\begin{corollary}
\label{efs eps cor}For $0<s<\frac{1}{\sqrt[4]{5}}$, set $\delta \left(
s\right) =\frac{2s^{4}}{1-s^{4}}\in \left( 0,\frac{1}{2}\right) $,
equivalently $s^{4}=\frac{\delta \left( s\right) }{2+\delta \left( s\right) }
$. Suppose $f\in C^{4,2\delta \left( s\right) }\left( \mathbb{R}^{n}\right) $
is nonnegative, flat, smooth and $\omega _{s}$-monotone. Then for any $0<t<s$%
, $f$ can be decomposed as a finite sum of squares of $C^{2,\delta \left(
t\right) _{n-1}}\left( \mathbb{R}^{n}\right) $ functions where $\delta
\left( t\right) _{n-1}$ is defined recursively by (\ref{recurse}) with $%
\delta _{0}=\delta \left( t\right) $ and $\eta =t^{2}$.
\end{corollary}

\begin{proof}
Theorem \ref{s'^m} shows that $f$ satisfies the differential inequalities%
\begin{equation*}
\left\vert \nabla ^{4}f\left( x\right) \right\vert \leq Cf\left( x\right)
^{t^{4}}\text{ and }\sup_{\theta \in \mathbb{S}^{n-1}}\left[ \partial
_{\theta }^{2}f\left( x\right) \right] _{+}\leq Cf\left( x\right) ^{t^{2}},\
\ \ \ \ \text{for }0<t<s,
\end{equation*}%
which imply that 
\begin{equation*}
\left\vert \nabla ^{4}f\left( x\right) \right\vert \leq Cf\left( x\right) ^{%
\frac{\delta \left( t\right) }{2+\delta \left( t\right) }}\text{ and }%
\sup_{\theta \in \mathbb{S}^{n-1}}\left[ \partial _{\theta }^{2}f\left(
x\right) \right] _{+}\leq Cf\left( x\right) ^{\eta \left( t\right) },
\end{equation*}%
for $t=\sqrt{\eta \left( t\right) }=\sqrt[4]{\frac{\delta \left( t\right) }{%
2+\delta \left( t\right) }}$, and in particular $\eta \left( t\right) =\sqrt{%
\frac{\delta \left( t\right) }{2+\delta \left( t\right) }}$. Thus part (1)
of the theorem shows that $f$ can be decomposed as a finite sum of squares
of $C^{2,\delta \left( t\right) _{n-1}}\left( \mathbb{R}^{2}\right) $
functions where $\delta \left( t\right) _{n-1}$ is defined recursively by (%
\ref{recurse}) with $\delta _{0}=\delta \left( t\right) $ and $\eta =t^{2}=%
\sqrt{\frac{\delta \left( t\right) }{2+\delta \left( t\right) }}$.
\end{proof}

\begin{remark}
The inequalities (\ref{diff prov}) in Theorem \ref{efs eps} also hold if the
smoothness assumption on $f$ is relaxed to $f\in C^{k}$, provided that $s$
is replaced by $s-\frac{C}{k}$ in (\ref{eps delta}) for a sufficiently large
constant $C$. See the proof of Theorem \ref{s'^m}.
\end{remark}

\begin{remark}
\label{necc diff ineq}The counterexamples in \cite{BoBrCoPe} show that the
differential inequalities in (\ref{diff prov}) cannot both be dropped. More
precisely, fix $\delta >0$. If we set $\beta =\delta _{n-1}$ in part (1) of
Theorem \ref{log counter'}, then the inequality $s<\delta _{n-1}$ implies
that there is an elliptical flat smooth $\omega _{s}$-monotone function $f$
that cannot be written as a finite sum of squares of $C^{2,\delta _{n-1}}$
functions, contradicting part (2) of Theorem \ref{efs eps} with $\eta =\sqrt{%
\frac{\delta }{2+\delta }}$.
\end{remark}

The utility of Corollary \ref{efs eps cor} for our purposes lies in the fact
that given any $0<s<\frac{1}{\sqrt[4]{5}}$, we can find $0<\delta <1$ so
small that $f$ can be decomposed as a finite sum of squares of $C^{2,\delta
}\left( \mathbb{R}^{2}\right) $ functions. We also conjecture that there
exists an extension of Theorem \ref{efs eps} to $C^{2m,2\delta }$ functions $%
f$ on $\mathbb{R}^{n}$, where the control distance that is used in the proof
is%
\begin{equation}
\rho _{f;\delta }\left( x\right) \equiv \max \left\{ f\left( x\right) ^{%
\frac{1}{2m+2\delta }},\left( \sup_{\theta \in \mathbb{S}^{n-1}}\left[
\partial _{\theta }^{2}f\left( x\right) \right] _{+}\right) ^{\frac{1}{%
2m-2+2\delta }},\left\vert \nabla ^{4}f\left( x\right) \right\vert ^{\frac{1%
}{2m-4+2\delta }},...,\left\vert \nabla ^{2m}f\left( x\right) \right\vert ^{%
\frac{1}{2\delta }}\right\} ,  \label{rho m}
\end{equation}%
and where the differential inequalities imposed include $\left\vert \nabla
^{2\left( m-p\right) }f\left( x\right) \right\vert \leq Cf\left( x\right) ^{%
\frac{p+\delta }{m+\delta }}$ for $0\leq p\leq m-2$. Note that Theorem \ref%
{s'^m} gives $\left\vert \nabla ^{2\left( m-1\right) }f\left( x\right)
\right\vert \lesssim f\left( x\right) ^{\left( s^{\prime }\right) ^{2\left(
m-1\right) }}$, and if we wish to obtain the case $p=1$ of the previous
inequalities from this, we need to dominate the right hand side $f\left(
x\right) ^{\left( s^{\prime }\right) ^{2\left( m-1\right) }}$ by $Cf\left(
x\right) ^{\frac{1+\delta }{m+\delta }}$. But this requires $s>\sqrt[2\left(
m-1\right) ]{\frac{1+\delta }{m+\delta }}$, which forces $s$ closer and
closer to $1$ as $m\rightarrow \infty $ since $\lim_{m\rightarrow \infty }m^{%
\frac{1}{2m-1}}=1$. As a consequence, such an extension of Theorem \ref{efs
eps} to $C^{2m,2\delta }$ functions would not be useful for hypoellipticity
in the third paper \cite{KoSa3} of this series, and so we will not pursue
the conjecture here.

\section{Counterexamples}

Here we begin by constructing an example of an elliptical flat smooth
function $f$ on $B_{\mathbb{R}^{4}}\left( 0,1\right) \times \left(
-1,1\right) $ that cannot be written as a finite sum of squares of $%
C^{2,\beta }$ functions for $\beta >0$. Even more, we prove the following
result that answers a question in \cite[Remark 1.4]{BoBrCoPe}.

\begin{theorem}
Given any modulus of continuity $\omega $, there is an elliptical flat
smooth function $f$ on $B_{\mathbb{R}^{4}}\left( 0,1\right) \times \left(
-1,1\right) $ that cannot be written as a finite sum of squares of $%
C^{2,\omega }$ functions.
\end{theorem}

Then we investigate the connection between $\omega _{s}$-montonicity and
these counterexamples. To construct our counterexample we modify the example
in \cite[Theorem 1.2 (d)]{BoBrCoPe} by adding an additional term $\eta
\left( t,r\right) $, $r=\left\vert \left( w,x,y,z\right) \right\vert $, and
to prepare for this we modify the construction of the function $C\left(
\varepsilon \right) $ appearing in their argument. But first recall the
following lemma, where%
\begin{equation*}
L\left( w,x,y,z\right) \equiv w^{4}+x^{2}y^{2}+y^{2}z^{2}+z^{2}x^{2}-2wxyz,\
\ \ \ \ \left( w,x,y,z\right) \in \mathbb{R}^{4}.
\end{equation*}%
and for a modulus of continuity $\omega $ and $h$ defined on the unit ball $%
B_{\mathbb{R}^{4}}\left( 0,1\right) $ in $\mathbb{R}^{4}$, 
\begin{equation*}
\left\Vert h\right\Vert _{C^{2,\omega }\left( B_{\mathbb{R}^{4}}\left(
0,1\right) \right) }\equiv \sum_{k=0}^{2}\left\Vert \nabla ^{k}h\right\Vert
_{L^{\infty }\left( B_{\mathbb{R}^{4}}\left( 0,1\right) \right)
}+\sup_{W,W^{\prime }\in B_{\mathbb{R}^{4}}\left( 0,1\right) }\frac{%
\left\vert \nabla ^{2}h\left( W\right) -\nabla ^{2}h\left( W^{\prime
}\right) \right\vert }{\omega \left( \left\vert W-W^{\prime }\right\vert
\right) }.
\end{equation*}

\begin{lemma}[{\protect\cite[Theorem 1.2 (d)]{BoBrCoPe}}]
\label{g control}Let $\omega $ be a modulus of continuity. For every $\nu
\in \mathbb{N}$ there is a decreasing function $\mathcal{C}_{\nu }:\left(
0,1\right) \rightarrow \left( 0,\infty \right) $ such that%
\begin{eqnarray*}
\lim_{\tau \searrow 0}\mathcal{C}_{\nu }\left( \tau \right) &=&\infty , \\
\sum_{j=1}^{\nu }\left\Vert g_{j,\tau }\right\Vert _{C^{2,\omega }\left( B_{%
\mathbb{R}^{4}}\left( 0,1\right) \right) } &\geq &\mathcal{C}_{\nu }\left(
\tau \right) ,
\end{eqnarray*}%
whenever $\left\{ g_{j,\tau }\right\} _{j=1}^{\nu }\subset C^{2,\omega
}\left( B_{\mathbb{R}^{4}}\left( 0,1\right) \right) $ satisfy%
\begin{equation}
L\left( w,x,y,z\right) +\tau =\sum_{j=1}^{\nu }g_{j,\tau }\left(
w,x,y,z\right) ^{2},\ \ \ \ \ \text{\ for }\left( w,x,y,z\right) \in B_{%
\mathbb{R}^{4}}\left( 0,1\right) .  \label{g sat}
\end{equation}
\end{lemma}

\begin{proof}
Fix $\nu \in \mathbb{N}$. Suppose, in order to derive a contradiction, that
for all $0<\tau <1$, there are $\nu $ functions $\left\{ g_{j,\tau }\right\}
_{j=1}^{\nu }\subset C^{2,\omega }\left( B_{\mathbb{R}^{4}}\left( 0,1\right)
\right) $ satisfying (\ref{g sat}) and $\left\Vert g_{j,\tau }\right\Vert
_{C^{2,\omega }}\leq C$, for a constant $C$ independent of $\tau $. Then the
collection of functions $\left\{ g_{j,\tau }\right\} _{1\leq j\leq \nu
,0<\tau <1}$ is bounded in $C^{2,\omega }\left( B_{\mathbb{R}^{4}}\left(
0,1\right) \right) $, and hence compact in $C^{2}\left( B_{\mathbb{R}%
^{4}}\left( 0,1\right) \right) $. Thus there is a decreasing sequence $%
\left\{ \tau _{n}\right\} _{n=1}^{\infty }\subset \left( 0,1\right) $ and a
set of $\nu $ functions $\left\{ g_{j}\right\} _{j=1}^{\nu }\subset
C^{2}\left( B_{\mathbb{R}^{4}}\left( 0,1\right) \right) $ such that $%
g_{j,\tau _{n}}\rightarrow g_{j}$ in $C^{2}\left( B_{\mathbb{R}^{4}}\left(
0,1\right) \right) $ for each $1\leq j\leq \nu $, and it follows from (%
\ref{g sat}) that%
\begin{equation*}
L\left( w,x,y,z\right) =\sum_{j=1}^{\nu }g_{j}\left( w,x,y,z\right) ^{2},\ \
\ \ \ \text{\ for }\left( w,x,y,z\right) \in B_{\mathbb{R}^{4}}\left(
0,1\right) ,
\end{equation*}%
contradicting \cite[Theorem 1.2 (c)]{BoBrCoPe}.
\end{proof}

Now we construct a lower bound $\mathcal{C}\left( \tau \right) $ for $%
\left\{ \mathcal{C}_{\nu }\left( \tau \right) \right\} _{\nu =1}^{\infty }$
as in \cite{BoBrCoPe}. First, use Lemma \ref{g control} to choose a strictly
decreasing sequence $\left\{ \tau _{n}\right\} _{n=1}^{\infty }\subset
\left( 0,1\right) $ such that%
\begin{equation*}
\mathcal{C}_{\nu }\left( \tau \right) \geq n^{2},\ \ \ \ \ \text{for }0<\tau
\leq \tau _{n}\text{ and }\nu \leq n,
\end{equation*}%
and then%
\begin{equation*}
\mathcal{C}\left( \tau \right) \equiv \sum_{n=1}^{\infty }n\mathbf{1}_{\left[
\tau _{n+1},\tau _{n}\right) }\left( \tau \right) ,
\end{equation*}%
so that we have%
\begin{equation}
\lim_{\tau \searrow 0}\mathcal{C}\left( \tau \right) =\infty \text{ and }%
\lim_{\tau \searrow 0}\frac{\mathcal{C}_{\nu }\left( \tau \right) }{\mathcal{%
C}\left( \tau \right) }=\infty ,\ \ \ \ \ \text{for all }\nu \in \mathbb{N}.
\label{ratio}
\end{equation}%
It is clear that we can now modify $\mathcal{C}$ to be strictly decreasing
and still satisfy (\ref{ratio}).

Now let $\varphi :\left( 0,1\right) \rightarrow \left( 0,1\right) $ be a
strictly increasing elliptical flat smooth function on $\left( 0,1\right) $,
and with $r=\left\vert \left( w,x,y,z\right) \right\vert $ define%
\begin{equation}
f\left( w,x,y,z,t\right) \equiv \varphi \left( t\right) L\left(
w,x,y,z\right) +\psi \left( t\right) +\eta \left( t,r\right)  \label{def f}
\end{equation}%
where $\psi \left( t\right) $ and $\eta \left( t,r\right) $ are smooth
nonnegative functions constructed as follows, in order that $f$ is
elliptical on $B_{\mathbb{R}^{4}}\left( 0,1\right) \times \left( -1,1\right) 
$, yet cannot be written as a finite sum of squares of $C^{2,\omega }$
functions.

The function $\psi \left( t\right) $ is constructed similar to that in \cite%
{BoBrCoPe} but incorporating an additional function as follows. First we fix
a smooth strictly increasing function $\lambda :\left( 0,1\right)
\rightarrow \left( 0,1\right) $ with $\lim_{r\searrow 0}\lambda \left(
r\right) =0$, so that the inverse function $\lambda ^{-1}\left( t\right) $
is also strictly increasing with limit $0$ at the origin. We will almost
exclusively choose $\lambda \left( r\right) =r$ for $0<r<1$. Next, we choose
a nondecreasing flat elliptical function $\psi _{0}$ on $\left( -1,1\right) $
such that $\frac{\psi _{0}\left( t\right) }{\varphi \left( t\right) \lambda
^{-1}\left( t\right) ^{4}}$ is also nondecreasing and%
\begin{equation}
\frac{1}{\sqrt{\varphi \left( t\right) }\lambda ^{-1}\left( t\right) ^{2}}%
\leq \mathcal{C}\left( \frac{\psi _{0}\left( t\right) }{\varphi \left(
t\right) \lambda ^{-1}\left( t\right) ^{4}}\right) ,\ \ \ \ \ 0<t<1,
\label{psi_0}
\end{equation}%
e.g. using that $\mathcal{C}$ may be assumed strictly decreasing, we may take%
\begin{equation*}
\psi _{0}\left( t\right) \equiv \varphi \left( t\right) \lambda ^{-1}\left(
t\right) ^{4}\mathcal{C}^{-1}\left( \frac{1}{\sqrt{\varphi \left( t\right) }%
\lambda ^{-1}\left( t\right) ^{2}}\right) .
\end{equation*}%
Then in order to obtain a \emph{smooth} such function, set 
\begin{equation*}
\psi \left( t\right) \equiv \int_{\frac{t}{2}}^{t}\psi _{0}\left( s\right)
g\left( \frac{t-s}{t}\right) \frac{ds}{t},
\end{equation*}%
where $g$ is smooth nonnegative function supported in $\left( 0,\frac{1}{2}%
\right) $ with $\int g=1$. Then $\psi \left( t\right) $ is smooth and
because $\psi _{0}$ is nondecreasing, we conclude from the definition of $%
\psi $, that $\psi $ is also nondecreasing, and moreover that $\psi \left(
t\right) \leq \psi _{0}\left( t\right) $ for $0<t<1$. Finally, since $%
\mathcal{C}\left( \tau \right) $ is decreasing, we obtain from (\ref{psi_0})
that%
\begin{equation}
\frac{1}{\sqrt{\varphi \left( t\right) }\lambda ^{-1}\left( t\right) ^{2}}%
\leq \mathcal{C}\left( \frac{\psi \left( t\right) }{\varphi \left( t\right)
\lambda ^{-1}\left( t\right) ^{4}}\right) ,\ \ \ \ \ 0<t<1.  \label{psi_0'}
\end{equation}

The function $\eta \left( t,r\right) $ is chosen to have the form $\eta
\left( t,r\right) =\sigma \left( r\right) h\left( \frac{t}{\lambda \left(
r\right) }\right) $ where $h$ is a smooth nonnegative function on $\left(
-1,1\right) $ with $h\left( 0\right) =1$, and where $\sigma \left( r\right) $
is an elliptical flat smooth function on $\left( 0,1\right) $, chosen so
small that $\eta \left( t,r\right) $ is a flat smooth function on $\left(
-1,1\right) \times \left( 0,1\right) $. More precisely we need only choose $%
\sigma \left( r\right) $ small enough so that for all $m,n\in \mathbb{N}$,%
\begin{eqnarray}
\frac{\partial ^{m+n}}{\partial t^{m}\partial r^{n}}\eta \left( t,r\right)
&=&\frac{\partial ^{n}}{\partial r^{n}}\left( \frac{\sigma \left( r\right) }{%
\lambda \left( r\right) ^{m}}h^{\left( m\right) }\left( \frac{t}{\lambda
\left( r\right) }\right) \right)  \label{der est} \\
&=&\sum_{k=0}^{n}\left( 
\begin{array}{c}
n \\ 
k%
\end{array}%
\right) \frac{\partial ^{k}}{\partial r^{k}}\left( \frac{\sigma \left(
r\right) }{\lambda \left( r\right) ^{m}}\right) \ \frac{\partial ^{n-k}}{%
\partial r^{n-k}}h^{\left( m\right) }\left( \frac{t}{\lambda \left( r\right) 
}\right)  \notag
\end{eqnarray}%
tends to $0$ as $\left( t,r\right) \rightarrow \left( 0,0\right) $. Thus we
now have%
\begin{equation}
f\left( w,x,y,z,t\right) \equiv \varphi \left( t\right) L\left(
w,x,y,z\right) +\psi \left( t\right) +\sigma \left( r\right) h\left( \frac{t%
}{\lambda \left( r\right) }\right) .  \label{def f'}
\end{equation}

With these constructions completed, we see that $f$ is an elliptical flat
smooth function on $B_{\mathbb{R}^{4}}\left( 0,1\right) \times \left(
-1,1\right) $. Now suppose, in order to derive a contradiction, that $%
f=\sum_{j=1}^{\nu }G_{j}^{2}$ where $G_{j}\in C^{2,\omega }\left( B_{\mathbb{%
R}^{4}}\left( 0,1\right) \times \left( -1,1\right) \right) $, i.e.%
\begin{eqnarray*}
\varphi \left( t\right) L\left( w,x,y,z\right) +\psi \left( t\right) +\sigma
\left( r\right) h\left( \frac{t}{\lambda \left( r\right) }\right)
&=&\sum_{j=1}^{\nu }G_{j}\left( w,x,y,z,t\right) ^{2}, \\
\text{for }\left( w,x,y,z,t\right) &\in &B_{\mathbb{R}^{4}}\left( 0,1\right)
\times \left( -1,1\right) .
\end{eqnarray*}%
Then since $h\left( \frac{t}{\lambda \left( r\right) }\right) $ vanishes for 
$\lambda \left( r\right) \leq \left\vert t\right\vert $, i.e. $r\leq \lambda
^{-1}\left( \left\vert t\right\vert \right) $, we have with $W\equiv \left(
w,x,y,z\right) $ and wlog $t>0$, that%
\begin{equation*}
\varphi \left( t\right) L\left( W\right) +\psi \left( t\right)
=\sum_{j=1}^{\nu }G_{j}\left( W,t\right) ^{2},\ \ \ \ \ \text{for }r\leq
\lambda ^{-1}\left( \left\vert t\right\vert \right) ,
\end{equation*}%
and rescaling $W$ by $\lambda ^{-1}\left( t\right) $ we have, 
\begin{eqnarray*}
\varphi \left( t\right) L\left( \lambda ^{-1}\left( t\right) W\right) +\psi
\left( t\right) &=&\sum_{j=1}^{\nu }G_{j}\left( \lambda ^{-1}\left( t\right)
W,t\right) ^{2}, \\
\text{for }r &=&\left\vert W\right\vert <1,t\in \left( 0,1\right) .
\end{eqnarray*}%
Multiplying by $\frac{1}{\varphi \left( t\right) \lambda ^{-1}\left(
t\right) ^{4}}$, and using that $L$ is homogeneous of degree four, we obtain%
\begin{eqnarray*}
L\left( W\right) +\frac{\psi \left( t\right) }{\varphi \left( t\right)
\lambda ^{-1}\left( t\right) ^{4}} &=&\sum_{j=1}^{\nu }\left( \frac{%
G_{j}\left( \lambda ^{-1}\left( t\right) W,t\right) }{\sqrt{\varphi \left(
t\right) }\lambda ^{-1}\left( t\right) ^{2}}\right) ^{2}, \\
\text{for }r &=&\left\vert W\right\vert <1,t\in \left( 0,1\right) .
\end{eqnarray*}%
Since $G_{j}\in C^{2,\omega }\left( B_{\mathbb{R}^{4}}\left( 0,1\right)
\times \left( -1,1\right) \right) $, the functions $W\rightarrow G_{j}\left(
W,t\right) $ lie in a bounded set in $C^{2,\omega }\left( B_{\mathbb{R}%
^{4}}\left( 0,1\right) \right) $ independent of $t$ and $j$, and hence also
the collection of functions%
\begin{equation*}
H_{j}^{t}\left( W\right) \equiv G_{j}\left( \lambda ^{-1}\left( t\right)
W,t\right) ,\ \ \ \ \ 1\leq j\leq \nu ,t\in \left( 0,1\right) ,
\end{equation*}%
is bounded in $C^{2,\omega }\left( B_{\mathbb{R}^{4}}\left( 0,1\right)
\right) $, say $\sum_{j=1}^{\nu }\left\Vert H_{j}^{t}\right\Vert
_{C^{2,\omega }\left( B_{\mathbb{R}^{4}}\left( 0,1\right) \right) }\leq 
\mathfrak{N}_{\nu }$. Thus with $\tau =\tau \left( t\right) \equiv \frac{%
\psi \left( t\right) }{\varphi \left( t\right) \lambda ^{-1}\left( t\right)
^{4}}$, we have from Lemma \ref{g control} and (\ref{psi_0'}) that%
\begin{eqnarray*}
&&\frac{\mathfrak{N}_{\nu }}{\sqrt{\varphi \left( t\right) }\lambda
^{-1}\left( t\right) ^{2}}\geq \sum_{j=1}^{\nu }\left\Vert \frac{H_{j}^{t}}{%
\sqrt{\varphi \left( t\right) }\lambda ^{-1}\left( t\right) ^{2}}\right\Vert
_{C^{2,\omega }\left( B_{\mathbb{R}^{4}}\left( 0,1\right) \right) }\geq 
\mathcal{C}_{\nu }\left( \tau \left( t\right) \right) \\
&=&\frac{\mathcal{C}_{\nu }\left( \tau \left( t\right) \right) }{\mathcal{C}%
\left( \tau \left( t\right) \right) }\mathcal{C}\left( \tau \left( t\right)
\right) \geq \frac{\mathcal{C}_{\nu }\left( \tau \left( t\right) \right) }{%
\mathcal{C}\left( \tau \left( t\right) \right) }\frac{1}{\sqrt{\varphi
\left( t\right) }\lambda ^{-1}\left( t\right) ^{2}},
\end{eqnarray*}%
which contradicts $\lim_{\tau \searrow 0}\frac{\mathcal{C}_{\nu }\left( \tau
\right) }{\mathcal{C}\left( \tau \right) }=\infty $ in (\ref{ratio}),
provided that we choose $\psi \left( t\right) $ to satisfy in addition that%
\begin{equation}
\lim_{t\searrow 0}\frac{\psi \left( t\right) }{\varphi \left( t\right)
\lambda ^{-1}\left( t\right) ^{4}}=\lim_{t\searrow 0}\tau \left( t\right) =0.
\label{in add}
\end{equation}%
This completes our construction of an elliptical flat smooth function $f$ on 
$B_{\mathbb{R}^{4}}\left( 0,1\right) \times \left( -1,1\right) $ as in (\ref%
{def f}) that cannot be written as a finite sum of squares of $C^{2,\omega }$
functions.

\begin{remark}
\label{nu dependence}In order to derive a contradiction in the above
argument, it is enough to take $t$ so small that 
\begin{equation*}
\frac{\mathcal{C}_{\nu }\left( \tau \left( t\right) \right) }{\mathcal{C}%
\left( \tau \left( t\right) \right) }>\mathfrak{N}_{\nu }=\sum_{j=1}^{\nu
}\left\Vert H_{j}^{t}\right\Vert _{C^{2,\omega }\left( B_{\mathbb{R}%
^{4}}\left( 0,1\right) \right) }.
\end{equation*}
\end{remark}

\subsection{Connection with weak monotonicity}

Here we investigate conditions on $0<s<1$ under which the function $f$ in (%
\ref{def f'}) above, i.e.%
\begin{equation*}
f\left( w,x,y,z,t\right) \equiv \varphi \left( t\right) L\left(
w,x,y,z\right) +\psi \left( t\right) +\sigma \left( r\right) h\left( \frac{t%
}{\lambda \left( r\right) }\right) ,
\end{equation*}%
is $\omega _{s}$-monotone on $B_{\mathbb{R}^{4}}\left( 0,1\right) \times
\left( -1,1\right) $, resulting in the following theorem that connects the
parameter $s$ to the functions $\varphi $ and $\psi $ in the definition of
the flat function $f_{\varphi ,\psi ,\sigma }$. We will assume the following
further restrictions on $f$.

\begin{description}
\item[Further restrictions] We suppose that the functions $\varphi ,\psi
,\sigma, \lambda $, and $h$ in the construction of $f_{\varphi ,\psi ,\sigma }$ in (\ref{def
	f'}) also satisfy

\begin{enumerate}
\item $\psi \left( t\right) =o\left( \varphi \left( t\right) t^{4}\right) $
as $t\searrow 0$,

\item $\sigma \left( t\right) =\varphi \left( t\right) $, for $t>0$,

\item $\lambda \left( r\right) =r$, for $r>0$,

\item there is a constant $0<\rho <1$ such that the function $h=h_{\rho }$
is a smooth nonnegative even function on $\mathbb{R}$ that is decreasing on $%
\left[ 0,\infty \right) $, and satisfies%
\begin{eqnarray*}
h_{\rho }\left( x\right) &=&1\text{, for }0\leq x\leq \rho , \\
0 &<&h_{\rho }\left( x\right) <1\text{, for }\rho <x<1, \\
\text{and }h_{\rho }\left( x\right) &=&0\text{, for }x\geq 1.
\end{eqnarray*}
\end{enumerate}
\end{description}

We denote such a function $f$ by $f_{\varphi ,\psi ,h_{\rho }}$ when we wish
to emphasize the dependence on $\varphi ,\psi ,h_{\rho }$. Recall that for a
modulus of continuity $\omega $, we defined in (\ref{omega mon}) the $\omega 
$-monotone functional of $f$ by 
\begin{equation*}
\left\Vert f\right\Vert _{\omega -\func{mon}}\equiv \sup_{x\in B\left(
0,1\right) ,\ y\in B\left( \frac{x}{2},\frac{\left\vert x\right\vert }{2}%
\right) }\frac{f\left( y\right) }{\omega \left( f\left( x\right) \right) }.
\end{equation*}%
Now for any functions $\varphi ,\psi $ as above, define the three functionals%
\begin{eqnarray*}
\mathcal{R}_{\left( \varphi ,\psi \right) }^{\omega }\left( \gamma \right)
&\equiv &\sup_{0<t\leq 1}\frac{\psi \left( t\right) }{\varphi \left(
t\right) }\frac{\varphi \left( \gamma t\right) }{\omega \left( \psi \left(
t\right) \right) },\ \ \ \ \ \text{for }0<\gamma <\infty , \\
\mathcal{S}_{\left( \varphi ,\psi \right) }^{\omega }\left( \gamma \right)
&\equiv &\sup_{0<t<1}\frac{\varphi \left( \gamma t\right) t^{4}}{\omega
\left( \psi (t)\right) },\ \ \ \ \ \text{for }0<\gamma <\infty , \\
\mathcal{T}_{\varphi }^{\omega }\left( \gamma \right) &\equiv &\sup_{0<t<1}%
\frac{\varphi \left( \gamma t\right) t^{4}}{\omega \left( \varphi
(t)t^{4}\right) },\ \ \ \ \ \text{for }0<\gamma <\infty ,
\end{eqnarray*}%
where $\mathcal{R}_{\left( \varphi ,\psi \right) }^{\omega }\left( \gamma
\right) \lesssim \mathcal{S}_{\left( \varphi ,\psi \right) }^{\omega }\left(
\gamma \right) $ for $0<\gamma <\infty $, because $\psi \left( t\right)
=o\left( \varphi \left( t\right) t^{4}\right) $.

\begin{theorem}
\label{1/4 and epsilon_0}Set 
\begin{equation*}
f_{\varphi ,\psi ,h_{\rho }}\left( W,t\right) =\varphi \left( t\right)
L\left( W\right) +\psi \left( t\right) +\varphi \left( r\right) h_{\rho
}\left( \frac{t}{r}\right) ,\ \ \ \ \ r=\left\vert W\right\vert ,W\in 
\mathbb{R}^{4},t>0,
\end{equation*}%
where $\varphi ,\psi ,h_{\rho }$ satisfy the conditions listed above. Let $%
\gamma _{\alpha }\equiv \frac{1+\sqrt{1+\alpha ^{2}}}{2\alpha }$ for $%
0<\alpha <\infty $.

\begin{enumerate}
\item Then for $0<\rho <1$,\ and every $\delta >0$, the function $%
f=f_{\varphi ,\psi ,h_{\rho }}$ on $\mathbb{R}^{n}$ is elliptical flat and
smooth, and there are positive constants $c_{\rho ,\delta }$ and $C_{\rho
,\delta }$, such that\emph{\ }%
\begin{eqnarray}
&&c\mathcal{S}_{\varphi ,\psi }^{\omega }\left( \frac{1}{2}\right) +c%
\mathcal{T}_{\varphi }^{\omega }\left( \gamma _{1}\right) \leq \left\Vert
f_{\varphi ,\psi ,h_{\rho }}\right\Vert _{\omega -\func{mon}}
\label{mon equiv} \\
&\leq &C_{\delta }\left[ \mathcal{R}_{\left( \varphi ,\psi \right) }^{\omega
}\left( 1+\delta \right) +\mathcal{S}_{\varphi ,\psi }^{\omega }\left( \frac{%
1}{2}+\delta \right) \right] +C_{\rho ,\delta }\mathcal{T}_{\varphi
}^{\omega }\left( \gamma _{\rho }+\delta \right) .  \notag
\end{eqnarray}

\item Now take $\varphi \left( t\right) =e^{-\frac{1}{t^{2}}}$ and suppose $%
0<s<s_{0}$ where%
\begin{equation}
s_{0}\equiv \left( \frac{1+\sqrt{2}}{2}\right) ^{-2}=0.68629.  \label{thresh}
\end{equation}

\begin{enumerate}
\item Then there are no functions $g_{\ell }\in C^{2,\beta }$ with $%
f=f_{\varphi ,\psi ,\sigma ,h_{\rho }}=\sum_{\ell =1}^{\nu }g_{\ell }^{2}$
for any $\nu \in \mathbb{N}$ if%
\begin{equation*}
\lim_{t\searrow 0}\frac{\varphi \left( t\right) ^{\frac{4}{\beta }}t^{\frac{%
16}{\beta }}}{\psi \left( t\right) }=\infty .
\end{equation*}

\item If $0<s<\min \left\{ \beta ,s_{0}\right\} $, then there is $0<\rho <1$
and a function $\psi \left( t\right) $ such that the elliptical flat smooth
function $f=f_{\varphi ,\psi ,h_{\rho }}$ is $\omega _{s}$-monotone but not $%
SOS_{\omega _{\beta }}$, i.e. there are no functions $g_{\ell }\in
C^{2,\beta }$ with $f=f_{\varphi ,\psi ,\sigma ,h_{\rho }}=\sum_{\ell
=1}^{\nu }g_{\ell }^{2}$ for $\nu \in \mathbb{N}$.
\end{enumerate}
\end{enumerate}
\end{theorem}

\begin{remark}
The quantities $\mathcal{R}_{\left( \varphi ,\psi \right) }^{\omega }$ and $%
\mathcal{S}_{\left( \varphi ,\psi \right) }^{\omega }$ in (\ref{mon equiv}) are the key functionals controlling the $%
\omega $-monotone functional of $f_{\varphi ,\psi ,h_{\rho }}$.

\begin{enumerate}
\item The estimate (\ref{mon equiv}) is sharp in the sense that the lower
bound `equals up to multiplicative constants' the limit as $\delta
\rightarrow 0$ and $\rho \rightarrow 1$ of the upper bound, namely 
\begin{eqnarray*}
\mathcal{R}_{\left( \varphi ,\psi \right) }^{\omega }\left( 1+\delta \right)
+\mathcal{S}_{\varphi ,\psi }^{\omega }\left( \frac{1}{2}+\delta \right) +%
\mathcal{T}_{\varphi }^{\omega }\left( \gamma _{\rho }+\delta \right)
&\longrightarrow &\mathcal{R}_{\left( \varphi ,\psi \right) }^{\omega
}\left( 1\right) +\mathcal{S}_{\varphi ,\psi }^{\omega }\left( \frac{1}{2}%
\right) +\mathcal{T}_{\varphi }^{\omega }\left( \gamma _{1}\right) \\
&\approx &\mathcal{S}_{\varphi ,\psi }^{\omega }\left( \frac{1}{2}\right) +%
\mathcal{T}_{\varphi }^{\omega }\left( \gamma _{1}\right)
\end{eqnarray*}%
since $\mathcal{R}_{\left( \varphi ,\psi \right) }^{\omega }\left( 1\right) $
is a constant.

\item Note also that the right hand side of (\ref{mon equiv}) is dominated
by a multiple of the single term $\mathcal{S}_{\varphi ,\psi }^{\omega
}\left( \gamma _{\rho }+\delta \right) $, but the smaller limiting term $%
\mathcal{S}_{\varphi ,\psi }^{\omega }\left( \gamma _{1}\right) $ is already
far larger than the lower bound.

\item The functional $\mathcal{T}_{\varphi }^{\omega }$ is an admissibility
requirement for the function $\varphi $, and plays no other role in
distinguishing which pairs of functions $\left( \varphi ,\psi \right) $ give
rise to $f_{\varphi ,\psi ,h_{\rho }}$ being $\omega $-monotone.
\end{enumerate}
\end{remark}

\subsubsection{Proof of necessity in part (1)}

Here we prove the lower bound 
\begin{equation*}
c\mathcal{S}_{\varphi ,\psi }^{\omega }\left( \frac{1}{2}\right) +c\mathcal{T%
}_{\varphi }^{\omega }\left( \gamma _{1}\right) \leq \left\Vert f_{\varphi
,\psi ,h_{\rho }}\right\Vert _{\omega -\func{mon}}.
\end{equation*}%
Given points $P,Q\in \mathbb{R}^{5}$ with $Q\in \partial B\left( \frac{P}{2},%
\frac{\left\vert P\right\vert }{2}\right) $, we have from $\omega $
-monotonicity that 
\begin{equation*}
\frac{f\left( Q\right) }{\omega \left( f\left( P\right) \right) }\leq
\left\Vert f\right\Vert _{\omega -\func{mon}}.
\end{equation*}%
We now consider two specific pairs of points $\left( P_{1},Q_{1}\right) $
and $\left( P_{2},Q_{2}\right) $, in order to derive the lower bounds above.

Let $P_{1}\equiv \left( 0,t\right) $ and $Q_{1}\equiv \left( W,\frac{t}{2}%
\right) $, where $W$ is any point in $\mathbb{R}^{4}$ with $r=\left\vert
W\right\vert =\frac{t}{2}$, so that $Q_{1}\in \partial B\left( \frac{P_{1}}{2%
},\frac{\left\vert P_{1}\right\vert }{2}\right) $, and 
\begin{eqnarray*}
f\left( P_{1}\right) &=&f\left( 0,t\right) =\psi \left( t\right) , \\
f\left( Q_{1}\right) &=&f\left( W,\left\vert W\right\vert \right) =\varphi
\left( \frac{t}{2}\right) L\left( W\right) +\psi \left( \frac{t}{2}\right)
\approx \varphi \left( \frac{t}{2}\right) t^{4}.
\end{eqnarray*}%
This gives 
\begin{equation*}
\left\Vert f\right\Vert _{\omega -\func{mon}}\geq \frac{f\left( Q\right) }{%
\omega \left( f\left( P\right) \right) }\approx \frac{\varphi \left( \frac{t%
}{2}\right) t^{4}}{\omega \left( \psi \left( t\right) \right) },\ \ \ \ \ 
\text{for all }0<t\leq 1,
\end{equation*}%
and thus 
\begin{equation*}
\left\Vert f\right\Vert _{\omega -\func{mon}}\geq c\mathcal{S}_{\varphi
,\psi }^{\omega }\left( \frac{1}{2}\right) .
\end{equation*}

Next let $P_{2}\equiv \left( W,\left\vert W\right\vert \right) $ and $%
Q_{2}\equiv \left( \frac{W}{2},\left( \frac{1}{2}+\frac{1}{\sqrt{2}}\right)
\left\vert W\right\vert \right) $, so that $Q_{2}\in \partial B\left( \frac{%
P_{2}}{2},\frac{\left\vert P_{2}\right\vert }{2}\right) $, and 
\begin{eqnarray*}
f\left( P_{2}\right) &=&f\left( W,\left\vert W\right\vert \right) =\varphi
\left( r\right) L\left( W\right) +\psi \left( r\right) \approx \varphi
(r)r^{4}, \\
f\left( Q_{2}\right) &=&\varphi \left( \left( \frac{1}{2}+\frac{1}{\sqrt{2}}%
\right) r\right) L\left( \frac{W}{2}\right) +\psi \left( \left( \frac{1}{2}+%
\frac{1}{\sqrt{2}}\right) r\right) \approx \varphi (\gamma _{1}r)r^{4},
\end{eqnarray*}%
where $\gamma _{1}=\frac{1}{2}+\frac{1}{\sqrt{2}}$. Therefore, 
\begin{equation*}
\left\Vert f\right\Vert _{\omega -\func{mon}}\geq \frac{f\left( Q\right) }{%
\omega \left( f\left( P\right) \right) }\approx \frac{\varphi (\gamma
_{1}r)r^{4}}{\omega \left( \varphi (r)r^{4}\right) }
\end{equation*}%
for all $r\in (0,1)$ and thus 
\begin{equation*}
\left\Vert f\right\Vert _{\omega -\func{mon}}\geq C\mathcal{T}_{\varphi
}^{\omega }\left( \gamma _{1}\right) .
\end{equation*}

\subsubsection{Proof of sufficiency in part (1)}

Fix a modulus of continuity and a function $f=f_{\varphi ,\psi ,h_{\rho }}$
given by 
\begin{equation*}
f\left( W,t\right) =\varphi \left( t\right) L\left( W\right) +\psi \left(
t\right) +\varphi \left( r\right) h_{\rho }\left( \frac{t}{\left\vert
W\right\vert }\right) ,
\end{equation*}%
as in the hypotheses of Theorem \ref{1/4 and epsilon_0}. We consider pairs
of points $\left( P,Q\right) \in \mathbb{R}^{5}\times \mathbb{R}^{5}$
restricted by 
\begin{eqnarray*}
\left( P,Q\right) &=&\left( \left( W,t\right) ,\left( V,u\right) \right) \in
\Omega , \\
\Omega &\equiv &\left( \overline{B_{\mathbb{R}^{4}}\left( 0,1\right) }\times %
\left[ 0,1\right] \right) \times \left( \overline{B_{\mathbb{R}^{4}}\left(
0,1\right) }\times \left[ 0,1\right] \right) ,
\end{eqnarray*}%
and will estimate the supremum,%
\begin{equation*}
\mathcal{M}f\left( \omega \right) \equiv \sup_{\left( P,Q\right) \in \Omega
:\ Q\in \overline{B_{P}}}\frac{f\left( Q\right) }{\omega \left( f\left(
P\right) \right) },
\end{equation*}%
where $B_{P}=B\left( \frac{P}{2},\frac{\left\vert P\right\vert }{2}\right) $
is the unique ball centered at $\frac{P}{2}$ that includes both the origin
and $P$ in its boundary. This is a localized version of the functional $%
\left\Vert f\right\Vert _{\omega _{s}-\func{mon}}$.

Now the functions%
\begin{equation*}
z\rightarrow f\left( zW,zt\right) =\varphi \left( zt\right) z^{4}L\left(
W\right) +\psi \left( zt\right) +\varphi \left( zr\right) h\left( \frac{t}{%
\left\vert W\right\vert }\right)
\end{equation*}%
are nondecreasing, which has the consequence that the supremum in $\mathcal{M%
}_{\omega }$ is achieved for $Q\in \partial B_{P}$, so%
\begin{equation*}
\mathcal{M}f\left( \omega \right) =\sup_{\left( P,Q\right) \in \Omega :\
Q\in \partial B_{P}}\frac{f\left( Q\right) }{\omega \left( f\left( P\right)
\right) }.
\end{equation*}

We now claim that it further suffices to restrict the supremum to pairs $%
\left( P,Q\right) =\left( \left( W,t\right) ,\left( V,u\right) \right) \in
\Omega \ $with$\ Q\in \partial B_{P}$ and $V\parallel W$, where $V$ and $W$
are parallel if $V=\lambda W$ or $W=\lambda V$ for some $\lambda \in \mathbb{%
R}$.

\begin{claim}
\begin{equation}
\mathcal{M}f\left( \omega \right) \approx \sup_{\substack{ \left( P,Q\right)
=\left( \left( W,t\right) ,\left( V,u\right) \right) \in \Omega  \\ Q\in
\partial B_{P}\text{ and }V\parallel W}}\frac{f\left( Q\right) }{\omega
\left( f\left( P\right) \right) }  \label{main sup}
\end{equation}
\end{claim}

\begin{proof}
Denote the supremum on the right hand side of (\ref{main sup}) by $\mathcal{M%
}_{\parallel }f\left( \omega \right) $, so that $\mathcal{M}_{\parallel
}f\left( \omega \right) \leq \mathcal{M}f\left( \omega \right) $. We have 
\begin{equation*}
f\left( Q\right) \approx \varphi \left( u\right) \left\vert V\right\vert
^{4}+\psi \left( u\right) +\varphi \left( \left\vert V\right\vert \right)
h\left( \frac{u}{\left\vert V\right\vert }\right) .
\end{equation*}%
Rotate the ball $B_{P}$ about its vertical axis, namely the diameter of $%
B_{P}$ that is parallel to the vector $\mathbf{e}_{t}$, so that $Q=\left(
V,u\right) $ is rotated to the point $Q^{\prime }=\left( V^{\prime
},u\right) $ in the plane spanned by $\mathbf{e}_{t}$ and $\mathbf{e}_{W}$,
for which $\left\vert V^{\prime }\right\vert \geq \left\vert V\right\vert $.
Then%
\begin{equation*}
f\left( Q\right) \lesssim \varphi \left( u\right) \left\vert V^{\prime
}\right\vert ^{4}+\psi \left( u\right) +\varphi \left( \left\vert V^{\prime
}\right\vert \right) h\left( \frac{u}{\left\vert V^{\prime }\right\vert }%
\right) \approx f\left( Q^{\prime }\right) \leq \mathcal{M}_{\parallel
}f\left( \omega \right) \omega \left( f\left( P\right) \right) ,
\end{equation*}%
since $Q^{\prime }\in \partial B_{P}$ and $V^{\prime }\parallel W$. Thus we
have 
\begin{equation*}
\mathcal{M}f\left( \omega \right) =\sup_{\substack{ \left( P,Q\right)
=\left( \left( W,t\right) ,\left( V,u\right) \right) \in \Omega  \\ Q\in
\partial B_{P}}}\frac{f\left( Q\right) }{\omega \left( f\left( P\right)
\right) }\lesssim \mathcal{M}_{\parallel }f\left( \omega \right) .
\end{equation*}
\end{proof}

Now $P=\left( W,t\right) $ and $Q=\left( V,u\right) $ satisfy $Q\in \partial
B_{P}$ if and only if 
\begin{eqnarray*}
\left\vert V-\frac{W}{2}\right\vert ^{2}+\left( u-\frac{t}{2}\right) ^{2}
&=&\left\vert \left( V,u\right) -\left( \frac{W}{2},\frac{t}{2}\right)
\right\vert ^{2}=\left\vert Q-\frac{P}{2}\right\vert ^{2} \\
&=&\left( \frac{\left\vert P\right\vert }{2}\right) ^{2}=\frac{\left\vert
W\right\vert ^{2}+t^{2}}{4}.
\end{eqnarray*}%
Set $r=\left\vert W\right\vert $, $z=\left\vert V\right\vert $ and suppose
that $V\parallel W$, so that $\left\vert V-\frac{W}{2}\right\vert
=\left\vert \lambda -\frac{1}{2}\right\vert r=\left\vert z-\frac{r}{2}%
\right\vert $. Under these conditions, we then have $Q\in \partial B_{P}$ if
and only if%
\begin{equation}
\left( z-\frac{r}{2}\right) ^{2}+\left( u-\frac{t}{2}\right) ^{2}=\frac{%
r^{2}+t^{2}}{4},  \label{bdry_cond}
\end{equation}%
and 
\begin{align}
f(P)& \approx \varphi \left( t\right) r^{4}+\psi \left( t\right) +\varphi
\left( r\right) h\left( \frac{t}{r}\right) ,  \label{fPQ} \\
f(Q)& \approx \varphi \left( u\right) z^{4}+\psi \left( u\right) +\varphi
\left( z\right) h\left( \frac{u}{z}\right) .  \notag
\end{align}

Here we prove the upper bound for $\mathcal{M}f\left( \omega \right) $,
which is comparable to $\left\Vert f\right\Vert _{\omega _{s}-\func{mon}}$.
To estimate the supremum in (\ref{main sup}), we will consider different
cases depending on the sizes of $\frac{t}{r}$ and $\frac{u}{z}$, and
depending on which of three terms dominates in the expression for $f(Q)$ in (%
\ref{fPQ}). We will use the abbreviation $\sup_{\func{restricted}}$ at
various places in the proof to denote the supremum of the ratio $\frac{%
f\left( P\right) }{\omega \left( f\left( Q\right) \right) }$ subject to the
restrictions in force at that time.

\bigskip

\textbf{The case }$r=\left\vert W\right\vert \leq t$: We will first prove
that when $r=\left\vert W\right\vert \leq t$, we have the upper bound, 
\begin{eqnarray}
\mathcal{M}_{1}f_{\varphi ,\psi ,h_{\rho }}\left( \omega \right) &\equiv
&\sup_{\substack{ \left( P,Q\right) =\left( \left( W,t\right) ,\left(
V,u\right) \right) \in \Omega  \\ Q\in \partial B_{P}\text{ and }V\parallel W%
\text{ and }\left\vert W\right\vert \leq t}}\frac{f\left( P\right) }{\omega
\left( f\left( Q\right) \right) }  \label{upper bound} \\
&\leq &C_{\delta }\mathcal{R}_{\left( \varphi ,\psi \right) }^{\omega
}\left( 1+\delta \right) +C_{\delta }\mathcal{T}_{\varphi }^{\omega }\left(
\gamma _{1}+\delta \right) +C_{\delta }\mathcal{S}_{\left( \varphi ,\psi
\right) }^{\omega }\left( \frac{1}{2}+\delta \right) ,  \notag
\end{eqnarray}%
where $\gamma _{\alpha }=\frac{1+\sqrt{1+\alpha ^{2}}}{2\alpha }$ and $%
0<\delta <1$. Note that $h(\frac{t}{r})=0$ in this case, so from (\ref{fPQ})
we have 
\begin{equation*}
f(P)\approx \varphi \left( t\right) r^{4}+\psi \left( t\right) .
\end{equation*}

\begin{proof}
We consider further subcases depending on the size of $\frac{u}{z}$, and on
which term dominates in the expression for $f\left( Q\right) $ in (\ref{fPQ}%
).

\textbf{Case }$u\geq z$: Suppose first that the variables $(V,u)$ satisfy $%
\frac{u}{z}\geq 1$ and $\varphi \left( u\right) z^{4}\geq \psi \left(
u\right) $, so that 
\begin{equation*}
f(Q)\approx \varphi \left( u\right) z^{4}.
\end{equation*}%
Using the restrictions $z\leq u$ and $r^{4}\varphi \left( t\right) \geq \psi
\left( t\right) $ we get 
\begin{equation*}
\sup_{\func{restricted}}\frac{f(Q)}{\omega \left( f\left( P\right) \right) }%
\approx \sup_{\func{restricted}}\frac{\varphi \left( u\right) z^{4}}{\omega
\left( \varphi \left( t\right) r^{4}+\psi \left( t\right) \right) },
\end{equation*}%
and from (\ref{bdry_cond}), $\left( z-\frac{r}{2}\right) ^{2}+\left( u-\frac{%
t}{2}\right) ^{2}=\frac{r^{2}+t^{2}}{4}$, together with the restriction $%
r\leq t$ we obtain $z=\frac{r}{2}$ and%
\begin{equation*}
u=\frac{t}{2}+\sqrt{\frac{r^{2}+t^{2}}{4}},
\end{equation*}%
since $z=\frac{r}{2}\leq \frac{t}{2}+\sqrt{\frac{r^{2}+t^{2}}{4}}=u$. Now in
the case where $\psi \left( t\right) $ dominates in the denominator, we have 
$r^{2}\leq \sqrt{\frac{\psi \left( t\right) }{\varphi \left( t\right) }}%
=o\left( t^{2}\right) $ and so $u=\frac{t}{2}+\sqrt{\frac{r^{2}+t^{2}}{4}}%
=\left( 1+o\left( 1\right) \right) t$ as $t\searrow 0$, and so under all of
these restrictions in $\sup_{\func{restricted}}$ we have for any $\delta >0$%
, 
\begin{eqnarray*}
\sup_{\func{restricted}}\frac{f(Q)}{\omega \left( f\left( P\right) \right) }
&\approx &\sup_{\func{restricted}}\frac{\varphi \left( u\right) z^{4}}{%
\omega \left( \varphi \left( t\right) r^{4}+\psi \left( t\right) \right) } \\
&\lesssim &\sup_{\func{restricted}}\frac{\varphi \left( \left( \frac{1}{2}+%
\sqrt{\frac{o\left( 1\right) +1}{4}}\right) t\right) \frac{\psi \left(
t\right) }{\varphi \left( t\right) }}{\omega \left( \psi \left( t\right)
\right) }\lesssim C_{\delta }\mathcal{R}_{\left( \varphi ,\psi \right)
}^{\omega }\left( 1+\delta \right) .
\end{eqnarray*}%
On the other hand, in the case when $\varphi \left( t\right) r^{4}$
dominates in the denominator, we can use the inequality $\omega \left(
xy\right) \geq \omega \left( x\right) y$ with $x=\varphi \left( t\right)
t^{4}$ to obtain%
\begin{eqnarray*}
\sup_{\func{restricted}}\frac{f(Q)}{\omega \left( f\left( P\right) \right) }
&\approx &\sup_{\func{restricted}}\frac{\varphi \left( u\right) z^{4}}{%
\omega \left( \varphi \left( t\right) r^{4}\right) }\approx \sup_{0<r\leq
t\leq 1}\frac{\varphi \left( \frac{t}{2}+\sqrt{\frac{r^{2}+t^{2}}{4}}\right)
r^{4}}{\omega \left( \varphi \left( t\right) r^{4}\right) } \\
&\lesssim &\sup_{0<r\leq t\leq 1}\frac{\varphi \left( \frac{t}{2}+\sqrt{%
\frac{r^{2}+t^{2}}{4}}\right) r^{4}}{\omega \left( \varphi \left( t\right)
t^{4}\right) \left( \frac{r}{t}\right) ^{4}}=\sup_{0<t\leq 1}\frac{\varphi
\left( \left( \frac{1}{2}+\sqrt{\frac{1}{2}}\right) t\right) t^{4}}{\omega
\left( \varphi \left( t\right) t^{4}\right) }=\mathcal{T}_{\varphi }^{\omega
}\left( \gamma _{1}\right) .
\end{eqnarray*}

\textbf{Case }$u\leq z$: In this case $f\left( Q\right) \lesssim \varphi (z)$
and 
\begin{equation*}
\sup_{\func{restricted}}\frac{f(Q)}{\omega \left( f\left( P\right) \right) }%
\approx \sup_{\func{restricted}}\frac{\varphi (z)}{\omega \left( \varphi
\left( t\right) r^{4}+\psi \left( t\right) \right) }.
\end{equation*}%
The supremum on the right hand side is maximized for $z$ as large as
possible, which by (\ref{bdry_cond}) occurs when $u=\frac{t}{2}$ and 
\begin{equation*}
z=\frac{r}{2}+\frac{1}{2}\sqrt{r^{2}+t^{2}}.
\end{equation*}

We now consider two cases, where $\frac{r}{t}$ is small and large. Note that
with 
\begin{equation*}
\Theta \left( \delta \right) \equiv \frac{\left( \frac{1}{2}+\delta \right)
^{2}-\frac{1}{4}}{\frac{1}{2}+\delta }=\frac{\delta \left( 1+\delta \right) 
}{\frac{1}{2}+\delta },
\end{equation*}%
we have%
\begin{equation*}
z=\frac{r}{2}+\frac{1}{2}\sqrt{r^{2}+t^{2}}\leq \left( \frac{1}{2}+\delta
^{\prime }\right) t,\ \ \ \ \ \text{for }0\leq r\leq \Theta \left( \delta
^{\prime }\right) t,
\end{equation*}%
as is easily seen by squaring the inequality $\frac{1}{2}\sqrt{r^{2}+t^{2}}%
\leq \left( \frac{1}{2}+\delta ^{\prime }\right) t-\frac{r}{2}$. Thus for $%
r\leq \Theta \left( \delta ^{\prime }\right) t$ with $\delta ^{\prime }>0$,
we obtain%
\begin{eqnarray*}
\sup_{\func{restricted}}\frac{f(Q)}{\omega \left( f\left( P\right) \right) }
&\approx &\sup_{\func{restricted}}\frac{\varphi \left( z\right) }{\omega
\left( \varphi \left( t\right) r^{4}+\psi \left( t\right) \right) }\lesssim
\sup_{0<t\leq 1}\frac{\varphi \left( \left( \frac{1}{2}+\delta ^{\prime
}\right) t\right) }{\omega \left( \psi \left( t\right) \right) } \\
&\lesssim &C_{\varepsilon }\sup_{0<t\leq 1}\frac{\varphi \left( \left( \frac{%
1}{2}+\delta ^{\prime }+\varepsilon \right) t\right) t^{4}}{\omega \left(
\psi \left( t\right) \right) },
\end{eqnarray*}%
since $\lim_{t\searrow 0}\frac{\varphi \left( \left( \frac{1}{2}+\delta
^{\prime }\right) t\right) }{\varphi \left( \left( \frac{1}{2}+\delta
^{\prime }+\varepsilon \right) t\right) t^{4}}=\infty $ for $\varepsilon >0$
and any $\varphi $ flat at the origin. We now choose $\delta ^{\prime }$ and 
$\varepsilon $ small enough that $\delta ^{\prime }+\varepsilon \leq \delta $%
, so as to conclude that%
\begin{equation*}
\sup_{\func{restricted}}\frac{f(Q)}{\omega \left( f\left( P\right) \right) }%
\lesssim C_{\varepsilon }\sup_{0<t\leq 1}\frac{\varphi \left( \left( \frac{1%
}{2}+\delta \right) t\right) t^{4}}{\omega \left( \psi \left( t\right)
\right) }=\mathcal{S}_{\left( \varphi ,\psi \right) }^{\omega }\left( \frac{1%
}{2}+\delta \right) .
\end{equation*}%
On the other hand, if $r>\Theta \left( \delta ^{\prime }\right) t$, then 
\begin{eqnarray*}
\sup_{\func{restricted}}\frac{f(Q)}{\omega \left( f\left( P\right) \right) }
&\approx &\sup_{\func{restricted}}\frac{\varphi \left( z\right) }{\omega
\left( \varphi \left( t\right) r^{4}\right) }\approx \sup_{\Theta \left(
\delta ^{\prime }\right) t<r\leq t\leq 1}\frac{\varphi \left( \frac{r}{2}+%
\frac{1}{2}\sqrt{r^{2}+t^{2}}\right) }{\omega \left( \varphi \left( t\right)
r^{4}\right) } \\
&\lesssim &\sup_{0<t\leq 1}\frac{\varphi \left( \gamma _{1}t\right) }{\omega
\left( \varphi \left( t\right) \Theta \left( \delta ^{\prime }\right)
^{4}t^{4}\right) }\lesssim \frac{1}{\Theta \left( \delta ^{\prime }\right)
^{4}}\sup_{0<t\leq 1}\frac{\varphi \left( \gamma _{1}t\right) }{\omega
\left( \varphi \left( t\right) t^{4}\right) } \\
&\lesssim &\frac{1}{\Theta \left( \delta ^{\prime }\right) ^{4}}%
\sup_{0<t\leq 1}\frac{\varphi \left( \left( \gamma _{1}+\delta \right)
t\right) t^{4}}{\omega \left( \varphi \left( t\right) t^{4}\right) }%
=C_{\delta }\mathcal{T}_{\varphi }^{\omega }\left( \gamma _{1}+\delta
\right) ,
\end{eqnarray*}%
using the flatness of $\varphi $ again.
\end{proof}

\textbf{The case }$r=\left\vert W\right\vert >t$: We will now prove that
when $r=\left\vert W\right\vert >t$ we have the following upper bound, 
\begin{equation*}
\mathcal{M}_{2}f_{\varphi ,\psi ,h_{\rho }}\left( \omega \right) \equiv \sup 
_{\substack{ \left( P,Q\right) =\left( \left( W,t\right) ,\left( V,u\right)
\right) \in \Omega  \\ Q\in \partial B_{P}\text{ and }V\parallel W\text{ and 
}\left\vert W\right\vert >t}}\frac{f\left( P\right) }{\omega \left( f\left(
Q\right) \right) }\leq C_{\rho ,\delta }\mathcal{T}_{\varphi }^{\omega
}\left( \gamma _{\rho }+\delta \right) .
\end{equation*}

\begin{proof}
We consider separately the cases $t\leq \rho r$ when $h_{\rho }(\frac{t}{r}%
)=1$, and $r>t>\rho r$ when $0<h_{\rho }(\frac{t}{r})<1$.

\textbf{Case }$t\leq \rho r$: We have $h_{\rho }(\frac{t}{r})=1$ and so from
(\ref{fPQ}) that 
\begin{equation*}
f(P)\approx \varphi \left( r\right) .
\end{equation*}%
In the case $u\leq z$, we have $f(Q)\lesssim \varphi (z)$, and by (\ref%
{bdry_cond}), i.e. $\left( z-\frac{r}{2}\right) ^{2}+\left( u-\frac{t}{2}%
\right) ^{2}=\frac{r^{2}+t^{2}}{4}$, we then have $\varphi (z)\approx
\varphi \left( \frac{r}{2}+\frac{\sqrt{r^{2}+t^{2}}}{2}\right) $ if we
choose $u=\frac{t}{2}$ and $z=\frac{r}{2}+\frac{\sqrt{r^{2}+t^{2}}}{2}$, so
that $u=\frac{t}{2}\leq \frac{r}{2}+\frac{\sqrt{r^{2}+t^{2}}}{2}=z$. We
conclude from the inequality $\omega \left( xy\right) \geq \omega \left(
x\right) y$ with $x=\varphi (r)r^{4}$ and $y=r^{-4}$ that%
\begin{eqnarray*}
\sup_{\func{restricted}}\frac{f(Q)}{\omega \left( f\left( P\right) \right) }
&\lesssim &\sup_{0<t\leq \rho r<1}\frac{\varphi \left( \frac{r}{2}+\frac{%
\sqrt{r^{2}+t^{2}}}{2}\right) }{\omega \left( \varphi \left( r\right)
\right) }\leq \sup_{0<r<1}\frac{\varphi \left( \left( \frac{1}{2}+\frac{%
\sqrt{1+\rho ^{2}}}{2}\right) r\right) }{\omega \left( \varphi (r)\right) }
\\
&\lesssim &\sup_{0<r<1}\frac{\varphi \left( \left( \frac{1}{2}+\frac{\sqrt{%
1+\rho ^{2}}}{2}\right) r\right) r^{4}}{\omega \left( \varphi
(r)r^{4}\right) }=\mathcal{T}_{\varphi }^{\omega }\left( \frac{1}{2}+\frac{%
\sqrt{1+\rho ^{2}}}{2}\right) \leq \mathcal{T}_{\varphi }^{\omega }\left(
\gamma _{1}\right) .
\end{eqnarray*}

On the other hand, if $u\geq z$, then $f(Q)\approx \varphi \left( u\right)
z^{4}+\psi \left( u\right) $, and by (\ref{bdry_cond}) we then have that $u$
is maximized when $z=\frac{r}{2}$ and $u=\frac{t}{2}+\frac{\sqrt{r^{2}+t^{2}}%
}{2}$, which we note satisfies the requirement $u\geq z$. Thus we have $%
u\leq \frac{\rho r}{2}+\frac{\sqrt{r^{2}+\rho ^{2}r^{2}}}{2}=\left( \frac{%
\rho }{2}+\frac{\sqrt{1+\rho ^{2}}}{2}\right) r$ and so%
\begin{eqnarray}
&&\sup_{\func{restricted}}\frac{f(Q)}{\omega \left( f\left( P\right) \right) 
}\approx \sup_{0<r<1}\frac{\varphi \left( u\right) r^{4}+\psi \left(
u\right) }{\omega \left( \varphi (r)\right) }  \label{low} \\
&\approx &\sup_{0<r<1}\frac{\varphi \left( \left( \frac{\rho }{2}+\frac{%
\sqrt{1+\rho ^{2}}}{2}\right) r\right) r^{4}+\psi \left( \left( \frac{\rho }{%
2}+\frac{\sqrt{1+\rho ^{2}}}{2}\right) r\right) }{\omega \left( \varphi
(r)\right) }  \notag \\
&\approx &\sup_{0<r<1}\frac{\varphi \left( \left( \frac{\rho }{2}+\frac{%
\sqrt{1+\rho ^{2}}}{2}\right) r\right) r^{4}}{\omega \left( \varphi
(r)\right) }\lesssim \mathcal{T}_{\varphi }^{\omega }\left( \left( \frac{%
\rho }{2}+\frac{\sqrt{1+\rho ^{2}}}{2}\right) \right) \lesssim \mathcal{T}%
_{\varphi }^{\omega }\left( \gamma _{1}\right) .  \notag
\end{eqnarray}

\textbf{Case }$r>t>\rho r$: Here we have $f\left( P\right) \approx \varphi
\left( t\right) t^{4}+\varphi (r)h\left( \frac{t}{r}\right) $ since $%
r\approx t$ and $\psi \left( t\right) =o\left( \varphi \left( t\right)
t^{4}\right) $. If $u\leq z$ we have $f(Q)\lesssim \varphi (z)$, and by (\ref%
{bdry_cond}), we have $\varphi (z)\approx \varphi \left( \frac{r}{2}+\frac{%
\sqrt{r^{2}+t^{2}}}{2}\right) $ if we maximize $z$ by choosing $u=\frac{t}{2}
$ and $z=\frac{r}{2}+\frac{\sqrt{r^{2}+t^{2}}}{2}$. From this and $r\leq 
\frac{1}{\rho }t$, we obtain 
\begin{equation*}
z=\frac{r}{2}+\frac{1}{2}\sqrt{r^{2}+t^{2}}\leq \frac{t}{2\rho }+\frac{1}{%
2\rho }\sqrt{t^{2}+\rho ^{2}t^{2}}=\frac{1+\sqrt{1+\rho ^{2}}}{2\rho }%
t=\gamma _{\rho }t,
\end{equation*}%
and so 
\begin{eqnarray*}
\sup_{\func{restricted}}\frac{f(Q)}{\omega \left( f\left( P\right) \right) }
&\lesssim &\sup_{\rho r<t<r}\frac{\varphi \left( \frac{r}{2}+\frac{\sqrt{%
r^{2}+t^{2}}}{2}\right) }{\omega \left( \varphi \left( t\right)
t^{4}+\varphi (r)h\left( \frac{t}{r}\right) \right) }\lesssim \sup_{\rho
r<t<r}\frac{\varphi \left( \gamma _{\rho }t\right) }{\omega \left( \varphi
\left( t\right) t^{4}\right) } \\
&\lesssim &\sup_{\rho r<t<r}\frac{\varphi \left( \left( \gamma _{\rho
}+\delta \right) t\right) t^{4}}{\omega \left( \varphi \left( t\right)
t^{4}\right) }=\mathcal{T}_{\varphi }^{\omega }\left( \gamma _{\rho }+\delta
\right) ,
\end{eqnarray*}%
where we have used the flatness of $\varphi $ as before.

Next, if $u\geq z$ we have $f(Q)\approx \varphi \left( u\right) z^{4}+\psi
\left( u\right) \lesssim \varphi \left( u\right) u^{4}$, and so maximizing $%
u $ with $z=\frac{r}{2}$ gives $u=\frac{t}{2}+\frac{\sqrt{r^{2}+t^{2}}}{2}$,
and so%
\begin{eqnarray*}
\sup_{\func{restricted}}\frac{f(Q)}{\omega \left( f\left( P\right) \right) }
&\lesssim &\sup_{\rho r<t<r\text{ and }u\geq z}\frac{\varphi \left( u\right)
u^{4}}{\omega \left( \varphi \left( t\right) t^{4}+\varphi (r)h\left( \frac{t%
}{r}\right) \right) } \\
&\lesssim &\sup_{\rho r<t<r\text{ and }u\geq z}\frac{\varphi \left( u\right)
u^{4}}{\omega \left( \varphi \left( t\right) t^{4}\right) }\lesssim
\sup_{\rho r<t<r}\frac{\varphi \left( \frac{t}{2}+\frac{\sqrt{r^{2}+t^{2}}}{2%
}\right) t^{4}}{\omega \left( \varphi \left( t\right) t^{4}\right) } \\
&\approx &\sup_{0<t<1}\frac{\varphi \left( \left( \frac{1}{2}+\frac{\sqrt{%
\frac{1}{\rho ^{2}}+1}}{2}\right) t\right) t^{4}}{\omega \left( \varphi
\left( t\right) t^{4}\right) }=\mathcal{T}_{\varphi }^{\omega }\left( \frac{1%
}{2}+\frac{\sqrt{\frac{1}{\rho ^{2}}+1}}{2}\right) \leq \mathcal{T}_{\varphi
}^{\omega }\left( \gamma _{\rho }\right) .
\end{eqnarray*}
\end{proof}

Combining the estimates for $\mathcal{M}_{1}f\left( \omega \right) $ and $%
\mathcal{M}_{2}f\left( \omega \right) $ completes the proof of Part (1) of
Theorem \ref{1/4 and epsilon_0}.

\subsubsection{Proof of part (2)(a)}

Denote by $\mathcal{C}_{2,\omega }^{\nu }$ the function. 
\begin{equation*}
\mathcal{C}_{2,\omega }^{\nu }\left( \tau \right) \equiv \inf \left\{
\left\Vert \mathbf{G}\right\Vert _{2,\omega }:\mathbf{G}=\left\{ G_{\ell
}\right\} _{\ell =1}^{\nu }\in \oplus ^{\nu }C^{2,\omega }\left( B_{\mathbb{R%
}^{4}}\left( 0,1\right) \right) \text{ and }L\left( W\right) +\tau
=\sum_{\ell =1}^{\nu }G_{\ell }\left( W\right) ^{2}\right\} .
\end{equation*}%
Note that by Lemma \ref{g control}, we have $\lim_{\tau \rightarrow 0}%
\mathcal{C}_{2,\omega }^{\nu }\left( \tau \right) =\infty $, but we will
require the sharper inequality given in (\ref{crucial low}) below. Suppose
that 
\begin{eqnarray*}
L\left( W\right) +\tau &=&\sum_{\ell =1}^{\nu }G_{\ell }\left( W\right) ^{2},
\\
G_{\ell }\left( W\right) &=&a_{\ell }+S_{\ell }\left( W\right) +Q_{\ell
}\left( W\right) +R_{\ell }\left( W\right) ,
\end{eqnarray*}%
where%
\begin{equation*}
S_{\ell }\left( W\right) =\sum_{\left\vert \alpha \right\vert =1}a_{\ell
,\alpha }W^{\alpha }\text{ and }Q_{\ell }\left( W\right) =\sum_{\left\vert
\alpha \right\vert =2}f_{\ell ,\alpha }W^{\alpha }.
\end{equation*}%
Then setting $W=0$ in the equation gives 
\begin{equation*}
\tau =\sum_{\ell =1}^{\nu }a_{\ell }^{2},
\end{equation*}%
and so%
\begin{eqnarray*}
L\left( W\right) &=&\sum_{\ell =1}^{\nu }\left[ a_{\ell }+S_{\ell }\left(
W\right) +Q_{\ell }\left( W\right) +R_{\ell }\left( W\right) \right]
^{2}-\tau \\
&=&\left( \sum_{\ell =1}^{\nu }a_{\ell }^{2}\right) -\tau +\sum_{\ell
=1}^{\nu }2a_{\ell }S_{\ell }\left( W\right) +\sum_{\ell =1}^{\nu }\left[
S_{\ell }\left( W\right) ^{2}+2a_{\ell }Q_{\ell }\left( W\right) \right] \\
&&+\sum_{\ell =1}^{\nu }2a_{\ell }R_{\ell }\left( W\right) +\sum_{\ell
=1}^{\nu }2S_{\ell }\left( W\right) R_{\ell }\left( W\right) +\sum_{\ell
=1}^{\nu }\left[ Q_{\ell }\left( W\right) +R_{\ell }\left( W\right) \right]
^{2}.
\end{eqnarray*}%
Now the sum of terms in the middle line vanishes identically since it is a
quadratic polynomial, and all of the remaining terms in the identity vanish
to order greater than $2$ at the origin (simply evaluate the identity at $%
W=0 $, then differentiate and evaluate at $W=0$, and finally differentiate
once more and evaluate at $W=0$, using that $R_{\ell }\left( 0\right) $, $%
\nabla R_{\ell }\left( 0\right) $ and $\nabla ^{2}R_{\ell }\left( 0\right)
=0 $ all vanish). Thus we conclude that%
\begin{equation}
L\left( W\right) -\sum_{\ell =1}^{\nu }\left[ Q_{\ell }\left( W\right)
+R_{\ell }\left( W\right) \right] ^{2}=\sum_{\ell =1}^{\nu }2a_{\ell
}R_{\ell }\left( W\right) +\sum_{\ell =1}^{\nu }2S_{\ell }\left( W\right)
R_{\ell }\left( W\right) .  \label{conclude}
\end{equation}


Now define $\delta _{\nu }>0$ by%
\begin{equation}
\delta _{\nu }^{2}\equiv \inf_{\left\{ Q_{\ell }\right\} _{\ell =1}^{\nu
}}\inf_{W\in \mathbb{S}^{3}}\left( L\left( W\right) -\sum_{\ell =1}^{\nu
}Q_{\ell }\left( W\right) ^{2}\right) ^{2},  \label{def delta}
\end{equation}%
where the infimum is taken over all collections $\left\{ Q_{\ell }\right\}
_{\ell =1}^{\nu }$ of quadratic forms $Q_{\ell }\left( W\right)
=\sum_{\left\vert \alpha \right\vert =2}f_{\ell ,\alpha }W^{\alpha }$, with $%
W\in \mathbb{S}^{4}$ and coefficients $f_{\ell ,\alpha }$ of modulus at most
a constant $C_{0}$, which will be determined in (\ref{R bound}) below. Since
the infimum is taken over a compact set, it is achieved, and must then be
positive since $L$ cannot be written as a sum of squares of quadratic forms.

Now fix a modulus of continuity $\omega $, and given $\tau >0$, suppose
there are functions $G_{\ell }\in C^{2,\omega }$ with $\sum_{\ell =1}^{\nu
}\left\Vert G_{\ell }\right\Vert _{C^{2,\omega }}=\left\Vert \mathbf{G}%
\right\Vert _{2,\omega }<\infty $ such that%
\begin{equation}
L\left( W\right) +\tau =\sum_{\ell =1}^{\nu }G_{\ell }\left( W\right) ^{2},\
\ \ \ \ \text{for all }\left\vert W\right\vert \leq 1.  \label{sos tau}
\end{equation}%
Recall that we can write 
\begin{equation*}
G_{\ell }\left( W\right) =a_{\ell }+S_{\ell }\left( W\right) +Q_{\ell
}\left( W\right) +R_{\ell }\left( W\right) ,
\end{equation*}%
where 
\begin{align}
\sum_{\ell =1}^{\nu }a_{\ell }^{2}& =\tau ,  \label{terms_est} \\
\sum_{\ell =1}^{\nu }|S_{\ell }\left( W\right) |& \leq \left\Vert \mathbf{G}%
\right\Vert _{2,\omega }|W| \\
\sum_{\ell =1}^{\nu }|Q_{\ell }\left( W\right) |& \leq \left\Vert \mathbf{G}%
\right\Vert _{2,\omega }|W|^{2} \\
\sum_{\ell =1}^{\nu }|R_{\ell }\left( W\right) |& \leq \left\Vert \mathbf{G}%
\right\Vert _{2,\omega }|W|^{2}\omega (W).
\end{align}%
Also note that from $\sum_{\ell =1}^{\nu }\left\vert Q_{\ell }\left(
W\right) \right\vert \leq C\sqrt{L\left( W\right) +\tau }$, we obtain that
for $0<\tau <1$, we have 
\begin{equation}
\left\vert f_{\ell ,\alpha }\right\vert \leq C_{0}\equiv C\sqrt{L\left(
W\right) +1}.  \label{R bound}
\end{equation}

From (\ref{conclude}) we have%
\begin{eqnarray}
L\left( W\right) -\sum_{\ell =1}^{\nu }Q_{\ell }\left( W\right) ^{2}
&=&L\left( W\right) -\sum_{\ell =1}^{\nu }\left[ Q_{\ell }\left( W\right)
+R_{\ell }\left( W\right) \right] ^{2}+\sum_{\ell =1}^{\nu }\left[ 2Q_{\ell
}\left( W\right) +R_{\ell }\left( W\right) \right] R_{\ell }\left( W\right)
\label{conclude'} \\
&=&h_{1}\left( W\right) +h_{2}\left( W\right) \equiv h\left( W\right) , 
\notag
\end{eqnarray}%
where 
\begin{align}
h_{1}\left( W\right) & \equiv \sum_{\ell =1}^{\nu }2a_{\ell }R_{\ell }\left(
W\right) +\sum_{\ell =1}^{\nu }2S_{\ell }\left( W\right) R_{\ell }\left(
W\right)  \label{h-def} \\
h_{2}\left( W\right) & \equiv \sum_{\ell =1}^{\nu }\left[ 2Q_{\ell }\left(
W\right) +R_{\ell }\left( W\right) \right] R_{\ell }\left( W\right) .
\end{align}%
Using the last line of (\ref{terms_est}) we obtain 
\begin{align*}
\left\vert h_{1}\left( W\right) \right\vert & \leq C\sqrt{\tau }\left\Vert 
\mathbf{G}\right\Vert _{2,\omega }\left\vert W\right\vert ^{2}\omega \left(
\left\vert W\right\vert \right) +C\left\Vert \mathbf{G}\right\Vert
_{2,\omega }^{2}\left\vert W\right\vert ^{3}\omega \left( \left\vert
W\right\vert \right) =C\left\Vert \mathbf{G}\right\Vert _{2,\omega
}^{2}\left\vert W\right\vert ^{2}\omega \left( \left\vert W\right\vert
\right) \left( \frac{\sqrt{\tau }}{\left\Vert \mathbf{G}\right\Vert
_{2,\omega }}+|W|\right) \\
\left\vert h_{2}\left( W\right) \right\vert & \leq C\left\Vert \mathbf{G}%
\right\Vert _{2,\omega }\left\vert W\right\vert ^{2}\left\vert R_{\ell
}\left( W\right) \right\vert \leq C\left\Vert \mathbf{G}\right\Vert
_{2,\omega }^{2}\left\vert W\right\vert ^{4}\omega \left( \left\vert
W\right\vert \right) .
\end{align*}%
So altogether we have%
\begin{align*}
\left\vert h\left( W\right) \right\vert \leq \left\vert h_{1}\left( W\right)
\right\vert +\left\vert h_{2}\left( W\right) \right\vert & \leq C\left\Vert 
\mathbf{G}\right\Vert _{2,\omega }^{2}\omega \left( \left\vert W\right\vert
\right) \left\vert W\right\vert ^{2}\left( \frac{\sqrt{\tau }}{\left\Vert 
\mathbf{G}\right\Vert _{2,\omega }}+|W|+\left\vert W\right\vert ^{2}\right)
\\
& \leq C\left\Vert \mathbf{G}\right\Vert _{2,\omega }^{2}\omega \left(
\left\vert W\right\vert \right) \left\vert W\right\vert ^{2}\left( \frac{%
\sqrt{\tau }}{\left\Vert \mathbf{G}\right\Vert _{2,\omega }}+|W|\right) ,
\end{align*}%
provided $|W|\leq 1$. Note that we can assume without loss of generality
that $\frac{\sqrt{\tau }}{\left\Vert \mathbf{G}\right\Vert _{2,\omega }}\leq
1$. Then if $|W|=\frac{\sqrt{\tau }}{\left\Vert \mathbf{G}\right\Vert
_{2,\omega }}$, we have 
\begin{equation*}
\left\vert h\left( W\right) \right\vert \leq C\left\Vert \mathbf{G}%
\right\Vert _{2,\omega }^{2}\omega \left( \left\vert W\right\vert \right)
\left\vert W\right\vert ^{3}.
\end{equation*}%
However, this estimate is too weak, and we need an improved bound on $%
|S_{\ell }(W)|$.

We return to (\ref{sos tau}) to obtain%
\begin{eqnarray*}
&&L\left( W\right) +\lambda ^{-4}\tau =\lambda ^{-4}\left( L\left( \lambda
W\right) +\tau \right) \\
&=&\sum_{\ell =1}^{\nu }\left[ \frac{S_{\ell }\left( W\right) }{\lambda }+%
\frac{a_{\ell }}{\lambda ^{2}}+Q_{\ell }\left( W\right) +\lambda
^{-2}R_{\ell }\left( \lambda W\right) \right] ^{2} \\
&=&\sum_{\ell =1}^{\nu }\left[ \frac{S_{\ell }\left( W\right) }{\lambda }%
\right] ^{2}+O\left( \sqrt{\sum_{\ell =1}^{\nu }\left[ \frac{S_{\ell }\left(
W\right) }{\lambda }\right] ^{2}}\left[ \frac{\sqrt{\tau }}{\lambda ^{2}}%
+\left\Vert \mathbf{G}\right\Vert _{2,\omega }\left\vert W\right\vert
^{2}+\left\Vert \mathbf{G}\right\Vert _{2,\omega }\left\vert W\right\vert
^{2}\omega \left( \lambda \left\vert W\right\vert \right) \right] \right) \\
&&+O\left( \left[ \frac{\sqrt{\tau }}{\lambda ^{2}}+\left\Vert \mathbf{G}%
\right\Vert _{2,\omega }\left\vert W\right\vert ^{2}+\left\Vert \mathbf{G}%
\right\Vert _{2,\omega }\left\vert W\right\vert ^{2}\omega \left( \lambda
\left\vert W\right\vert \right) \right] ^{2}\right) ,
\end{eqnarray*}%
and hence%
\begin{eqnarray*}
\sum_{\ell =1}^{\nu }\left[ \frac{S_{\ell }\left( W\right) }{\lambda }\right]
^{2} &\leq &C\left\vert W\right\vert ^{4}+\lambda ^{-4}\tau +C\left[ \frac{%
\sqrt{\tau }}{\lambda ^{2}}+\left\Vert \mathbf{G}\right\Vert _{2,\omega
}\left\vert W\right\vert ^{2}+\left\Vert \mathbf{G}\right\Vert _{2,\omega
}\left\vert W\right\vert ^{2}\omega \left( \lambda \left\vert W\right\vert
\right) \right] ^{2}; \\
\text{i.e. }\sum_{\ell =1}^{\nu }\left[ \frac{S_{\ell }\left( W\right) }{%
\left\vert W\right\vert \lambda }\right] ^{2} &\leq &C\left\vert
W\right\vert ^{2}+\frac{\tau }{\left\vert W\right\vert ^{2}\lambda ^{4}}+C%
\left[ \frac{\sqrt{\tau }}{\left\vert W\right\vert \lambda ^{2}}+\left\Vert 
\mathbf{G}\right\Vert _{2,\omega }\left\vert W\right\vert +\left\Vert 
\mathbf{G}\right\Vert _{2,\omega }\left\vert W\right\vert \omega \left(
\lambda \left\vert W\right\vert \right) \right] ^{2} \\
&\leq &C\frac{\tau }{\left\vert W\right\vert ^{2}\lambda ^{4}}+C\left\Vert 
\mathbf{G}\right\Vert _{2,\omega }^{2}\left\vert W\right\vert ^{2},\ \ \ \ \ 
\text{provided }\lambda \left\vert W\right\vert \text{ remains bounded}.
\end{eqnarray*}%
But now we note that 
\begin{equation*}
\left\Vert \frac{S_{\ell }\left( W\right) }{\left\vert W\right\vert }%
\right\Vert _{\infty }=\left\Vert \sum_{\left\vert \alpha \right\vert
=1}a_{\ell ,\alpha }\left( \frac{W}{\left\vert W\right\vert }\right)
^{\alpha }\right\Vert _{\infty }\approx \sum_{\left\vert \alpha \right\vert
=1}\left\vert a_{\ell ,\alpha }\right\vert ,
\end{equation*}%
and so we conclude that%
\begin{equation*}
\sum_{\left\vert \alpha \right\vert =1}\left\vert a_{\ell ,\alpha
}\right\vert \leq C\left( \frac{\tau }{\left\vert W\right\vert ^{2}\lambda
^{2}}+\left\Vert \mathbf{G}\right\Vert _{2,\omega }^{2}\left\vert
W\right\vert ^{2}\lambda ^{2}\right) ^{\frac{1}{2}}\leq 2C\left\Vert \mathbf{%
G}\right\Vert _{2,\omega }^{\frac{1}{2}}\tau ^{\frac{1}{4}}
\end{equation*}%
if we choose $\lambda =\frac{\sqrt[4]{\tau }}{\sqrt{\mathfrak{N}}\left\vert
W\right\vert }$, and thus 
\begin{equation*}
\left\Vert S_{\ell }\left( W\right) \right\vert \leq C\left\Vert \mathbf{G}%
\right\Vert _{2,\omega }^{\frac{1}{2}}\tau ^{\frac{1}{4}}|W|.
\end{equation*}

Using this together with (\ref{terms_est}) in (\ref{h-def}) we obtain 
\begin{equation*}
\left\vert h\left( W\right) \right\vert \leq \left\vert h_{1}\left( W\right)
\right\vert +\left\vert h_{2}\left( W\right) \right\vert \leq C\left\Vert 
\mathbf{G}\right\Vert _{2,\omega }^{2}\omega \left( \left\vert W\right\vert
\right) \left\vert W\right\vert ^{2}\left( \frac{\sqrt{\tau }}{\left\Vert 
\mathbf{G}\right\Vert _{2,\omega }}+\frac{\sqrt[4]{\tau }}{\sqrt{\left\Vert 
\mathbf{G}\right\Vert _{2,\omega }}}|W|+\left\vert W\right\vert ^{2}\right) .
\end{equation*}%
If $|W|=\frac{\sqrt[4]{\tau }}{\sqrt{\mathfrak{N}}}$ we have 
\begin{equation*}
\left\vert h\left( W\right) \right\vert \leq C\left\Vert \mathbf{G}%
\right\Vert _{2,\omega }^{2}\omega \left( \left\vert W\right\vert \right)
\left\vert W\right\vert ^{4},
\end{equation*}%
and from (\ref{conclude'}) we obtain 
\begin{equation*}
L\left( \frac{W}{\left\vert W\right\vert }\right) -\sum_{\ell =1}^{\nu
}Q_{\ell }\left( \frac{W}{\left\vert W\right\vert }\right) ^{2}=\left\vert 
\frac{L\left( W\right) -\sum_{\ell =1}^{\nu }Q_{\ell }\left( W\right) ^{2}}{%
\left\vert W\right\vert ^{4}}\right\vert \leq \frac{\left\vert h\left(
W\right) \right\vert }{\left\vert W\right\vert ^{4}}\leq C\left\Vert \mathbf{%
G}\right\Vert _{2,\omega }^{2}\omega \left( \left\vert W\right\vert \right) ,
\end{equation*}%
if $|W|=\frac{\sqrt[4]{\tau }}{\sqrt{\mathfrak{N}}}$. Using (\ref{def delta}%
) and (\ref{R bound}) we thus have the following estimate 
\begin{equation*}
\delta _{\nu }\leq C\left\Vert \mathbf{G}\right\Vert _{2,\omega }^{2}\omega
\left( \frac{\sqrt[4]{\tau }}{\sqrt{\left\Vert \mathbf{G}\right\Vert
_{2,\omega }}}\right) ,\ \ \ \ \ \text{for }C_{0}\geq C\left\Vert \mathbf{G}%
\right\Vert _{2,\omega }\ ,
\end{equation*}%
where $C_{0}$ is the constant defined in (\ref{R bound}). In the special
case $\omega (r)=r^{\beta }$ we have 
\begin{equation*}
\delta _{\nu }\leq C\left\Vert \mathbf{G}\right\Vert _{2,\omega }^{2-\frac{%
\beta }{2}}\tau ^{\frac{\beta }{4}},
\end{equation*}%
or equivalently 
\begin{equation*}
\left\Vert \mathbf{G}\right\Vert _{2,\omega }\geq \left( \frac{\delta _{\nu }%
}{C}\right) ^{\frac{2}{4-\beta }}\left( \frac{1}{\tau }\right) ^{\frac{\beta 
}{8-2\beta }},\ \ \ \ \ \text{provided }\left\Vert \mathbf{G}\right\Vert
_{2,\omega }\leq \frac{C_{0}}{C}.
\end{equation*}%
Altogether we have obtained thus far the crucial lower bound%
\begin{equation}
\mathcal{C}_{2,\omega _{\beta }}^{\nu }\left( \tau \right) \geq \left( \frac{%
\delta _{\nu }}{C}\right) ^{\frac{2}{4-\beta }}\tau ^{-\frac{\beta }{%
8-2\beta }}.  \label{crucial low}
\end{equation}%
The next lemma finishes the proof of part 2(a) of Theorem \ref{1/4 and
epsilon_0}.

\begin{lemma}
\label{failure}Suppose $0<\beta <1$ and let $f_{\varphi ,\psi }\left(
W,t\right) $ be as in (\ref{def f}). If 
\begin{equation}
\limsup_{t\rightarrow 0}\frac{\psi \left( t\right) }{\varphi \left( t\right)
^{\frac{4}{\beta }}t^{\frac{16}{\beta }}}=0,  \label{contra}
\end{equation}%
then $f_{\varphi ,\psi }$ fails to satisfy $\mathcal{SOS}_{2,\omega _{\beta
}}^{\nu }$ for any $\nu \in \mathbb{N}$. Note in particular we may even take
both $\varphi $ and $\psi $ to be nearly monotone functions on $\left(
-1,1\right) $.
\end{lemma}

\begin{proof}
Assume, in order to derive a contradiction, that $f_{\varphi ,\psi }\left(
W,t\right) $ has the property $\mathcal{SOS}_{2,\omega _{\beta }}^{\nu }$
for some $\nu \in \mathbb{N}$, i.e. $f_{\varphi ,\psi }=\sum_{\ell =1}^{\nu
}G_{\ell }^{2}$ where $G_{\ell }\in C^{2,\omega }\left( \Omega \right) $,
i.e.%
\begin{eqnarray*}
&&\varphi \left( t\right) L\left( w,x,y,z,t\right) +\left[ \psi \left(
t\right) +\varphi \left( r\right) h\left( \frac{t}{r}\right) \right]
=\sum_{\ell =1}^{\nu }G_{\ell }\left( x,y,z,t\right) ^{2}, \\
&&\text{for }\left( x,y,z,t\right) \in \Omega =B_{\mathbb{R}^{3}}\left(
0,1\right) \times \left( -1,1\right) .
\end{eqnarray*}%
Then since $h\left( \frac{t}{r}\right) $ vanishes for $r\leq \left\vert
t\right\vert $, we have with $W\equiv \left( w,x,y,z\right) $, and without
loss of generality $t>0$, that%
\begin{equation*}
\varphi \left( t\right) L\left( W\right) +\psi \left( t\right) =\sum_{\ell
=1}^{\nu }G_{\ell }\left( W,t\right) ^{2},\ \ \ \ \ \text{for }r\leq t,
\end{equation*}%
and replacing $W$ by $tW$ we have, 
\begin{eqnarray*}
\varphi \left( t\right) L\left( tW\right) +\psi \left( t\right)
&=&\sum_{\ell =1}^{\nu }G_{\ell }\left( tW,t\right) ^{2}, \\
\text{for }\left\vert W\right\vert &\leq &1,t\in \left( 0,1\right) .
\end{eqnarray*}%
Multiplying by $\frac{1}{\varphi \left( t\right) t^{2}}$, and using that $L$
is homogeneous of degree four, we obtain%
\begin{eqnarray*}
\mathbf{L}\left( W\right) +\frac{\psi \left( t\right) }{\varphi \left(
t\right) t^{4}} &=&\sum_{\ell =1}^{\nu }\left( \frac{G_{\ell }\left(
tW,t\right) }{\sqrt{\varphi \left( t\right) }t^{2}}\right) ^{2}, \\
\text{for }\left\vert W\right\vert &\leq &1,t\in \left( 0,1\right) .
\end{eqnarray*}

Since $G_{\ell }\in C^{2,\omega }\left( B_{\mathbb{R}^{4}}\left( 0,1\right)
\times \left( -1,1\right) \right) $, the functions $W\rightarrow G_{\ell
}\left( W,t\right) $ lie in a bounded set in $C^{2,\omega }\left( B_{\mathbb{%
R}^{4}}\left( 0,1\right) \right) $ independent of $t$ and $j$, and hence
also the collection of functions%
\begin{equation*}
H_{\ell }^{t}\left( W\right) \equiv G_{\ell }\left( tW,t\right) ,\ \ \ \ \
1\leq \ell \leq \nu ,t\in \left( 0,1\right) ,
\end{equation*}%
is bounded in $C^{2,\omega }\left( B_{\mathbb{R}^{4}}\left( 0,1\right)
\right) $, say 
\begin{equation}
\sum_{\ell =1}^{\nu }\left\Vert H_{\ell }^{t}\right\Vert _{C^{2,\omega
}\left( B_{\mathbb{R}^{4}}\left( 0,1\right) \right) }\leq \mathfrak{N}_{\nu
},\ \ \ \ \ t\in \left( 0,1\right) .  \label{say bound}
\end{equation}%
Thus with $\tau =\tau \left( t\right) \equiv \frac{\psi \left( t\right) }{%
\varphi \left( t\right) t^{4}}$, we have from (\ref{say bound}) and (\ref%
{crucial low}) that%
\begin{eqnarray*}
&&\frac{\mathfrak{N}_{\nu }}{\sqrt{\varphi \left( t\right) t^{4}}}\geq
\sum_{\ell =1}^{\nu }\left\Vert \frac{H_{\ell }^{t}}{\sqrt{\varphi \left(
t\right) }t^{2}}\right\Vert _{C^{2,\omega }\left( B_{\mathbb{R}^{4}}\left(
0,1\right) \right) }\geq \mathcal{C}_{2,\omega }^{\nu }\left( \tau \left(
t\right) \right) \\
&\geq &\left( \frac{\delta _{\nu }}{C}\right) ^{\frac{2}{4-\beta }}\tau
\left( t\right) ^{-\frac{\beta }{8-2\beta }}=\left( \frac{\delta _{\nu }}{C}%
\right) ^{\frac{2}{4-\beta }}\left( \frac{\psi \left( t\right) }{\varphi
\left( t\right) t^{4}}\right) ^{-\frac{\beta }{8-2\beta }},
\end{eqnarray*}%
and hence%
\begin{equation*}
\left( \frac{\delta _{\nu }}{C}\right) ^{\frac{2}{4-\beta }}\leq
\liminf_{t\rightarrow 0}\frac{\mathfrak{N}_{\nu }}{\sqrt{\varphi \left(
t\right) t^{4}}}\left( \frac{\psi \left( t\right) }{\varphi \left( t\right)
t^{4}}\right) ^{\frac{\beta }{8-2\beta }}=\mathfrak{N}_{\nu
}\liminf_{t\rightarrow 0}\left( \frac{\psi \left( t\right) }{\varphi \left(
t\right) ^{\frac{4}{\beta }}t^{\frac{16}{\beta }}}\right) ^{\frac{\beta }{%
8-2\beta }},
\end{equation*}%
contradicting (\ref{contra}) as required. This completes the proof of Lemma %
\ref{failure}.
\end{proof}

\subsubsection{Proof of part (2)(b)}

Choose $s<s^{\prime }<\beta $. If we set $\psi \left( t\right) =\varphi
\left( \frac{t}{2}\right) ^{\frac{1}{s^{\prime }}}t^{\frac{4}{s^{\prime }}}$%
, then $f$ is $\omega _{s}$-monotone by part (1), and we have%
\begin{equation*}
\lim_{t\searrow 0}\frac{\varphi \left( t\right) ^{\frac{4}{\beta }}t^{\frac{%
16}{\beta }}}{\varphi \left( \frac{t}{2}\right) ^{\frac{1}{s^{\prime }}}t^{%
\frac{4}{s^{\prime }}}}=\lim_{t\searrow 0}\frac{e^{-\frac{1}{t^{2}}\frac{4}{%
\beta }}t^{\frac{16}{\beta }}}{e^{-\frac{1}{t^{2}}\frac{4}{s^{\prime }}}t^{%
\frac{4}{s^{\prime }}}}=\lim_{t\searrow 0}e^{\frac{4}{t^{2}}\left( \frac{1}{%
s^{\prime }}-\frac{1}{\beta }\right) }t^{4\left( \frac{4}{\beta }-\frac{1}{%
s^{\prime }}\right) }=\infty
\end{equation*}%
since $\beta >s^{\prime }$, and hence by part (2)(b), we cannot write $f$ as
a finite sum of squares of $C^{2,\beta }$ functions.

This completes the proof of part (2), and hence that of Theorem \ref{1/4 and
epsilon_0}.

\subsection{Extension to general moduli of continuity and proof of Theorem 
\protect\ref{log counter'}}

We first note that part (1) of Theorem \ref{log counter'} is implied by
Theorem \ref{1/4 and epsilon_0}.\ To prove part (2) let 
\begin{equation*}
\psi \left( t\right) \equiv \omega ^{-1}\left( \varphi \left( t\right)
\right) ,
\end{equation*}%
so that%
\begin{equation*}
f_{\varphi ,\psi ,h_{\rho }}\left( W,t\right) \approx \varphi \left(
t\right) r^{4}+\psi \left( t\right) +\varphi \left( r\right) h\left( \frac{t%
}{r}\right) \lesssim \varphi \left( r\right) .
\end{equation*}%
Then since $\omega ^{-1}$ vanishes to infinite order at the origin, we have 
\begin{equation*}
\lim_{t\searrow 0}\frac{\varphi \left( t\right) ^{\frac{4}{\beta }}t^{\frac{%
16}{\beta }}}{\psi \left( t\right) }=\lim_{t\searrow 0}\frac{\varphi \left(
t\right) ^{\frac{4}{\beta }}t^{\frac{16}{\beta }}}{\omega ^{-1}\left(
\varphi \left( t\right) \right) }\geq c_{N}\lim_{t\searrow 0}\frac{\varphi
\left( t\right) ^{\frac{4}{\beta }}t^{\frac{16}{\beta }}}{\varphi \left(
t\right) ^{N}}=c_{N}\lim_{t\searrow 0}\frac{e^{-\frac{1}{t^{2}}\frac{4}{%
\beta }}t^{\frac{16}{\beta }}}{e^{-\frac{1}{t^{2}}N}}=\infty
\end{equation*}%
for $N>\frac{4}{\beta }$, Part (2)(a) of Theorem \ref{1/4 and epsilon_0}
shows that $f_{\varphi ,\psi _{\rho },h_{\rho }}$ cannot be written as a
finite sum of squares of $C^{2,\beta }$ functions. On the other hand using $%
\delta <1/2$ we have for $N>\max \{2,\left( \gamma _{\rho }+\delta \right)
^{2}\}$ 
\begin{eqnarray*}
\mathcal{R}_{\left( \varphi ,\psi \right) }^{\omega }\left( 1+\delta \right)
&\equiv &\sup_{0<t\leq 1}\frac{\psi \left( t\right) }{\varphi \left(
t\right) }\frac{\varphi \left( (1+\delta )t\right) }{\omega \left( \psi
\left( t\right) \right) }\lesssim \sup_{0<t\leq 1}\frac{\varphi \left(
t\right) ^{N}\varphi \left( (1+\delta )t\right) }{\varphi \left( t\right)
^{2}}=\sup_{0<t<1}\frac{e^{-\frac{N}{t^{2}}-\frac{1}{(1+\delta )^{2}t^{2}}}}{%
e^{-\frac{2}{t^{2}}}}\leq 1, \\
\mathcal{T}_{\varphi }^{\omega }\left( \gamma _{\rho }+\delta \right)
&\equiv &\sup_{0<t\leq 1}\frac{\varphi \left( \left( \gamma _{\rho }+\delta
\right) t\right) t^{4}}{\omega \left( \varphi \left( t\right) t^{4}\right) }%
\lesssim \sup_{0<t\leq 1}\frac{\varphi \left( \left( \gamma _{\rho }+\delta
\right) t\right) }{\varphi \left( t\right) ^{\frac{1}{N}}}=\sup_{0<t<1}\frac{%
e^{-\frac{1}{\left( \gamma _{\rho }+\delta \right) ^{2}t^{2}}}}{e^{-\frac{1}{%
Nt^{2}}}}\leq 1, \\
\mathcal{S}_{\left( \varphi ,\psi \right) }^{\omega }\left( \frac{1}{2}%
+\delta \right) &\equiv &\sup_{0<t\leq 1}\frac{\varphi \left( \frac{t}{2}%
+\delta t\right) t^{4}}{\omega \left( \psi \left( t\right) \right) }\lesssim
\sup_{0<t\leq 1}\frac{\varphi \left( \frac{t}{2}+\delta t\right) }{\varphi
\left( t\right) }=\sup_{0<t<1}\frac{e^{-\frac{1}{\left( (1/2+\delta
)t\right) ^{2}}}}{e^{-\frac{1}{t^{2}}}}\leq 1.
\end{eqnarray*}

Thus, Part (1) of Theorem \ref{1/4 and epsilon_0} shows that $f_{\varphi
,\psi,h_{\rho }}$ is $\omega $-monotone, which completes the proof of
Theorem \ref{log counter'}.

We end the paper by collecting the previous results into a somewhat sharp
theorem in all dimensions, which can be summed up as \emph{roughly saying}
that an elliptical flat smooth function can be written as a finite sum of
squares of regular functions `\emph{if and only}' if it is H\"{o}lder
monotone.

\begin{theorem}
\label{power weakly mon SOS}Suppose that $f$ is elliptical flat smooth and H%
\"{o}lder monotone, i.e. $\omega _{s}$-monotone on $\mathbb{R}^{n}$ for some 
$0<s<1$ and $n\geq 1$. Then there is $\delta >0$ such that $f$ is a finite
sum of squares of $C^{2,\delta }$ functions. Conversely, for every modulus
of continuity $\omega $ satisfying $\omega _{s}\ll \omega $ for all $0<s<1$,
there is an elliptical flat smooth $\omega $-monotone function $f$ on $%
\mathbb{R}^{5}$ that \emph{cannot} be written as a finite sum of squares of $%
C^{2,\delta }$ functions for any $0<\delta <1$.
\end{theorem}

\begin{proof}
The first assertion is a consequence of Theorem \ref{efs eps}, while the
converse assertion was proved in part (2) of Theorem \ref{log counter'}.
\end{proof}

\end{document}